\documentclass{amsart}
\usepackage[utf8]{inputenc}
\usepackage{amsfonts}
\usepackage{amssymb}
\usepackage{amsmath}
\usepackage{amsthm}
\usepackage{amscd}
\usepackage{amstext}
\usepackage{appendix}
\usepackage{bm}
\usepackage{caption}
\usepackage{comment}
\usepackage[dvipsname]{xcolor}
\usepackage{enumitem}
\usepackage{fancyhdr}
\usepackage{hyperref}
\usepackage{graphicx}
\usepackage{lipsum}
\usepackage{mathabx}
\usepackage{mathrsfs}
\usepackage{mathtools}
\usepackage{romannum}
\usepackage[section]{placeins}
\usepackage{stmaryrd}
\usepackage[switch, modulo]{lineno}
\usepackage{tikz}
\usepackage{xcolor}
\hypersetup{
    colorlinks,
    linkcolor={black},
    citecolor={black},
    urlcolor={black}
}
\theoremstyle{plain}
    \newtheorem{theorem}{Theorem}[section]
    \newtheorem{lem}[theorem]{Lemma}
    \newtheorem{prop}[theorem]{Proposition}
    \newtheorem{cor}[theorem]{Corollary}
    \newtheorem{rem}[theorem]{Remark}
\theoremstyle{definition}
    \newtheorem{defn}[theorem]{Definition}

\renewcommand{\thepage}{\roman{page}}
\title{Characterizations of reflexive Banach spaces}
\author{Tianyi Zhou}
\date{}
        
\begin{document}

\clearpage
\thispagestyle{empty}

\begin{abstract}
In this paper we survey known results of characterizations of reflexive Banach spaces, which are based on convergence of usual and generalized arithmetic mean (or Cesàro sum), weakly compact subsets, affine sets in a Banach space or its dual and an unbounded bi-orthogonal system generalized from the one in a finite-dimensional Banach space. We also include results that describe precisely when a subspace is linearly isomorphic to $\ell^1$ or $c_0$ in a Banach space that has a Schauder basis, which can imply non-reflexivity of a Banach space in general and is proven to be equivalent to non-reflexivity when the given Schauder basis is unconditional. After reflexivity, we will also study other geometric properties that are strictly stronger, implications among them and their characterizations.
\end{abstract}

\maketitle

\vspace{0.5cm}
\begin{center}
\large\bfseries{Introduction}
\end{center}
\vspace{0.5cm}

The characterization of a reflexive Banach space is a classic topic that can date back to around a hundred years ago. The basic ideas of such characterization includes fining the precise descriptions of the closed unit ball or an affine set, and the behaviors of a Schauder basis (if exists) or the dual space. Relevant results that follow ideas above will be included in this paper. Much of the progress that describes a reflexive Banach space by weakly compact sets was completed and summarized by R. James in \cite{7} and \cite{8}. Results that describe reflexivity by the dual space are mainly from R. James in \cite{9} and P. Vlastimil in \cite{15}. Since it is clear that a reflexive Banach space contains no non-reflexive subspaces (see \cite{3}), one can check if a Banach space is reflexive or not by checking the existence of a subspace that linearly isomorphic to a non-reflexive Banach space, such as $\ell^1$ or $c_0$, and can refer to \cite{6}, \cite{16} and \cite{18}. In the case when a Schauder basis exists, reflexivity can be proven when such a basis has certain properties that are related to the existence of $\ell^1$ and $c_0$ (see \cite{6} and \cite{16}).\\

Results on characterizing reflexivity can naturally be applied to the study of a super-reflexive Banach space, which is a Banach space on which no reflexive Banach spaces can be finitely (with respect to dimension) representable. As a result, the study of properties between reflexivity and super-reflexivity and properties that are weaker or stronger then emerged. The result of such study can be summarized by the following diagram from \cite{12}: 

\begin{figure}[!htb]
    \centering
    \includegraphics[scale = 0.7]{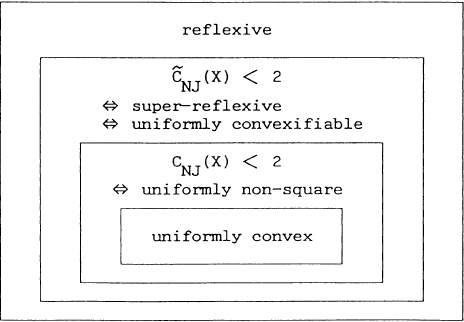}
    \caption{}\label{diagram}
    \label{}
\end{figure}
\noindent
where $C_{NJ}(X)$ is the \textbf{von Neumann-Jordan constant} of $(X, \|\cdot\|)$, which is first introduced in \cite{20} and plays an important role in characterizing super-reflexive Banach spaces; $\tilde{C}_{NJ}(X)$ is the infimum among all $C_{NJ}(X)$ when $X$ is equipped with a norm that is equivalent to $\|\cdot\|$; The definition of a \textbf{uniformly non-square} space can be found in \cite{26} and the one of a \textbf{uniformly convex} can be found in \cite{11}.\\

In this paper, we mean to give an overview of known results that arise from or are related to definitions in the diagram above. However, as the study of reflexivity is rather complete, there will definitely be topics we cannot cover due to space limitations. One can refer to \cite{28} for more implications between reflexivity and stronger properties, such as \textbf{rotundity} and \textbf{smooth norm}, and, for instance, \cite{27}, for another precise geometric description of reflexivity. Precise proofs of main results will be included. However, we do not mean to copy and paste the original proofs as several of them, from the author's perspective, lack details or explanations between statements, and those proofs will be recreated based on the original ideas.\\

This paper is organized as follows. The \textbf{Section 1} discusses how arithmetic means (of vectors) in a Banach spaces can be related to reflexivity and a few other stronger properties that will appear in later sections. \textbf{Section 2} deals with the equivalence between non-reflexivity and a pointwise vanishing sequence of continuous linear functionals. \textbf{Section 3} is devoted to the characterizations of reflexivity by affine sets in a Banach space or in its dual. In \textbf{Section 4} we find the relation among reflexivity, the existence of $\ell^1$ or $c_0$, and a Schauder basis with certain properties. In \textbf{Section 5} we generalized the bi-orthogonal system that always exists in a finite-dimensional normed linear space, to an infinitely-dimensional Banach space and show how such a system describes reflexivity. \textbf{Section 6} deals with the characterizations of reflexivity by weakly compact sets, and gives an summary that includes previous characterizations. The last section gives an overview of properties that are stronger than reflexivity and their characterizations.

\setcounter{page}{1}
\renewcommand{\thepage}{\arabic{page}}
\fancyhf[rh]{\arabic{page}}

\section{Arithmetic means in Banach spaces}

In this section, we start with the definition of uniform convexity, and will later see how it is stronger than all other geometric properties we will describe. We will discuss the relation between reflexivity, uniform convexity, the convergence of Cesàro sums (the usual arithmetic mean) of a bounded sequence and the convergence of a generalized arithmetic means of a bounded sequence, which is defined by an $\mathbb{N}$-by-$\mathbb{N}$ matrix, each of whose rows up to $1$.

\begin{defn}\label{Definition 1.1}

A normed linear space is \textbf{uniformly convex} if given any two vectors $x, y$ in the closed unit ball of the space, for any $\epsilon\in (0, 1)$ there exists a $\delta\in (0, 1)$ such that:

$$
\|x-y\|>\epsilon \hspace{0.3cm} \Longrightarrow \hspace{0.3cm} 1 < \|x+y\| < 2-\delta
$$

\end{defn}

\begin{defn}\label{Definition 1.2}

Given a normed linear space $X$ and a sequence $(x_n)_{n\in \mathbb{N}} \subseteq X$, a \textbf{summability method} $T$ is a real matrix $\big[c_{m, n} \big]_{m, n \in \mathbb{N}}$ such that $\{c_{m, n}\}_{n \in \mathbb{N}}$ is summable for each $m \in \mathbb{N}$. The $T$-\textbf{means} of the sequence is another sequence $\{t_m\}$ where each $t_m = \lim_{N \rightarrow \infty}\sum_{n \leq N}c_{m, n}x_n$. We say $T$ is \textbf{regular} if $x_n \rightarrow x$ implies $t_m \rightarrow x$. 

\end{defn}

\begin{theorem}[{\cite[Silverman-Toeplitz Theorem]{2}}]
In the set-up of \textbf{Definition \ref{Definition 1.2}}, the following statements are equivalent:

\begin{enumerate}[label = (\arabic*)]

    \item $T$ is regular.
    \item $\sup_m \sum_{n\in \mathbb{N}} \big\vert\, c_{m, n} \,\big\vert < \infty$.
    \item $\lim_m c_{m, n}=0$ for all $n\in \mathbb{N}$.
    \item $\lim_m \sum_{n\in \mathbb{N}} c_{m, n} = 1$.
    
\end{enumerate}

\end{theorem}

\begin{defn}

In the set-up of \textbf{Definition \ref{Definition 1.2}}, for a regular $T = \big[ c_{m, n} \big]_{m, n\in \mathbb{N}}$, we call $T$ essentially positive if:

$$
\lim_m \sum_{n\in \mathbb{N}} \big\vert\, c_{m, n} \,\big\vert = 1
$$
    
\end{defn}

\begin{defn}

Given a sequence $(x_n)_{n\in \mathbb{N}}$ in a Banach space $X$, for each $N\in\mathbb{N}$, define:

$$
s_N = \frac{1}{N}\sum_{i\leq N}x_i
$$
Then $X$ is said to have \textbf{Banach-Saks property} if whenever $(x_n)_{n\in \mathbb{N}}$ is weakly convergent, $(s_N)_{N\in \mathbb{N}}$ converges in norm to the same limit.
    
\end{defn}

\subsection{Uniformly convexity\texorpdfstring{$\,\implies\,$}{ implies }Banach-Saks property}

\begin{prop}[\cite{11}]\label{Proposition 1.3}
A normed linear space $X$ is \textbf{unifomrly convex} if and only if for any two vectors $x, y\in X$=, for any $\epsilon\in (0, 1)$ there exists $\delta\in (0, 1)$:
 
\begin{equation}\label{e1}
\|x-y\| > \epsilon\max\big( \|x\|, \|y\|\big) \hspace{0.3cm}\implies\hspace{0.3cm} \frac{1}{2}\|x+y\| \leq (1-\delta)\max\big( \|x\|, \|y\| \big)
\end{equation}

\end{prop}

\begin{proof}

By replacing $x, y$ in (\ref{e1}) by $\dfrac{x}{\max(\|x\|, \|y\|)}$ and $\dfrac{y}{\max(\|x\|, \|y\|)}$, the conclusion follows immediately by \textbf{Definition \ref{Definition 1.1}}.

\end{proof}

\begin{theorem}[\cite{11}]\label{Theorem 1.4}

A uniformly convex Banach space has the \textbf{Banach-Saks Property}.

\end{theorem}

\begin{proof}

Suppose $\|x_n\|\leq M\,\forall\,n \in \mathbb{N}$ and hence we can assume the sequence is in the closed unit ball. By definition, for any two elements $x, y$ in the closed unit ball, if $\|x-y\|>\dfrac{1}{2}$ then we can find $\delta_M$ such that $\dfrac{1}{2}\|x+y\| \leq 1 - \delta_M$. Then we need to find a subsequence $\{x_{k_i}\} \subseteq \{x_i\}$ such that $\dfrac{1}{2}\|x_{k_i}+x_{k_{i+1}}\| \leq \theta\,\forall\,i \in \mathbb{N}$ where $\theta = \max \left( \dfrac{3}{4}, 1-\delta_M \right) < 1$.

\begin{enumerate}[label = Case \alph*)]
    \item Fix $k_1 = 2$. If $\|x_2\| \leq \dfrac{1}{2}$ then set $k_2 = 3$ and then we have $\dfrac{1}{2} \|x_{k_1}+x_{k_2} \|\leq \dfrac{1}{2} \left( \dfrac{1}{2}+1 \right) \leq \theta$
    
    \item Fix $k_1 = 2$. If $\|x_2\| > \dfrac{1}{2}$. Suppose for all other integers $m, \|x_2-x_m\| \leq \dfrac{1}{2}$. Pick $f_n \in X^*$ such that $f_n(x_2) > \|x_2\| - \dfrac{1}{n}$. WLOG we can assume $x_n$ converges to zero weakly so that:
    
    $$
    \sup_i\|x_2-x_i\| \geq \lim_{i\rightarrow\infty}f_n(x_2 - x_i) = f_n(x_2) > \|x_2\|-\frac{1}{n}\,\implies\,\|x_2\|\leq\sup_i\|x_2-x_i\| \leq \frac{1}{2}
    $$
    which is absurd.
    
\end{enumerate}

\noindent
Hence when $\|x_2\| > \dfrac{1}{2}$ we can find $k_2$ such that $\|x_{k_1}-x_{k_2}\| > \dfrac{1}{2}$. Therefore, $\dfrac{1}{2} \|x_{k_1}+x_{k_2}\| \leq 1 - \delta_0 \leq \theta$.\\

\noindent    
Now after $k_2$ is determined, set $k_3 = k_2+1$ and set $k_4$ using the method above. Then set $k_5 = k_4+1$ and so on. After we have $\{x_{k_i}\}$ define $y_i^{(1)} = \dfrac{1}{2} \left( x_{k_{2i-1}} + x_{k_{2i}} \right)$. Then the sequence $\left\{ y_i^{(1)} \right\}_{i \in \mathbb{N}}$ also converges to zero and bounded by $\theta$ in norm. In the new sequence $\left\{ y_i^{(1)} \right\}$ find the subsequence $\left\{ y_{k_i}^{(1)} \right\}$ using the same method and then define $y_i^{(2)} = \dfrac{1}{2}\left( y_{k_{2i-1}}^{(1)} + y_{k_{2i}}^{(1)} \right)$ and obtain the new sequence $\left\{ y_i^{(2)} \right\}$. Notice this new sequence will be bounded by $\theta^2$. Therefore, by induction for any $k \in\mathbb{N}$ we will have a subsequence $\left\{ y_i^{(k)} \right\}$ of $\{x_i\}$ such that $\left\| y_i^{(k)} \right\| \leq \theta^k$ and $y_i^{(k)}$ converges to zero weakly. \\

\noindent    
Fix an integer $i$ and then consider $\left\{ y_i^{(n)} \right\}_{n \in \mathbb{N}}$. For $n > 0, y_i^{(n)} = \dfrac{1}{2} \left( y_{k_{2i-1}}^{(n-1)} + y_{k_{2i}}^{(n-1)} \right)$. Also each $y_{k_{2i-1}}^{(n-1)}$ and $y_{k_{2i}}^{(n-1)}$ is the mean of another two elements from $\left\{ y_i^{(n-2)} \right\}$. For a fixed $p \in \mathbb{N}$ we can find a finite subset $F_i^{(p)} = \left\{ l_{1,i}^{(p)}, l_{2,i}^{(p)}, \cdots, l_{2^p, i}^{(p)} \right\} \subseteq \mathbb{N}$ (ordered strictly increasing) such that $y_i^{(p)} = \dfrac{1}{2^p}\left( x_{l_{1,i}^{(p)}} + x_{l_{2,i}^{(p)}} + \,\cdots\, + x_{l_{2^p,i}^{(p)}} \right)$. Now define:

$$
\begin{aligned}
&n_1 = 1, \hspace{0.3cm} n_2 = l_{1,1}^{(1)}, \hspace{0.3cm} n_3 = l_{2,1}^{(1)},\hspace{0.3cm} n_4 = l_{1,1}^{(2)}\\
&n_5 = l_{2,1}^{(2)}, \hspace{0.3cm} n_6 = l_{3,1}^{(2)}, \hspace{0.3cm} n_7 = l_{4,1}^{(2)}, \hspace{0.3cm} n_8 = l_{1,1}^{(3)}\\
&\cdots\\
& n_{2^p-1} = l_{2^p, 1}^{(p-1)}, \hspace{0.3cm}n_{2^p} = l_{1,1}^{(p)}, \cdots
\end{aligned}
$$
For the new subsequence $\{x_{n_i}\}$ we want to show that:

$$
\left\| \frac{1}{k}(x_{n_1} + x_{n_2}+ \,\cdots\, x_{n_k}) \right\| \overset{k \rightarrow \infty}{\rightarrow} 0
$$
and it suffices to show that:

$$
\left\|\frac{1}{2^m-1}\left( x_{n_1}+x_{n_2}+\,\cdots\,x_{n_{2^m-1}} \right) \right\|\overset{m \longrightarrow \infty}{\rightarrow} 0
$$
Fix $q, m \in \mathbb{N}$ and $q < m$. Define $k = 2^m-1$.
$$
\begin{aligned}
&\hspace{0.47cm} \|x_{n_1}+x_{n_2}+\,\cdots\,x_{n_k}\|\\ 
&\leq \left\| x_{n_1}+x_{n_2}+\,\cdots\,x_{n_{2^q-1}} \right\| \\
& + \sum_{2 \leq j \leq 2} \left\| x_{n_{2^q(j-1)}}+x_{n_{2^q(j-1)+1}} + \,\cdots\,x_{(n_{j2^q})-1} \right\| \\
& + \sum_{3 \leq j \leq 4} \left\| x_{n_{2^q(j-1)}} + x_{n_{2^q(j-1)+1}} + \,\cdots\, x_{(n_{j2^q})-1} \right\| \\
& + \sum_{5 \leq j \leq 8} \left\| x_{n_{2^q(j-1)}}+x_{n_{2^q(j-1)+1}} + \,\cdots\,x_{(n_{j2^q})-1} \right\| \\
& + \cdots \\
& + \sum_{2^{m-1}+1 \leq j \leq 2^m} \left\| x_{n_{2^q(j-1)}} + x_{n_{2^q(j-1)+1}} + \,\cdots\, x_{n_k} \right\|
\end{aligned}
$$
Recall that for any $i\in\mathbb{N}$, $\|x_i\|\leq 1$. Then when $j=2$ we have:

$$
j = 2\,\implies\,x_{n_{(j-1)2^q}}+x_{n_{(j-1)2^q+1}}+\,\cdots\,x_{n_{j2^q-1}} = x_{l_1^{(q)}}+x_{l_2^{(q)}}+\,\cdots\,+x_{l_{2^q-1}^{(q)}} = 2^q y_1^{(q)}
$$
and similarly when $j=3$ or $4$:

$$
\begin{aligned}
&j = 3\,\implies\,x_{n_{(j-1)2^q}} + x_{n_{(j-1)2^q+1}}+\,\cdots\,x_{n_{j2^q-1}} = x_{n_{2^{q+1}}} + x_{n_{2^{q+1}+1}} + \,\cdots\,x_{n_{3\,\cdot\,2^q-1}} = \\
&\hspace{1.88cm}x_{l_{1,1}^{(q+1)}} + x_{l_{2,1}^{(q+1)}}\,\cdots\,x_{l_{2^q,1}^{(q+1)}}\\
&j = 4\,\implies\,x_{n_{(j-1)2^q}} + x_{n_{(j-1)2^q+1}}+\,\cdots\,x_{n_{j2^q-1}} = x_{n_{3\,\cdot\,2^q}}+x_{n_{3\,\cdot\,2^q+1}}+\,\cdots\,x_{n_{2^{q+2}-1}} = \\
&\hspace{1.88cm}x_{l_{2^q+1,1}^{(q+1)}}+x_{l_{2^q+2,1}^{(q+1)}}\,\cdots\,x_{l_{2^{q+1},1}^{(q+1)}}\\
\end{aligned}
$$
Therefore $y_1^{(q+1)}$ can be rewritten as the following form:

$$
y_1^{(q+1)} = \frac{1}{2^{q+1}}\left[ x_{l_{1,1}^{(q+1)}} + x_{l_{2,1}^{(q+1)}} \,\cdots\, x_{l_{2^q,1}^{(q+1)}} + x_{l_{2^q+1,1}^{(q+1)}} + x_{l_{2^q+2,1}^{(q+1)}} \,\cdots\, x_{l_{2^{q+1},1}^{(q+1)}} \right]
$$
Meanwhile, we know that:

$$
y_1^{(q+1)} = \frac{1}{2}\left( y_{k_1}^{(q)}+y_{k_2}^{(q)} \right) = \frac{1}{2^{q+1}} \left[ \sum_{j \in F_{k_1}^{(q)}}x_j + \sum_{i \in F_{k_2}^{(q)}}x_i \right]
$$
Since $F_{k_1}^{(q)}$ is unique to $y_{k_1}^{(q)}$, $F_{k_2}^{(q)}$ is unique to $y_{k_2}^{(q)}$, $F_{k_1}^{(q)}$ is disjoint from $F_{k_2}^{(q)}$ and the max number in $F_{k_1}^{(q)}$ will be smaller than the min number in $F_{k_2}^{(q)}$, then we will have $F_{k_1}^{(q)} \cup F_{k_2}^{(q)} = \left\{ l_{1,1}^{(q+1)}, l_{2,1}^{(q+1)}, \cdots\,l_{2^{q+1},1}^{(q+1)} \right\}$. Then:

$$
\sum_{3 \leq j \leq 4} \left\| x_{n_{2^q(j-1)}} + x_{n_{2^q(j-1)+1}} + \,\cdots\, x_{(n_{j2^q})-1} \right\| \leq 2^{q+1}\theta^{q+1}
$$
When $5 \leq j \leq 8$:

$$
\begin{aligned}
&\hspace{1cm} y_1^{(q+2)} = \frac{1}{2^{q+2}} \left[ x_{l_{1, 1}^{(q+2)}} + x_{l_{2, 1}^{(q+2)}} + \,\cdots\, + x_{l_{2^{q+2}, 1}^{(q+2)}} \right] = \frac{1}{2}\left[ y_{k_1}^{(q+1)} + y_{k_2}^{(q+1)} \right]\\ 
&\implies\, \left\| x_{l_{1, 1}^{(q+2)}} + x_{l_{2, 1}^{(q+2)}} + \,\cdots\, + x_{l_{2^{q+2}, 1}^{(q+2)}} \right\| \leq 2^{q+2}\theta^{q+2}
\end{aligned}
$$
By induction we have::

$$
\begin{aligned}
&\hspace{1.48cm} \|x_{n_1}+x_{n_2}+\,\cdots\,x_{n_k}\|\\
&\hspace{1cm} \leq \|x_{n_1}+x_{n_2}+\,\cdots\,x_{n_{2^m-1}}\| + \\
&\hspace{1.44cm} \sum_{2 \leq j \leq 2}\|x_{n_{2^q(j-1)}}+x_{n_{2^q(j-1)+1}}+\,\cdots\,x_{(n_{j2^q})-1}\| + \\
&\hspace{1.44cm} \sum_{3 \leq j \leq 4}\|x_{n_{2^q(j-1)}}+x_{n_{2^q(j-1)+1}}+\,\cdots\,x_{(n_{j2^q})-1}\| + \\
&\hspace{1.44cm} \sum_{5 \leq j \leq 8}\|x_{n_{2^q(j-1)}}+x_{n_{2^q(j-1)+1}}+\,\cdots\,x_{(n_{j2^q})-1}\| + \\
&\hspace{1.44cm} \cdots\\
&\hspace{1.44cm} \sum_{2^{m-1}+1 \leq j \leq 2^m}\|x_{n_{2^q(j-1)}}+x_{n_{2^q(j-1)+1}}+\,\cdots\,x_{n_k}\|  \\
&\hspace{1cm} \leq \sum_{q \leq i \leq m-1}2^i\theta^i\\
&\implies\, \frac{1}{k}\|x_{n_1}+x_{n_2}+\,\cdots\,x_{n_k}\| = \frac{1}{2^m-1}\|x_{n_1}+x_{n_2}+\,\cdots\,x_{n_{2^m-1}}\|\\
& \hspace{0.44cm} < \frac{2^q}{2^m-1} + \sum_{q \leq i \leq m-1}\theta^i = \frac{2^q}{2^m-1} + \theta^q\frac{1-\theta^{m-q}}{1-\theta}
\end{aligned}
$$
Now for any $\epsilon > 0$, find $q_0, m\in\mathbb{N}$ with $(q_0 < m)$ such that $\dfrac{\theta^{q_0} (1-\theta^{m-q_0})}{1-\theta} < \epsilon$. Then find $m_0 > m$ such that $\dfrac{2^{q_0}}{2^{m_0}-1} < \epsilon$. Hence now we have:

$$
\frac{1}{2^{m_0}-1} \left\| \sum_{1\leq i \leq 2^{m_0}-1}x_{n_i} \right\| < 2\epsilon
$$
and then we can conclude that:

$$
\lim_k \left\| \frac{1}{k} \sum_{1\leq i \leq k} x_{n_i} \right\| = 0
$$

\end{proof}

\subsection{Banach-Saks property \texorpdfstring{$\,\Longrightarrow\,$}{ implies } reflexivity}

\begin{defn}
A Banach space is said to have the \textbf{property} $\mathcal{P} (w\mathcal{P}\, resp.)$ if for every bounded sequence $\{x_n\}$ there is a regular summability method $T$ such that there exists a subsequence of $\{x_n\}$ whose $T$-means converge strongly (weakly $resp.$).
\end{defn}

\begin{theorem}[\cite{14}]\label{Theorem 1.6}

In a Banach space, given a infinite real matrix $T = 
[c_{m, n}]_{m, n\in\mathbb{N}}$ such that each line of sequence is summable, the following statements are equivalent:

\begin{enumerate}[label = \alph*)]

    \item $B$ is reflexive
    \item $B$ has property $\mathcal{P}$ and, for each bounded sequence $\{x_n\}$, the regular summability method $T$ that is given by the property $\mathcal{P}$ and associated with $\{x_n\}$ is essentially positive.
    \item $B$ has property $w\mathcal{P}$ and, for each bounded sequence $\{x_n\}$, the regular summability method $T$ that is given by the property $\mathcal{P}$ and associated with $\{x_n\}$ is essentially positive.
    
\end{enumerate}

\end{theorem}

\begin{proof}$\hspace{0.44cm}\\$
\begin{enumerate}[label = \arabic*)]

    \item $a)\implies b)$: WLOG assume the bounded sequence $\{x_n\}$ is inside the closed unit ball. By reflexivity there exists $x$ with $\|x\|\leq 1$ such that $x$ is a clustered point of $\{x_n\}$ and hence also a clustered point of $\operatorname{conv}\{x_n\}$. Therefore, there exists a sequence $\{s_m\}\subseteq \operatorname{conv}\{x_n\}$ such that $\lim_m \|s_m - x\| = 0$. For each $s_m$ let $s_m = \sum_{1 \leq j \leq a_m}c_{m,j}x_j$ where each $c_{m, j}\in(0, 1)$ and $\sum_{1\leq j \leq a_m}c_{m, j} = 1$. For each $m, k \in \mathbb{N}$, define:
    
    $$
    d_{m,k} =
    \begin{cases}c_{m, j}, \hspace{1cm} k\in \big\{m+j-1\,\vert\, j\in \{1, \cdots, a_m\}\big\} \\
    0, \hspace{1.44cm} \text{otherwise}\\
    \end{cases}
    $$
    Then define a matrix $T = [d_{m, k}]_{m, k \in \mathbb{N}}$. Since each column of $T$ only has one non-zero value, the second condition in \textbf{Definition \ref{Definition 1.2}} is satisfied and obviously for each $m\in\mathbb{N}$, $\sum_{k\in\mathbb{N}}d_{m, k} = 1$, which implies that $T$ is essentially positive.

    \item $b)\implies c)$: Immediate.

    \item $c)\implies a)$: Suppose $T$ is the given regular and essentially positive matrix and the $T$-means $\{t_m\}$ converges to $x$ weakly. WLOG we assume $x \neq 0$. By \textbf{Hahn-Banach Theorem} because there exists $g_x \in X^*$ such that $g_x(x) = \|x\|$ we must have $\liminf\|x_n\| > 0$. Otherwise, we can find $\delta > 0$ such that because $\inf_{n \geq k}\|x_n\| < \delta$, we have: 
    
    $$
    2\delta < \|x\| = g_x(x) \leq \inf_{n \geq k}g_x(x_n) + \delta \leq 2\delta
    $$
    which leas to a contradiction.\\

    \noindent
    Because $\sup_m\sum_{n \in \mathbb{N}}\vert\,c_{m, n}\,\vert < \infty$, for each $\epsilon > 0$ we can find $N_1 \in\mathbb{N}$ such that $\sup_m\sum_{n > N_1}\vert\,c_{m, n}\,\vert < \epsilon$. Meanwhile, because $\forall\,n\in\mathbb{N}, c_{m, n} \overset{m \rightarrow \infty}{\longrightarrow} 0$, we can find $N_2\in\mathbb{N}$ such that $\|\sum_{n \leq N_1}c_{m,n}x_n\| < \epsilon$ for all $n\geq N_2$ (for $\{x_n\}$ being bounded). Now pick $f\in X^*_{\leq 1}$ such that $\|x\|\leq f(x)+\epsilon$. Then find $L \in \mathbb{N}$ bigger than $N_2$ such that:
    
    $$
    \begin{aligned}
    &\hspace{0.44cm} \|x\|\leq f(x)+\epsilon \leq f(t_L)+2\epsilon = f\left[ \sum_{n \in \mathbb{N}}c_{L, n}x_n \right] +2\epsilon \\
    & = f\left[ \sum_{n \leq N_1}c_{L, n}x_n \right]+f\left[ \sum_{n > N_1}c_{L, n}x_n \right] + 2\epsilon \leq f\left[ \sum_{n > N_1}c_{L, n}x_n \right] + 3\epsilon
    \end{aligned}
    $$
    Next we claim that the following statement is false:
    
    $$
    \forall\,n\in\mathbb{N}\, \exists\,n_1 \geq n \hspace{0.2cm}\text{such that}\hspace{0.2cm}\exists\,k \geq n_1\,\forall\,j \in \{n_1, n_1+1, n_1+2, \cdots, k\},\hspace{0.5cm} f \left( \sum_{n_1\leq i\leq k}c_{L,n}x_n \right) > \|x_j\|
    $$
    Assume by contradiction that it is true. Find $n_1 \geq N_1, k_1 \geq n_1$ such that: 
    
    $$
    f\left( \sum_{n_1\leq i\leq k_1}c_{L,n}x_n \right) > \|x_{n_1}\|+\epsilon
    $$
    Then find $n_2 \geq k_1+1, k_2 \geq n_2$ such that:
    
    $$
    f\left(\sum_{n_2\leq i\leq k_2}c_{L,n}x_n \right) > \|x_{n_1}\|+\epsilon
    $$B
    ecause $\liminf\|x_n\| > 0$ we can assume $\inf_{n \geq K}\|x_n\| = \delta > 0$ and by repeating the steps above for infinitely many times we will have:
    
    $$
    f \left( \sum_{i \geq n_1}c_{L,n}x_n \right) > \sum_{i \geq 1}\|x_{n_i}\| > \sum_{i \geq K}\inf_{n\geq K}\|x_n\| = \infty
    $$
    which is absurd. Hence we have the following statement:
    
    \begin{equation}\label{e2}
    \exists\,n_1\in\mathbb{N} \,\forall\,n \geq n_1 \hspace{0.2cm} \text{such that} \hspace{0.2cm} \forall\,k \geq n\,\exists\,j \in \{n, n+1, n+2, \cdots, k\}, \hspace{0.2cm} f\left( \sum_{n_1\leq i\leq k}c_{L,n}x_n \right) \leq \|x_j\|
    \end{equation}
    Now since $\sum_{n < N_1}c_{L, n}x_n$ converges we can find $M_1$ such that $\|\sum_{n > M_1}c_{L, n}x_n\| < \epsilon$. WLOG we can assume $N_1 \geq n_1$ where $n_1$ is given in (\ref{e2}) (since (\ref{e2}) is independent to the choice of $L$). Now find $j_1 \in \{N_1, N_1+1, \cdots, M_1\}$ such that: 
    
    $$
    f\left(\sum_{N_1\leq i\leq M_1}c_{L,n}x_n \right) \leq \|x_j\|
    $$
    and hence:
    
    \begin{equation}\label{e3}
    f \left( \sum_{n > N_1}c_{L,n}x_n \right) + 3\epsilon \leq \|x_{j_1}\| + f\left( \sum_{n > M_1}c_{L,n}x_n \right) +3\epsilon \leq \|x_{j_1}\|+4\epsilon
    \end{equation}
    Since $\epsilon > 0$ is arbitrarily small, for a smaller $\epsilon$ we can find another $N_1$ and the corresponding $j_1$ such that (\ref{e3}) holds. Hence there is a subsequence $\{x_{j_k}\}_{k \in \mathbb{N}}$ such that $\|x\|\leq \limsup_k\|x_{j_k}\| \leq \limsup_i\|x_i\|$ and hence we have $f(x) \leq \limsup_n f(x_n)\,\forall\,f \in X^*$ \\

    \noindent
    Again fix $\epsilon > 0. f \in X^*_{\leq 1}$ and find large $m_1$ such that $\sum_{n \in \mathbb{N}}c_{m_1, n} \in (1-\epsilon, 1+\epsilon)$ and $f(t_{m_1}) \in (f(x)-\epsilon, f(x)+\epsilon)$. Fix $N\in\mathbb{N}$ and find $m_2\in\mathbb{N}$ such that $\|\sum_{n \leq N}c_{m_2, n}x_n\| < \epsilon$ (because $\lim_m c_{m, n}=0$ for all $n\in\mathbb{N}$ and $\{x_n\}$ bounded) and then we have: 
    
    $$
    f\left(\sum_{n > N}c_{m_2,n}x_n \right) \in \big( f(x) - 2\epsilon, f(x) + 2\epsilon \big)
    $$
    Therefore for $m > \max(m_1, m_2)$ we have:
    
    $$
    f(x)+2\epsilon > f \left( \sum_{n > N}c_{m,n}x_n \right) \geq \inf_{n > N}f(x_n)\sum_{n > N}c_{m,n} > \inf_{n > N}f(x_n)(1-\epsilon)
    $$
    When $\epsilon$ is getting smaller, we need a bigger $N$ such that $\|\sum_{n \leq N}c_{m, n}x_n\| < \epsilon$. Hence we can conclude $f(x) \geq \liminf_n f(x_n)$. Therefore, we have $\liminf_n f(x_n) \leq f(x) \leq \limsup_n f(x_n)$ for all $f \in X^*_{\leq 1}$. We can now conclude the closed unit ball in $X$ is weakly compact and hence $X$ is reflexive.
 
\end{enumerate}
\end{proof}

\begin{cor}[\cite{14}]\label{Corollary 1.7}
A Banach space that possesses Banach-Saks Property is reflexive.
\end{cor}

\begin{proof}
Define $T = [t_{i,j}]_{i,j\in\mathbb{N}}$ where $t_{i,j} = \dfrac{1}{i}$ when $i \leq j$ and zero otherwise. Then $T$ is a regular summable and essentially positive matrix and each weakly convergent (therefore bounded) sequence will have its $T$-means converging in norm. By \textbf{Theorem \ref{Theorem 1.6}} we have the space is reflexive.
\end{proof}

\begin{rem}
A counter-example to the converse of \textbf{Corollary \ref{Corollary 1.7}} can be found in \cite{1}.
\end{rem}

\begin{defn}
In a Banach space $X$ with a Schauder basis $\{x_n\}$, we call $\{x_n\}$ \textbf{boundedly complete} if $\|\sum_{i \leq n}a_ix_i\| < M\,\forall\, n \in \mathbb{N}$ for some fixed $M \in \mathbb{N}$ implies $\sum_{i \geq 1}a_ix_i$ converges.
\end{defn}

\begin{theorem}[\cite{14}]
A Schauder basis in a Banach space with property $w\mathcal{P}$ will be boundedly complete.
\end{theorem}

\begin{proof}
Let $\{x_n\}$ be the Schauder basis of $X$ and $T = [c_{m, n}]_{m, n\in\mathbb{N}}$ be a regular summarble method. Suppose $\{v_n\}_n$ is a uniformly bounded by $M\in\mathbb{N}$ and that $v_n = \sum_{i\leq n} a_i x_i$ for each $n\in\mathbb{N}$. Then by \textbf{Theorem 1.9}, the $T$-means of a subsequence $\{v_n\}$, say $\{t_m\}$, converges to some $x$ in norm. Assume:

$$
x = \sum_{i\in\mathbb{N}}b_i x_i, \hspace{0.5cm} 
$$
and, by restricting to the proper subsequence given by \textbf{Theorem \ref{Theorem 1.6}}, we can write each $t_m$ in the following form:

$$
\begin{aligned}
& \hspace{1.02cm} t_m = \sum_{n\in\mathbb{N}}c_{m, n}v_n = \sum_{n\in\mathbb{N}} c_{m, n}\sum_{i\leq n}a_i x_i\\
&\implies\, t_m = \lim_n \sum_{k\leq n}c_{m, k}\sum_{i\leq k}a_i x_i = \lim_n \sum_{i\leq n} \left[ a_i \big( \sum_{i\leq k\leq n}c_{m, k}\big)\right] x_i\\
\end{aligned}
$$
Meanwhile, since the sequence $\Big\{ \sum_{n\in\mathbb{N}}c_{m, n}\Big\}_m$ is bounded and converges to $1$, then for each $m, n\in\mathbb{N}$, we have:

$$
\begin{aligned}
&\hspace{1.44cm} \left\| \sum_{i\leq n}\left[ a_i\left( \sum_{i\leq k \leq n}c_{m, k}\right) \right] x_i - \sum_{i\leq n}\left[ a_i \left( \sum_{k\geq i}c_{m, k}\right) \right] x_i \right\| = \left\| \sum_{i\leq n}\left[ a_i \left( \sum_{k>n}c_{m, k}\right) \right] x_i \right\|\\
&\hspace{1cm} = \left\vert\, \sum_{k>n}c_{m, k} \,\right\vert \left\| \sum_{i\leq n}a_ix_i\right\| \leq M \left\vert\, \sum_{k>n}c_{m, k}\,\right\vert \overset{n\rightarrow\infty}{\longrightarrow} 0\\
& \implies\, t_m = \lim_n \sum_{i\leq n}\left[ a_i\left( \sum_{i\leq k \leq n}c_{m, k}\right) \right]x_i = \lim_n \sum_{i\leq n}\left[ a_i\left( \sum_{k\geq i}c_{m, k}\right) \right]x_i = \sum_{i\in\mathbb{N}}\left[ a_i\left( \sum_{k\geq i}c_{m, k}\right) \right]x_i \overset{m\rightarrow\infty}{\longrightarrow} x\\
&\implies\, \forall\,i\in\mathbb{N}, \hspace{0.3cm} a_i\lim_m \sum_{k\geq i}c_{m, k} = b_i
\end{aligned}
$$
By \textbf{Definition \ref{Definition 1.2}}, we have that for each $n\in\mathbb{N}$, $\lim_m c_{m, n} = 0$. Therefore, for each $i\in\mathbb{N}$:

$$
\begin{aligned}
& \hspace{1cm} \lim_m \sum_{k<i}c_{m, k} = \sum_{k<i}\lim_m c_{m, k} = 0\\
&\implies\, a_i \lim_m\sum_{k\geq i}c_{m, k} =  a_i \lim_m\sum_{k\geq i}c_{m, k} + \lim_m a_i \sum_{k<i}c_{m, k} = a_i \lim_m \sum_{k\geq 1}c_{m, k} = b_i\\
&\implies\, a_i = b_i
\end{aligned}
$$
Hence $\sum_{i\in\mathbb{N}}a_ix_i$ converges.

\end{proof}

\section{Characterizations by dual space}

In this section, we will discover the equivalence between non-reflexivity of a Banach space and behaviors of bounded linear functionals and all main results are from \cite{9}. As an immediate consequence of {\cite[Goldstine's Theorem]{22}} (or one can refer to {\cite[Theorem $V.4.7$]{3}}), a Banach space is reflexive if and only if its closed unit ball is weakly compact, if and only if every bounded linear functional attains supremum on its closed unit ball. Besides, in a non-reflexive Banach space, we can see how much the weak topology differs from the norm topology in the dual space by showing there exists a weakly vanishing sequence of bounded linear functionals, whose convex hull's closure is positively distant away from zero.

\subsection{Main Theorems}

\begin{theorem}\label{Theorem 2.1}
Given $B$ a separable Banach space, the following statements are equivalent:

\begin{enumerate}[label = \alph*)]

    \item $B$ is not reflexive.
    \item If $\theta \in (0, 1)$ then there is a sequence $\{f_n\}\subseteq B^*_{\leq 1}$ such that $\|f\|\geq\theta$ for all $f \in \operatorname{conv}\{f_n\}$ and $(f_n)_{n\in \mathbb{N}}$ converges to zero pointwise.
    
    \item If $\theta\in (0, 1)$ and $\{\lambda_n\}\subseteq (0, 1)$ such that $\sum_{n \geq 1}\lambda_n = 1$ then there is $\alpha \in [\theta, 1]$ and $\{g_n\}\subseteq B^*_{\leq 1}$ converging pointwise to zero where $\|\sum_{n \geq 1}\lambda_n g_n\| = \alpha$ and for each $n$:
    
    $$
    \left\|\sum_{i \leq n}\lambda_i g_i \right\|\leq \alpha\left(1-\theta\sum_{i > n}\lambda_i\right).
    $$
    
    \item There is a $f \in B^*$ such that $f$ cannot obtain its maximum on the closed unit ball of $B$.
    
\end{enumerate}
\end{theorem}

\begin{theorem}\label{Theorem 2.2}
\textup{\textbf{Theorem \ref{Theorem 2.1}}} will also be true when $B$ is not separable.
\end{theorem}

\begin{theorem}\label{Theorem 2.3}
If $X$ is a separable bounded weakly closed subset of a complete LCTVS, then the following are equivalent:

\begin{enumerate}
    \item $X$ is not weakly compact.
    
    \item There is a $\theta > 0$ and $\{f_n\}\subseteq B^{\ast}_{\leq 1}$ such that $\|f\|\geq\theta$ for all $f \in \operatorname{conv}\{f_n\}$ and $\{f_n\}$ converge to zero pointwise on $X$.
    
    \item There exists $\theta > 0$ such that, if $\{\lambda_n\}\subseteq (0, 1)$ is a sequence with $\sum_{n \geq 1}\lambda_n = 1$, then there is $\alpha \geq \theta$ and $\{g_n\}\subseteq B^{\ast}_{\leq 1}$ converging to zero pointwise on $X$ such that:

    $$
    \left\| \sum_{n\geq 1}\lambda_n g_n\right\| = \alpha
    $$
    and, for each $x\in X$ and for each $n\in\mathbb{N}$:
    
    $$
    \left\| \sum_{i \leq n}\lambda_i g_i(x)\right\| < \alpha \left( 1 - \theta\sum_{i > n}\lambda \right).
    $$
    \item There exists $f\in X^*$ such that $f$ cannot attain $\|f\|$ on the closed unit ball of $X$.
    
\end{enumerate}
\end{theorem}

\begin{theorem}\label{Theorem 2.4}
\textup{\textbf{Theorem \ref{Theorem 2.3}}} will also be true when $X$ is not separable.
\end{theorem}

\subsection{Important Lemmas}

\begin{lem}\label{Lemma 2.5}

Given a Banach space $B, \theta \in (0, 1)$ and $\{f_n\}\subseteq B^*_{\leq 1}$ with $\|f\|\geq\theta$ for all $f \in \operatorname{conv}\{f_n\}$, if $\{\lambda_n\}\subseteq (0, 1)$ satisfy $\sum_{n \geq 1}\lambda_n = 1$ then we can find $\alpha \in [\theta, 1]$ and $\{g_n\}\subseteq B^*$ where, for each $k \in \mathbb{N}, g_k \in V_k(\{f_n\}) = \operatorname{conv}\{f_i\}_{i \geq k}$ such that: 

$$
\left\| \sum_{k \geq 1}\lambda_k g_k \right\| = \alpha
$$
and for each $n\in \mathbb{N}$: 

$$
\left\| \sum_{i \leq n}\lambda_i g_i \right\| < \alpha\left( 1-\theta\sum_{i > n}\lambda \right)
$$

\end{lem}

\begin{proof}

For each $\lambda_k$ find $\epsilon_k\in(0, 1)$ such that:

$$
\sum_{k\geq 1}\frac{\epsilon_k}{\lambda_{k+1}} < 1-\theta
$$
so that:

\begin{equation}\label{e5}
\sum_{k \geq 1}\frac{\lambda_k\epsilon_k}{\sum_{i > k}\lambda_i\sum_{i \geq k}\lambda_i} < \sum_{k \geq 1} \frac{\lambda_k\epsilon_k}{\lambda_k \lambda_{k+1}} = \sum_{k\geq 1}\frac{\epsilon_k}{\lambda_{k+1}} < 1-\theta
\end{equation}
Define:

$$
\alpha_n = \inf \left\{ \left\|\sum_{i < n}\lambda_i g_i + \left( \sum_{i \geq n}\lambda_i \right)g \right\| \,\vert\, g \in V_n(\{f_i\}) \right\}
$$
and then find $g_n \in V_n(\{f_i\})$ such that:

$$
\left\| \sum_{i < n}\lambda_i g_i+\left( \sum_{i \geq n}\lambda_i \right)g_n \right\| < \alpha_n(1+\epsilon_n)
$$
Since $V_{n+1}(\{f_i\}) \subseteq V_n(\{f_i\})$, for each $g\in V_{n+1}(\{f_i\})$, we have:

$$
\begin{aligned}
&\hspace{1cm}  \dfrac{\lambda_n}{\sum_{i\geq n}\lambda_i}g_n + \dfrac{\sum_{i>n}\lambda_i}{\sum_{i\geq n}\lambda_i}g \in V_n\\
&\implies\, \left\{ \lambda_n g_n + \sum_{i > n}\lambda_i g: g \in V_{n+1}(\{f_i\}) \right\} \subseteq \left\{ \sum_{i \geq n}\lambda_i g: g \in V_n(\{f_i\}) \right \}
\end{aligned}
$$
Hence we have $\alpha_n$ is non-decreasing. Since $\{f_i\}\subseteq B^*_{\leq 1}$ we have $\alpha_n \leq 1$ for all $n \in \mathbb{N}$ and hence $\alpha_n \uparrow \alpha = \sup_n\alpha_n$ exists such that $\alpha = \|\sum_{i \geq 1}\lambda_i g_i\|$ (because $\sum_{i > n}\lambda_i \overset{n \rightarrow \infty}{\rightarrow} 0$). Therefore for each $n\in\mathbb{N}$, by (\ref{e5}):

$$
\begin{aligned}
&\hspace{1.44cm} \left\| \sum_{i \leq n}\lambda_i g_i \right\| \\
&\hspace{1cm} = \left\| \left( \frac{\lambda_n+\sum_{i > n}\lambda_i}{\sum_{i \geq n}\lambda_i} \right)\sum_{i < n}\lambda_i g_i + \frac{\lambda_n}{\sum_{i \geq n}\lambda_i}(\sum_{i \geq n}\lambda_i) g_n \right\|\\
&\hspace{1cm} \leq \frac{\lambda_n}{\sum_{i \geq n}\lambda_i} \left\| \sum_{i < n}\lambda_i g_i \right\| + \frac{\sum_{i > n}\lambda_i}{\lambda_{i \geq n}\lambda_i} \left\| \sum_{i < n}\lambda_i g_i \right\|\\
&\hspace{1cm} < \left( \sum_{i > n}\lambda_i \right) \left[\frac{\lambda_n\alpha_n(1+\epsilon_n)}{\sum_{i > n}\lambda_i\sum_{i \geq n}\lambda_i} + \frac{1}{\sum_{i \geq n}\lambda_i} \left\| \sum_{i < n}\lambda_i g_i \right\| \right]\\
&\implies\, \dfrac{1}{\sum_{i>n}\lambda_i} \left\|\sum_{i\leq n}\lambda_i g_i \right\| < \dfrac{\lambda_n\alpha_n(1+\epsilon)} {\sum_{i>n}\lambda_i \sum_{i\geq n}\lambda_i} + \dfrac{1}{\sum_{i\geq n}\lambda_i} \left\|\sum_{i<n}\lambda_i g_i \right\| 
\end{aligned}
$$
By induction we will have:

\begin{equation}\label{e6}
\begin{aligned}
&\hspace{0.44cm} \left\|\sum_{i \leq n}\lambda_i g_i \right\|\\
& < \left(\sum_{i > n}\lambda_i \right)\left[ \frac{\lambda_n\alpha_n(1+\epsilon_n)}{\sum_{i > n}\lambda_i\sum_{i \geq n}\lambda_i} + \frac{1}{\sum_{i \geq n}\lambda_i} \left\|\sum_{i < n}\lambda_i g_i \right\| \right]\\
& < \left(\sum_{i > n}\lambda_i \right) \left[ \frac{\lambda_n\alpha_n(1+\epsilon_n)}{\sum_{i > n}\lambda_i\sum_{i \geq n}\lambda_i} + \frac{\lambda_{n-1}\alpha_{n-1}(1+\epsilon_{n-1})}{\sum_{i > n-1}\lambda_i\sum_{i \geq n-1}\lambda_i} + \frac{1}{\sum_{i \geq n-1}\lambda_i} \left\| \sum_{i < n-1}\lambda_i g_i \right\| \right]\\
& < \left( \sum_{i > n}\lambda_i \right) \left[ \sum_{2 \leq k \leq n}\frac{\lambda_k\alpha_k(1+\epsilon_k)}{\sum_{j \geq k}\lambda_j\sum_{j > k}\lambda_j} + \frac{1}{\sum_{j \geq 2}\lambda_j}\|\lambda_1 g_1\| \right]\\
& =  \left(\sum_{i > n}\lambda_i \right) \left[ \sum_{1\leq k \leq n}\frac{\lambda_k\alpha_k(1+\epsilon_k)}{\sum_{j \geq k}\lambda_j\sum_{j > k}\lambda_j} \right]
\end{aligned}
\end{equation}
Since $\alpha=\sup_n\alpha_n$, by (\ref{e6}) we have:

$$
\begin{aligned}
&\hspace{0.44cm} \left\| \sum_{i \leq n}\lambda_i g_i \right\| < \alpha \left( \sum_{i>n}\lambda_i \right) \left[ \sum_{1\leq k \leq n} \dfrac{\lambda_k (1+\epsilon_k)}{\sum_{j\geq k}\lambda_j \sum_{j>k}\lambda_j} \right]\\
&= \alpha \left( \sum_{i > n}\lambda_i\right) \left[ \sum_{1\leq k \leq n}\dfrac{\lambda_k}{\sum_{j \geq k}\lambda_j \sum_{j>k}\lambda_j} + \sum_{1\leq k \leq n}\dfrac{\lambda_k\epsilon_k}{\sum_{j \geq k}\lambda_j \sum_{j>k}\lambda_j} \right]\\
&< \alpha \left(\sum_{i > n}\lambda_i \right) \left[ \sum_{1\leq k \leq n} \left( \frac{1}{\sum_{j > k}\lambda_j} - \frac{1}{\sum_{j \geq k}\lambda_j} \right) + (1-\theta) \right]\\
& = \alpha\left( \sum_{i>n}\lambda_i \right) \left[ \dfrac{1}{\sum_{i>n}\lambda_i} - 1 + 1 - \theta\right] = \alpha\left( 1 - \theta\sum_{i>n}\lambda_i \right)
\end{aligned}
$$

\end{proof}

\begin{defn}
Given a normed linear space $X$ and $\{\phi_n\}_n \subseteq X^*_{\leq 1}$ we define:

$$
L(\{\phi_n\}_n) = \{\omega\in X^*\,\vert\,\,\forall\,x \in X, \liminf_n \phi_n(x)\leq \omega(x)\leq \limsup_n \phi_n(x)\}
$$
\end{defn}

\begin{lem}\label{Lemma 2.7}
Given a Banach space $B, \theta \in (0, 1), \{f_n\}_n\subseteq B^*_{\leq 1}$ such that:

$$
\forall\,f\in \operatorname{conv}\{f_n\}_{n\in \mathbb{N}}\,\forall\,\omega \in L\big( \{f_n\}_{n\in\mathbb{N}} \big), \hspace{0.3cm} \|f-\omega\|\geq\theta
$$
If $\{\lambda_i\}\subseteq (0,\infty)$ satisfies $\sum_{i \geq 1}\lambda_i = 1$, then we can find $\alpha \in [\theta, 2]$ and $\{g_i\}\subseteq B^*_{\leq 1}$ such that for any $n\in\mathbb{N}$ and $\omega \in L(\{\phi_n\})$:

$$
\left\| \sum_{k \geq 1}\lambda_k(g_k-\omega) \right\| = \alpha, \hspace{1cm}
\left\| \sum_{k \leq n}\lambda_k(g_k-\omega) \right\| < \alpha\left( 1-\theta\sum_{i > n}\lambda \right)
$$
\end{lem}

\begin{proof}
Similar to the proof of \textbf{Lemma \ref{Lemma 2.5}}, find $\{\epsilon_k\}_k \subseteq (0, 1)$ such that:

$$
\sum_{k \geq 1}\frac{\lambda_k\epsilon_k}{\sum_{j > k}\lambda_j\sum_{j \geq k}\lambda_j} < 1-\theta
$$
For each $n$ and each sequence of bounded linear functionals $\{\phi_i\} \subseteq X^{\ast}$, define $V_n(\{\phi_i\}) = \operatorname{conv}\{\phi_i\}_{i \geq n}$ and define $V(\{\phi_i\}_i)$ be the set of all sequences of linear functionals $\{\psi_j\}_j$ such that for each $j, \psi_j\in V_j(\{\phi_i\}_i)$. In this case for every $\{\psi_j\}_j \in V(\{\phi_i\}_i)$ we will have $V(\{\psi_j\}_j) \subseteq V(\{\phi_i\})$. To complete the proof, for each $n\in\mathbb{N}$, we will find a real number $\alpha_n$, a bounded linear functional $g_n\in X^{\ast}_{\leq 1}$ and sequences of bounded linear functionals $\{\phi_i^n\}, \{\psi_j^n\} \subseteq X^{\ast}_{\leq 1}$.\\

\noindent
First let $\{\psi_j^0\}_j = \{f_j\}_j$, the sequence given by the lemma. Starting from $n = 1$, define:

$$
\alpha_n = \inf \left\{ \sup_{\omega \in L(\{\phi_i\})} \left\| \sum_{i < n}\lambda_i g_i + \left(\sum_{i \geq n}\lambda_i \right)g - \omega \right\|: g \in V_n\left (\{\psi_j^{n-1}\}_j \right), \{\phi_i\}_i \in V\left( \{\psi_j^{n-1}\}_j \right) \right\}
$$
Then pick $g_n \in V_n(\{\psi_j^{n-1}\}_j), \{\phi_i^n\}_i \in V(\{\psi_j^{n-1}\}_j)$ such that:

\begin{equation}\label{e7}
\alpha_n  \leq \sup_{\omega \in L(\{\phi_i^n\}_{i \geq 1})} \left\|\sum_{i < n}\lambda_i g_i + \left(\sum_{i \geq n}\lambda_i \right) g_n - \omega \right\| < \alpha_n(1+\epsilon_n)
\end{equation}
Then pick $\omega_n \in L(\{\phi_i^n\}_{i \geq 1})$ such that:

$$
\alpha_n(1-\epsilon_n) < \left\|\sum_{i < n}\lambda_i g_i + \left( \sum_{i \geq n}\lambda_i \right)g_n - \omega_n \right\| < \alpha_n(1+\epsilon_n)
$$
and $x_n \in B_{\leq 1}$ such that:

$$
\alpha_n(1-\epsilon_n) < \sum_{i < n}\lambda_i g_i(x_n) + \left(\sum_{i \geq n}\lambda_i \right)g_n(x_n) - \omega_n(x_n)
$$
Find $\{\psi_j^n\}_j \subseteq\{\phi_i^n\}_i$ such that $\lim_j\psi_j(x_n) = \liminf_i\phi_i(x_n)$. Since $\{\psi_j^n(x_n)\}_j$ is convergent, then for each $\omega \in L(\{\psi_j^n\}_{j \geq 1}), \omega(x_n) = \lim_j\psi_j(x_n) = \liminf_i\phi_i(x_n) \leq \omega_n(x_n)$. Therefore for each $\omega \in L(\{\psi_j^n\}_j)$:

\begin{equation}\label{e8}
\alpha_n(1-\epsilon_n) < \sum_{i < n}\lambda_i g_i(x_n) + \left(\sum_{i \geq n}\lambda_i \right)g_n(x_n) - \omega(x_n)
\end{equation}
After finding each $g_n$, because $g_n \in V_n(\{\psi_i^{n-1}\})_i$, we will have $\{g_i\}_{i > n} \subseteq V(\{\psi_j^n\}_j)$ and hence $L(\{g_i\}_{i\geq n}) \subseteq L(\{\psi_j^n\}_j) \subseteq L(\{\phi_i^n\})$. Notice that for each $n\in\mathbb{N}$, $L(\{g_i\}_{i\geq 1}) = L(\{g_i\}_{i\geq n})$. Then by (\ref{e7}) and (\ref{e8}), for each $n\in\mathbb{N}$ and $\omega \in L(\{g_i\}_{i\geq n})$, we have:

\begin{equation}\label{e9}
\alpha_n(1-\epsilon_n) < \left\| \sum_{i< n}\lambda_i g_i + \left( \sum_{i\geq n}\lambda_i \right) g_n - \omega \right\| \leq \sup_{\omega \in L(\{\phi_i^n\}_{i \geq 1})} \left\| \sum_{i < n}\lambda_i g_i + \left( \sum_{i \geq n}\lambda_i \right) g_n - \omega \right\| < \alpha_n(1+\epsilon_n)
\end{equation}
Since, for each $n\in\mathbb{N}$, $V(\{\psi_j^{n-1}\}_j)\subseteq V(\{\psi_j^n\}_j)$, similar to the proof of \textbf{Lemma \ref{Lemma 2.5}}, for each $n\in\mathbb{N}$, we have:

$$
\begin{aligned}
&\hspace{0.42cm} \left\{ \sup_{\omega \in L(\{\phi_i\})}\left\| \sum_{i < n+1}\lambda_i g_i + \left(\sum_{i \geq n+1}\lambda_i \right)g - \omega \right\|: g \in V_{n+1}\left( \{\psi_j^n\}_j \right), \{\phi_i\}_i\in V\left(\{\psi_j^n\}_j \right) \right\} \\
& \subseteq \left\{ \sup_{\omega \in L(\{\phi_i\})} \left\| \sum_{i < n}\lambda_i g_i + \left( \sum_{i \geq n}\lambda_i \right) g - \omega \right\|: g \in V_n\left( \{\psi_j^{n-1}\}_j \right), \{\phi_i\}_i\in V\left( \{\psi_j^{n-1}\}_j \right)\right\} \\
\end{aligned}
$$
Hence we have $\alpha_n$ is non-decreasing. Since for each $\omega \in V(\{f_n\}_n)$, we have $\omega\leq 1$, we have $\alpha_n\leq 2$ for each $n\in\mathbb{N}$, which implies that $\alpha = \lim_n\alpha_n = \sup_n\alpha_n$ exists. Since $\alpha_1\geq \theta$, we then have $\theta\leq\alpha \leq 2$. For each $\omega \in L(\{g_i\}_i)$, by (\ref{e9}) we have:

$$
\begin{aligned}
&\hspace{1cm} \alpha = \lim_n \alpha_n(1-\epsilon_n) \leq \lim_n \left\| \sum_{i<n}\lambda_ig_i + \left( \sum_{i\geq n}\lambda_i g_i \right) g_n - \omega \right\| \leq \lim_n\alpha_n(1+\epsilon_n) = \alpha\\
&\implies\, \left\| \sum_{i\in\mathbb{N}} \lambda_ig_i - \omega \right\| = \alpha
\end{aligned}
$$
Fix $n\in\mathbb{N}$ and $\omega \in L(\{g_i\}_i)$. Then by (\ref{e9}):

$$
\begin{aligned}
&\hspace{1.48cm} \left\| \sum_{i \leq n}\lambda_i (g_i - \omega) \right\| = \left\| \frac{\lambda_n+\sum_{i > n}\lambda_i}{\sum_{i \geq n}\lambda_i} \left[ \sum_{i < n}\lambda_i (g_i-\omega) \right] + \lambda_n (g_n-\omega) \right\| \\
&\hspace{1cm} \leq \frac{\lambda_n}{\sum_{i \geq n}\lambda_i} \left\| \sum_{i < n}\lambda_i (g_i-\omega) + \left( \sum_{i \geq n}\lambda_i \right)(g_n-\omega) \right\| + \frac{\sum_{i > n}\lambda_i}{\sum_{i \geq n}\lambda_i} \left\| \sum_{i < n}\lambda_i(g_i - \omega) \right\|\\
&\hspace{1cm} \leq \frac{\lambda_n}{\sum_{i \geq n}\lambda_i} \left\| \sum_{i<n}\lambda_ig_i - \left(\sum_{i\geq n}\lambda_i \right)g_n - \omega \right\| + \frac{\sum_{i > n}\lambda_i}{\sum_{i \geq n}\lambda_i} \left\| \sum_{i < n}\lambda_i(g_i - \omega) \right\|\\
&\hspace{1cm} < \left( \sum_{i > n}\lambda_i \right) \left[ \frac{\lambda_n\alpha_n(1+\epsilon_n)}{\sum_{i > n}\lambda_i\sum_{i \geq n}\lambda_i} + \frac{1}{\sum_{i \geq n}\lambda_i} \left\| \sum_{i < n}\lambda_i(g_i - \omega) \right\| \right]\\
&\implies\, \dfrac{1}{\sum_{i>n}\lambda_i}\left\| \sum_{i\leq n}\lambda_i (g_i-\omega) \right\| < \dfrac{\lambda_n\alpha_n(1+\epsilon_n)}{\sum_{i>n}\lambda_i \sum_{i\geq n}\lambda_i} + \dfrac{1}{\sum_{i\geq n}\lambda_i} \left\| \sum_{i<n}\lambda_i(g_i-\omega) \right\|
\end{aligned}
$$
Then by induction on $n$ we have:

\begin{equation}\label{e10}
\begin{aligned}
&\hspace{0.44cm} \left\| \sum_{i \geq n}\lambda_i(g_i - \omega) \right\|\\
&< \left( \sum_{i > n}\lambda_i \right) \left[ \frac{\lambda_n\alpha_n(1+\epsilon_n)}{\sum_{i > n}\lambda_i\sum_{i \geq n}\lambda_i} + \frac{\lambda_{n-1}\alpha_{n-1}(1+\epsilon_{n-1})}{\sum_{i > n-1}\lambda_i\sum_{i \geq n-1}\lambda_i} + \frac{1}{\sum_{i \geq n-1}\lambda_i} \left\| \sum_{i < n-1}\lambda_i(g_i-\omega) \right\| \right]\\
&< \left( \sum_{i > n}\lambda_i \right) \left[ \sum_{2 \leq k \leq n}\frac{\lambda_k\alpha_k(1+\epsilon_k)}{\sum_{i > k}\lambda_i\sum_{i \geq k}\lambda_k} + \frac{1}{\sum_{i \geq 2}\lambda_i}\|\lambda_1(g_1-\omega)\| \right]\\
&< \left( \sum_{i>n}\lambda_i \right) \sum_{1\leq k \leq n} \frac{\lambda_k\alpha_k(1+\epsilon_k)}{\sum_{i>k}\lambda_i \sum_{i\geq k}\lambda_i}
\end{aligned}
\end{equation}
Since $\alpha = \sup_n\alpha_n$, from (\ref{e10}) we have:

$$
\begin{aligned}
&\hspace{0.48cm}\left\| \sum_{i \leq n}\lambda_i(g_i - \omega) \right\|\\
&< \alpha \left( \sum_{i>n}\lambda_i \right) \left[ \sum_{1\leq k \leq n}\frac{\lambda_k}{\sum_{i>k}\lambda_i \sum_{i\geq k}\lambda_i} + \sum_{1\leq k \leq n}\frac{\lambda_k\epsilon_k}{\sum_{i>k}\lambda_i \sum_{i\geq k}\lambda_i} \right]\\
&< \alpha \left( \sum_{i > n}\lambda_i \right) \left[ \sum_{k \leq n}\left( \frac{1}{\sum_{i > k}\lambda_i}-\frac{1}{\sum_{i \geq k}\lambda_i} \right) + (1-\theta) \right]\\
&= \alpha \left( 1-\theta\sum_{i > n}\lambda_i \right)
\end{aligned}
$$

\end{proof}

\subsection{Proof of Main Theorems}

\begin{proof}[Proof of Theorem 2.1:]
$\hspace{0.44cm}\\$
\begin{itemize}
    \item $1)\implies 2)$:
    Suppose $Q$ is the canonical mapping from $X$ to $X^{\ast\ast}$ and $B$ is not reflexive. Now pick $F\in X^{\ast\ast}$ such that $d[F, Q(B)] > \theta$ where $\|F\| < 1$ and $\theta \in (0, 1)$. Suppose $\{x_n\}$ is dense in $B$ and for each $n\in\mathbb{N}$ pick $f_n \in B^*$ such that:
    
    \begin{enumerate}[label = \alph*)]
        \item $\|f_n\| < 1$
        \item $F(f_n) = 0$
        \item $f_n(x_i) = 0\,\forall\,i \leq n$
    \end{enumerate}

    \noindent
    Indeed, according to \textbf{Theorem \ref{Theorem 1.16}}, this will be true if there is $M > 0$ such that $\theta \leq M\|t_0 F + \sum_{i \leq n}t_if_i\|$ for any $\{t_i\}_{i\leq n}\subseteq \mathbb{C}$. Here if we let $M = \dfrac{\theta}{d[F, Q(X)]} < 1$, then given $F, Q(x_1) \in X^{\ast\ast}$, we can find $f_1 \in X^{\ast}_{< 1}$ such that $F(f_1) = \theta$ and $Q(x_1)(f_1) = f_1(x_1) = 0$. Next, with $F, Q(x_1), Q(x_2)\in X^{\ast\ast}$, we then can find $f_2\in X^{\ast}_{< 1}$ such that $F(f_2) = \theta$, $f_2(x_1) = f_2(x_2) = 0$. By induction, we can find $\{f_n\}\subseteq X^{\ast}_{< 1}$ that satisfies conditions $a), b)$ and $c)$. Notice that, for each $n\in\mathbb{N}$:

    $$
    \|f_n\| = \sup_{\phi\in X^{\ast\ast}} \vert\, \phi(f_n)\,\vert
    $$
    Hence $\|f_n\|\geq\theta$ for each $n$ and hence for each $f\in \operatorname{conv}\{f_n\}$, $\|f\|\geq\theta$. Meanwhile, we have:

    $$
    \forall\,i\in\mathbb{N}, \hspace{0.3cm} \lim_n f_n(x_i) = 0 \hspace{0.3cm}\implies\hspace{0.3cm} \forall\,x\in X, \hspace{0.3cm}\lim_n f_n(x) = 0
    $$
    
    \item $2)\implies 3)$:
    Given $\{f_n\} \subseteq X^{\ast}_{\leq 1}$, the set $\{g_n\} \subseteq \operatorname{conv}\{f_n\}$ and the $\alpha\in [\theta, 1]$ obtained by \textbf{Lemma \ref{Lemma 2.5}} completes the proof 
    
    \item $3)\implies 4)$:
    We will show $\sum_{i \geq 1}\lambda_i g_i$ does not attain its $\sup$ in the closed unit ball of $X$. Fix $x \in B$ and find $n\in\mathbb{N}$ such that $g_i(x) < \alpha\theta\,\forall\,i > n$. Then:
    
    $$
    \begin{aligned}
    &\hspace{0.44cm} \sum_{i \geq 1}\lambda_i g_i(x) \\
    &< \sum_{i \leq n}\lambda_i g_i(x) + \alpha\theta\sum_{i > n}\lambda_i \\
    &< \left\| \sum_{i \geq n}\lambda_i g_i \right\| + \alpha\theta\sum_{i < > n}\lambda_i \\
    & < \alpha \left( 1-\theta\sum_{i > n}\lambda_i \right) + \alpha\sum_{i > n}\lambda_i = \alpha
    \end{aligned}
    $$
    Because $\|\sum_{i \geq 1}\lambda_i g_i\| = \alpha$, such function never obtains its norm in the closed unit ball.
    
    \item $4)\implies 1)$:
    If $B$ is reflexive, then foe each $f \in X^*_{\leq 1}$ we can find $F \in X^{\ast\ast}_{\leq 1}$ such that $F(f) = \|f\|$. Then we can find $x \in X_{\leq 1}$ such that $Q(x) = F$ and hence $f(x) = \|f\|$
    
\end{itemize}
\end{proof}

\begin{proof}[Proof of Theorem 2.2:]
$\hspace{0.44cm}\\$
\begin{itemize}

    \item $1)\implies 2)$:
    According to the \textbf{Eberlein–Šmulian's Theorem}, if $B$ is not reflexive, there exists a (linearly) independent sequence $\{x_i\}_{i\in\mathbb{N}}\subseteq B_{\leq 1}$ such that $\{x_i\}$ has no clustered points in weak topology. Therefore, let $X$ denote the closure of the span of $\{x_i\}$ and then $X$ will not be reflexive. By \textbf{Theorem \ref{Theorem 2.1}}, there is $\{f_n\}\subseteq X^{\ast}_{\leq 1}$ such that $\{f_n\}$ converges to zero pointwise on $X$ and $\left\| f\Big|_X \right\|\geq\theta$ for each $f\in\operatorname{conv}\{f_n\}$. Extending each $f_n$ by setting $f_n\Big|_{B\backslash X} = 0$ completes the proof.
    
    \item $2)\implies 3)$:
    Since $\{f_n\}$ converges to zero pointwise, $L(\{f_n\}) = \{0\}$ and hence the proof follows by \textbf{Lemma \ref{Lemma 2.7}}.
    
    \item $3)\implies 4)$:
    By making $\lambda_1$ comparatively large enough, we can find $\delta \in (0, \frac{1}{2}\theta^2)$ such that for each $n, \lambda_{n+1} < \delta\lambda_n$. We then will show that $\sum_{i \geq 1}\lambda_i (g_i - \omega)$ does not obtain its $\sup$ in the closed unit ball for each $\omega \in L(\{g_i\})$. By the fact that $\{g_n\}$ converges to zero pointwise, fix an arbitrary $x \in B_{\leq 1}$ and then we can find large $m, n \in \mathbb{N}\, (m \leq n)$ such that: 
    
    $$
    \begin{aligned}
    & \hspace{1cm} g_{n+1}(x) < \inf_{k \geq m}g_k(x) + \theta^2 - 2\delta \leq \liminf_k g_k(x) + \theta^2 - 2\delta \leq \omega(x) + \theta^2 - 2\delta\\
    &\implies\, g_{n+1}(x) - \omega(x) < \theta^2 - 2\delta
    \end{aligned}
    $$
    According to the proof in \textbf{Lemma \ref{Lemma 2.7}} we have $\theta \leq \alpha$ and hence $g_{n+1}(x) - \omega(x) \leq \alpha\theta - 2\delta$, which implies:
    
    $$
    \begin{aligned}
    &\hspace{0.42cm}\sum_{i \geq 1}\lambda_i (g_i - \omega)(x)\\
    & < \sum_{i \leq n}\lambda_i (g_i - \omega)(x) + (\alpha\theta - 2\delta)\lambda_{n+1} + \sum_{i > n+1}\lambda_i(g_i - \omega)(x)\\
    & \leq \left\| \sum_{i \leq n}\lambda_i (g_i - \omega) \right\| + (\alpha\theta - 2\delta)\lambda_{n+1} + 2\sum_{i > n+1}\lambda_i\\
    & < \alpha\left( 1-\theta\sum_{i > n}\lambda_i \right) + (\alpha\theta-2\theta)\lambda_{n+1} + 2\sum_{i > n+1}\lambda_i\\
    & < \alpha - \alpha\theta\sum_{i > n}\lambda_i + (\alpha\theta - 2\delta)\lambda_{n+1} + 2\delta\sum_{i > n}\lambda_i\\
    & = \alpha - (\alpha\theta - 2\delta)\sum_{i > n+1}\lambda_i < \alpha - (\theta^2 - 2\delta) < \alpha
    \end{aligned}
    $$
    Hence $\sum_{i \geq 1}\lambda_i (g_i - \omega)(x) < \alpha\,\forall\,x \in B_{\leq 1}$
    
\end{itemize}
\end{proof}

\begin{rem}[\cite{13}]\label{Remark 2.8}

Given topological vector space $T$, according to {\cite[Theorem 5.11.4]{15}}, there exists a family of Banach spaces $(\{B_{\alpha}\}_, \|\cdot\|_{\alpha})_{\alpha\in\Omega}$ such that $T$ is linearly isomorphic to a subspace (not necessarily closed in the product topology generated by each norm in $B_{\alpha}$) of $\prod_{\alpha\in\Omega} B_{\alpha}$. Set $B = \prod_{\alpha\in\Omega}B_{\alpha}$. By the same theorem, if $T$ is complete, then its image in $B$ will be closed and hence weakly closed. Set $P_{\alpha}$ to be the canonical projection from $B$ to $B_{\alpha}$. Given a separable subset $X\subseteq B$, by \textbf{Tychonorff's Theorem}, $X$ is weakly compact iff $X_{\alpha} = P_{\alpha}X$ is weakly compact for each $\alpha\in\Omega$. Therefore, if $X$ is not weakly compact, there exists $\beta\in\Omega$ such that the weakly closure of $X_{\beta}$ is not weakly compact in $B_{\beta}$. Notice that the weak closure of $X_{\beta}$ is contained in $\overline{\operatorname{conv}(X_{\beta})}^{\|\cdot\|_{\beta}}$. If $X$ is separable, so is $X_{\beta}$. Hence $\overline{\operatorname{conv}(X_{\beta})}^{\|\cdot\|_{\beta}}$ is separable and so is the weak closure of $X_{\beta}$. 
    
\end{rem}

\begin{proof}[Proof of Theorem 2.3:]

According to \textbf{Remark \ref{Remark 2.8}} it suffices to prove the case where $X$ is a separable bounded weakly closed subset in a Banach space.\\

\begin{itemize}

    \item $1)\implies 2)$:
    Suppose $X$ is not weakly compact and let $Y = \overline{\operatorname{Span}X}$. Let $Q$ be the canonical mapping from $B$ to $B^{\ast\ast}$ and if $Y$ is reflexive, then $X$ will be weakly compact and hence we can find $F \in Y^{\ast\ast}_{\leq 1}$ that is not in $Q(Y)$. Assume $d[F, Q(Y)] > \delta > 0$\\

    \noindent
    Now find $\{x_n\}$ a countable dense subset of $X_{\leq 1}$. According to \textbf{Theorem \ref{Theorem 1.16}}, by setting $M = \dfrac{\delta}{d[F, Q(Y)]} < 1$, for each $n \in \mathbb{N}$ we could find $f_n \in Y^{\ast}_{\leq 1}$ such that:
    \begin{enumerate}
        \item $\|f_n\| < 1$
        \item $F(f_n) = \delta$
        \item $f_n(x_i) = 0\,\forall\,i \leq n$
    \end{enumerate}

    \noindent
    because by any $\{t_i\}_{i\leq n}\subseteq \mathbb{C}$:
    $$
    \delta = \left\vert\, F(f_n) + \sum_{i \leq n}t_if_i(x_i)\, \right\vert \leq M\left\| F+\sum_{i \leq n}t_i f_i \right\| 
    $$
    Now given for each $f \in \operatorname{conv}\{f_i\}$, suppose $f = \sum_{i \leq m}\alpha_i f_{n_i}$ where $n_1$ is the smallest among $\{n_1, n_2, \cdots, n_m\}$. Then:
    
    $$
    \delta  = \left\vert\,F(f) + f(x_{n_1})\,\right\vert \leq \|F\|\|f\|
    $$
    Therefore letting $\theta = \dfrac{\delta}{\|F\|}$ will prove $2)$.
    
    \item $2)\implies 3)$:
    The result follows mainly by \textbf{Lemma \ref{Lemma 2.7}}. Notice that $\alpha_n$ (obtained in the proof of \textbf{Lemma \ref{Lemma 2.7}}) is still non-decreasing. By scaling we could find $N > 0$ such that given $\|x\| \leq 1$ we will have $\,\vert f_n(x)\,\vert \leq N\,\forall\,n \in\mathbb{N}$. Because $X$ is bounded, then each $\alpha_n \leq N\sup_{x\in X}\|x\|$ and hence $\{\alpha_n\}$ will still converges.
    
    \item $3)\implies 4)$:
    Similar to the proof of \textbf{Theorem 2.2}, we will have $\|\sum_{i\geq 1}\lambda_i g_i\| = \alpha$ but $\sum_{i\geq 1}\lambda_ig_i$ never obtained its norm in the closed unit ball of $X$.
    
    \item $4)\implies 1)$:
    If $X$ is weakly compact, then because each $f \in B^*$ is continuous with respect to the weak topology, hence each $f$ can obtain $\sup_{x\in X}f(x)$ in $X$
    
\end{itemize}
\end{proof}

\begin{proof}[Proof of Theorem 2.4:]
$\hspace{0.44cm}\\$
\begin{itemize}
    \item $1)\implies2)$: First, according to \textbf{Remark \ref{Remark 2.8}}, it suffices to consider the case where $X$ is a bounded weakly closed subset of a Banach space. If $X$ is not weakly compact, then the result follows by the proof of \textbf{Theorem 2.2}.
    
    \item $2)\implies 3)\implies 4)\implies 1)$: The result follows by the proof of \textbf{Theorem 2.3}
    
\end{itemize}
\end{proof}

\section{Characterizations by affine sets}

This section will focus on characterizing reflexivity or non-reflexivity by studying weakly compact subsets in a Banach space or its dual. Different from the geometric method that is used in \textbf{Section 2}, in \textbf{Section 3.2}, the method to construct a weakly vanishing sequence of bounded linear functionals, whose closed convex hull is distant from zero by using \textbf{Helly's condition} (see \textbf{Theorem \ref{Theorem 1.16}}), a result that precisely describes when a finite set of hyperplanes in a Banach space has non-empty intersection. The method to prove weakly compactness in \textbf{Section 3.1} is still geometric and is mainly based on several separations theorems.

\subsection{Affine sets in a Banach space}

\begin{defn}\label{Definition 3.1}
Given a normed linear space $X$, call a subset $E$ \textbf{affine} iff $E-e$ is a proper subspace for any $e \in E$. Given any subset $A \subseteq X$, we use $\operatorname{lin}(A)$ to denote the smallest flat set that contains $A$.
\end{defn}

\begin{defn}
Given a normed linear space $X$, for any subset $A \subseteq X$ define $\|A\| = \inf_{a \in A}\|a\|$. 
\end{defn}

\begin{theorem}[{\cite[Theorem I.6.2 (Mazur's Theorem)]{4}}]\label{Theorem 3.3}
Given a locally convex topological vector space (in short LCTVS) $X$, let $K$ be a convex set with non-empty interior and $E$ a affine set
disjoint from $K$. Then there is a closed hyperplane $H$ such that $H\cap K=\emptyset$ but $E \subseteq H$. In other words, there is $f \in X^*$ and a real number such that $f(x) = c\,\forall\,x \in E$ and $f(y) > c$ for any $y\in \operatorname{int}(K)$.

\end{theorem}

\begin{prop}[{\cite[Chapter II, Theorem 4.1]{4}}]\label{Proposition 1.13}
Given $E$, a closed proper affine set that does not contain zero, in a normed linear space $X$, we can always find a hyperplane $H$ such that $\|H\| = \|E\|$.
\end{prop}

\begin{proof}
Since $\|E\| > 0$, the open ball that is centred at the origin and has radius $\|E\|$ is disjoint from $E$. Then, by \textbf{Theorem \ref{Theorem 3.3}}, there is a hyperplane $H$ that separates $E$ and that open ball. Therefore $\|E\|\geq\|H\|$. However, because $H$ is disjoint from that open ball with radius $\|E\|$ we then have $\|H\| \geq \|E\|$\\ 
\end{proof}

\begin{prop}[{\cite[Chapter II, Theorem 4.2]{4}}]\label{Proposition 1.14}
A real-valued function $f$ defined on a subset $A \subseteq X$ has an extension $F$ on $X$ with $\|F\|\leq M$ iff for some $M>0$, $\vert\,\sum_{i \leq n}t_i f(x_i)\,\vert \leq M\|\sum_{i \leq n}t_i x_i\|$ for any choice of real numbers $t_i$ and any choice of $x_i \in A$.
\end{prop}

\begin{proof}

The $(\implies)$ direction is obvious. On the other direction, we can define:

$$
F: \operatorname{Span}(A) \rightarrow\mathbb{R}, \hspace{0.3cm} \sum_{i\leq n}t_ix_i \mapsto \sum_{i\leq n}t_if(x_i)
$$
Since $F$ is a bounded linear functional on $\operatorname{Span}(A)$, $F$ can be continuously extended to be defined on $\overline{\operatorname{Span}(A)}$. Notice that the function $M\|\cdot\|$ is a semi-norm defined on $X$ and dominates $F$. Then by \textbf{Hahn-Banach Theorem}, $F$ can be extended to $X$ and satisfies $\|F\|\leq M$.

\end{proof}

\begin{prop}[{\cite[Chapter II, Corollary 4.1]{4}}]\label{Proposition 1.15}
Given a finite set $F_x = \{x_1, \cdots, x_n\} \subseteq X$ and a finite set $F_r = \{r_1, \cdots, r_n\} \subseteq \mathbb{R}$ there exists $f \in X^*$ such that $f(x_i) = r_i$ for all $1\leq i\leq n$ iff for some $M>0$, $\vert\,\sum_{i \leq n}t_i r_i\,\vert \leq M\|\sum_{i \leq n}t_i x_i\|$ for all choice of real numbers $t_i$
\end{prop}

\begin{proof}
Given $f\in X^{\ast}$ with $f(x_i) = r_i$ for each $i\leq n$, we will have $\vert\,f(x)\,\vert\leq \|f\|\|x\|$ for each $x\in \operatorname{Span}(F_x)$. On the other hand, if there exists $M>0$ such that the given inequality holds for all choice of scalars $t_i$, define $f$ on $F_x$ by $f(x_i) = r_i$ for each $i\leq n$. Then the conclusion follows by \textbf{Proposition \ref{Proposition 1.14}}.
\end{proof}

\begin{theorem}[{\cite[Chapter II, Theorem 4.3]{4}}][Helly's condition]\label{Theorem 1.16}
If $\{f_1, \cdots, f_n\}\subseteq X^*, M > 0, \{c_1, \cdots, c_n\}\subseteq \mathbb{R}$, then for each $\epsilon > 0$ there is $x \in X$ with $\|x\| < M+\epsilon$ such that $f_i(x) = c_i$ for all $i\leq n$ iff:

$$
\left\vert\,\sum_{i \leq n}t_i c_i \,\right\vert \leq M\|\sum_{i \leq n}t_i f_i\|
$$
for all choice of real numbers $t_i$. \\

\noindent
Equivalently, if we define $H_i = f_i^{-1}\{c_i\}$, $E = \bigcap_{i \leq n}H_i$ and:
$$
M_0 = \sup \left\{ \frac{\vert\,\sum_{i \leq n}t_i c_i \,\vert}{\|\sum_{i \leq n}t_i f_i\|} : \{t_i\}_{i \leq n} \subseteq \mathbb{C} \hspace{0.2cm}\&\hspace{0.2cm} \left\| \sum_{i \leq n}t_i f_i \right\| > 0 \right\}
$$
Then we claim $\|E\| = M_0$.
\end{theorem}

\begin{proof}$\hspace{0.44cm}\\$
\begin{enumerate}[label = \arabic*)]

    \item $(\Longrightarrow):$ For each $\epsilon>0$ and $x\in E$, we have that for any finite set of scalars $\{t_i\}_{i\leq n}$:

    $$
    \left\vert\, \sum_{i\leq n}t_i c_i \,\right\vert = \left\vert\, \big(\sum_{i\leq n}c_i f_i\big)(x) \,\right\vert \leq \|x\| \left\|\sum_{i\leq n}t_if_i \right\| \hspace{0.3cm} \implies\hspace{0.3cm} M_0\leq \|x\|
    $$
    If $\left\| \sum_{i\leq n}t_if_i \right\|>0$, since $x$ is arbitrarily piced from $E$, we then have $M_0\leq \|E\|$.
    
    \item $(\Longleftarrow)$: WLOG suppose the set $\{f_i\}$ is linear independent and hence for each $i\leq n$, $\bigcap_{j \leq n, j \neq i}\operatorname{Ker}f_j \backslash \operatorname{Ker}f_i$ is non-empty. For each $i\leq n$, pick $y_i \in \bigcap_{j \leq n, j \neq i}\operatorname{Ker}f_j \backslash \operatorname{Ker}f_i$. If we set $y = \sum_{i \leq n}c_i y_i$, we then have $f_i(y) = c_i$ for all $i\leq n$. Therefore $y \in E$. By \textbf{Proposition \ref{Proposition 1.13}}, we can find a hyperplane $H$ that contains $E$ such that $\|H\|=\|E\|$. Suppose $H = f^{-1}\{c\}$ for some $c \in \mathbb{C}$ and $f \in X^*$. Then:

    $$
    E-y = \bigcap_{i\leq n}\big( H_i - y\big) = \bigcap_{i\leq n}\operatorname{Ker}f_i \subseteq H-y = \operatorname{Ker}f
    $$
    which implies that $f = \sum_{i\leq n}r_if_i$ for a finite set of scalars $\{r_i\}_{i\leq n}$ and that, for each $x\in E$, $f(x) = f(y)$. Suppose $\{x_n\} \subseteq X_{=1}$ is a sequence such that $f(x_n)\rightarrow \|f\|$. Then $\dfrac{c}{f(x_n)}x_n\in H$ for each $n\in\mathbb{N}$ and:

    $$
    \forall\,n\in\mathbb{N}, \hspace{0.3cm} \|H\| \leq \left\| \frac{c}{f(x_n)}x_n \right\| \hspace{0.3cm}\implies\hspace{0.3cm} \|H\| \leq \lim_n \left\| \frac{c}{f(x_n)}x_n\right\|  = \frac{\vert\,c\,\vert}{\|f\|}
    $$
    Hence, we have $\|H\|= \dfrac{\vert\,c\,\vert}{\|f\|}$. Then:
    
    $$
    \|E\| = \|H\| = \frac{\vert\,c\,\vert}{\|f\|} = \frac{\vert\,\sum_{i \leq n}t_ic_i\,\vert}{\|\sum_{i \leq n}t_i f_i\|} \leq M_0
    $$
    Hence $\|E\|\leq M_0$, which implies $\|E\|=M_0$.
    
\end{enumerate}
\end{proof}

\begin{cor}[{\cite[Chapter II, Corollary 4.2]{4}}]\label{Corollary 1.17}
Given a normed linear space $X$ and $\phi\in X^{\ast\ast}$, let $\{f_i\}_{i \leq n}\subseteq X^*, H_i = \{x \in X\,\vert\,f_i(x) = \phi(f_i)\}, E = \bigcap_{i \leq n}H_i$, then we have $\|E\| \leq \|\phi\|$
\end{cor}

\begin{proof}
Immediately follows by the second part of \textbf{Theorem \ref{Theorem 1.16}}.
\end{proof}

\begin{cor}[{\cite[Chapter II, Corollary 4.3]{4}}]\label{Corollary 1.18}
Let $Q$ be the canonical mapping from $X$ to $X^{\ast\ast}$ and pick $\phi \in X^{\ast\ast}\backslash Q(X)$. Fix $c \in \mathbb{C}$ with and define $H_0 = \{f\in X^*\,\vert\,\phi(f) = c\}$. For any $M > 0$, if $M > \dfrac{\vert\,c\,\vert}{d[\phi, Q(X)]}$ then the zero linear functional in $X^*$ is in the weak-$\ast$ closure of $H_0\cap X^*_{< M}$
\end{cor}

\begin{proof}
It suffices to prove the case where $c\neq 0$. To show that $0$ is in the weak-$\ast$ closure of $H_0\cap X^{\ast}_{< M}$, it suffices to show that for any finite (linear independent) subset $\{x_i\}_{1\leq i\leq n}\subset X$ and $\epsilon\in(0, 1)$, there exists $h \in H_0\cap X^{\ast}_{< M}$ such that $\vert\, h(x_i)\,\vert < \epsilon$ for each $1\leq i\leq n$. Let $\{x_i\}_{1\leq i \leq n}$ be a linear independent subset of $X$. For each $1\leq i \leq n$ set $H_i = \operatorname{Ker}Q(x_i)$ and $E = \bigcap_{0\leq i \leq n}H_i$. Since the set of linear functionals $\big\{Q(x_1), \cdots, Q(x_n), \phi\big\}$ is linear independent, given $\{t_i\}_{0\leq i \leq n}\subseteq\mathbb{C}$ and $f\in E$, we have:

$$
\left\vert\, \left[ t_0\phi + \sum_{1\leq i \leq n}t_iQ(x_i)\right](f) \,\right\vert = \left\vert\, t_0\phi(f) + \sum_{1\leq i \leq n}t_i f(x_i) \,\right\vert = \vert\, t_0\phi(f)\,\vert
$$
Therefore by \textbf{Theorem \ref{Theorem 1.16}}, for each $y\in E$:

\begin{equation}\label{e4}
\begin{aligned}
\|E\|
& = \sup\left\{ \frac{\vert\, t_0\phi(f) \,\vert}{\| t_0\phi + \sum_{1\leq i\leq n}t_iQ(x_i) \|} : \{t_i\}_{0 \leq i \leq n}\subseteq\mathbb{C} \hspace{0.3cm}\&\hspace{0.3cm} \left\| t_0\phi + \sum_{1\leq i \leq n}t_iQ(x_i) \right\| > 0\right\}\\
& = \sup\left\{ \frac{\vert\, \phi(f) \,\vert}{\left\| \phi + \sum_{1\leq i\leq n} \dfrac{t_i}{t_0} Q(x_i) \right\|} : t_0 \in\mathbb{C} \backslash\{0\},\, \{t_i\}_{1\leq i \leq n}\subseteq\mathbb{C} \hspace{0.3cm}\&\hspace{0.3cm} \left\| t_0\phi + \sum_{1\leq i \leq n}t_iQ(x_i) \right\| > 0\right\}\\
& = \sup\left\{ \frac{\vert\, \phi(f) \,\vert}{\| \phi + \sum_{1\leq i\leq n}t_iQ(x_i) \|} : \{t_i\}_{1\leq i \leq n}\subseteq\mathbb{C} \right\} \leq \frac{\vert\,\phi(f)\,\vert}{d\big[ \phi, Q(X)\big]} = \frac{\vert\, c \,\vert}{d\big[ \phi, Q(X)\big]} < M
\end{aligned}
\end{equation}
Hence, there exists $h\in E$ with $\|h\|<M$ such that $h(x_i) = 0$ for each $i\leq n$.

\end{proof}

\begin{theorem}[{\cite[Chapter III, Theorem 2.1]{4}}]\label{Theorem 1.19}
In a normed linear space $X$ the following statements about a bounded subset $E \subseteq X$ are equivalent:
\begin{enumerate}[label = \arabic*)]
    \item $E$ is weakly compact
    \item $E$ is weakly sequentially compact
    \item $E$ is weakly countably compact
    \item For each $\{e_n\}_{n \in\mathbb{N}}\subseteq E$ there is $e \in E$ such that $\liminf_n f(e_n) \leq f(e) \leq \limsup_n f(e_n)$ for all $f\in X^*$.
    \item If $\{K_n\}$ is a decreasing sequence of closed convex sets in $X$ and $K_n\cap E\neq\emptyset$ for all $n\in\mathbb{N}$, then $\bigcap_{n \in \mathbb{N}}K_n\cap E\neq\emptyset$.
\end{enumerate}
\end{theorem}

\noindent
Below are several set-up results that are borrowed from \cite{4} and for proving \textbf{Theorem \ref{Theorem 1.19}}.

\begin{theorem}[{\cite[Theorem I.6.4 (Eidelheit Separation Theorem)]{4}}]\label{Theorem 3.11}
In a LCTVS $X$ let $K_1, K_2$ be convedx sets such that $K_1$ has an interior point and $K_2$ contains no interior points of $K_1$. Then there is a closed hyperplane $H$ separating $K_1$ from $K_2$; on the other hands, there is $f \in X^*$ such that $\sup_{k_2\in K_2}f(k_2) \leq \inf_{k_1 \in K_1}f(k_1)$
\end{theorem}

\begin{proof}
Let $K = K_1 - K_2$. Therefore $K_1$ has interior points and $0$ is disjoint from $\operatorname{int}(K)$. By \textbf{Theorem 3.3}, there is a non-zero $f \in X^*$ such that $f\vert_K \geq 0\,\implies\,f(k_1)> f(k_2)$ for any $k_1\in K_1$ and any $k_2\in K_2$. Now set $c = \inf_{k_1\in K_1}f(k_1)$ and let $H = f^{-1}\{c\}$. Then we have for each $k_1 \in \operatorname{int}(K_1), f(k_1) > c, f(k_2) \leq c$ for any $k_2\in K_2$.
\end{proof}

\begin{cor}[{\cite[Chapter I, Theorem 6.5]{4}}]\label{Corollary 1.22}
If $K$ is a closed convex set in a LCTVS $X$ and if $x \notin K$ then there exists $f\in X^*$ such that $f(x) > \sup_{k\in K}f(k)$.
\end{cor}

\begin{proof}
Take a convex neighborhood $K_1$ of $x$ and apply \textbf{Theorem 3.11} to $K_1$ and $K_2 = K$
\end{proof}

\begin{cor}[{\cite[Chapter I, Corollar 6.2]{4}}]\label{Corollary 1.23}
Given $E$ a subset of a LCTVS $X$, we have 
$$
\overline{\operatorname{conv}(E)} = \bigcap_{f \in X^*}\left\{ x \in X\,\vert\,f(x) \leq \sup_{y \in E}f(y) \right\}
$$
\end{cor}

\begin{defn}
Given a topological vector space $X$, say a subspace $\Gamma \leq X^*$ is total if for all $x\in X$, $f(x)=0$ for all $f\in \Gamma$ implies $x=0$. We use $\tau(\Gamma)$ to denote the weak topology generated by $\Gamma$.
\end{defn}

\begin{theorem}[{\cite[Theorem V.3.9]{3}}]\label{Theorem 1.25}
In a topological vector space $X$, let $\Gamma$ be a total subspace of linear functionals in $X^*$. Then all linear functionals continuous with respect to $\tau(\Gamma)$ belong to $\Gamma$
\end{theorem}

\begin{proof}
Clearly all functionals in $\Gamma$ are continuous with respect to $\tau(\Gamma)$. Now let $g\in X^*$ be a $\tau(\Gamma)$-continuous linear functional. Suppose for some finite set $\{f_i\}_{i \leq n}\subseteq \Gamma$ and $\epsilon\in (0, 1)$, we have:

$$
N(0, f_1, f_2, \cdots, f_n, \epsilon) = \{x\in X\,\vert\,f_i(x) < \epsilon\,\forall\,i \leq n\} \subseteq g^{-1}\Big( \big\{\lambda\in\mathbb{C}: \vert\, \lambda\,\vert \leq 1\big\} \Big)
$$
WLOG assume $\{f_i\}_{i \leq n}$ is linear independent. Pick $x_0 \in \bigcap_{i \leq n}\operatorname{Ker}f_i$ (by $X$ having infinite dimension). Therefore $\vert\,g(x_0)\,\vert < 1$ and also $\vert\,g(nx_0)\,\vert < 1$ for all $n\in\mathbb{N}$, which implies $g(x_0)=0$. Hence we have $\bigcap_{i \leq n}\operatorname{Ker}f_i \subseteq \operatorname{Ker}g$ and conclude $g$ is a linear combination of $\{f_i\}_{i \leq n}$
\end{proof}

\begin{proof}[Proof of \textbf{Theorem \ref{Theorem 1.19}}:]

The implications of $1)\,\implies\,2),2)\,\implies\,3), 3)\,\implies\,4)$ are immediate. Then we will work on $4)\,\implies\,5), 5)\,\implies\,1)$. Assume $4)$ holds. Then pick $x_n \in K_n \cap E$ and therefore obtain $\{x_n\}\subseteq E$. By \textbf{Corollary \ref{Corollary 1.23}} we have: 

$$
K_n = \bigcap_{f \in X^*}\left\{ x \in X\,\vert\,f(x) \leq \sup_{y \in K_n}f(y) \right\}$$
By \textbf{Corollary \ref{Corollary 1.23}} we can find $x \in E$ such that $\liminf_n f(x_n) \leq f(x) \leq \limsup_n f(x_n)$, for large $n \in \mathbb{N}$ we have $x \in K_n$ and hence $x \in \bigcap_{n \geq 1}K_n\cap E$.\\

\noindent
Then we will prove $5)\,\implies\,1)$. Let $Q$ be the canonical mapping from $X$ to $X^{\ast\ast}$. Assume by contradiction that there exists $\phi$ in the weak-$\ast$ closure of $Q(E)$ but not in $Q(E)$. Therefore for each $v \in E$ we can find $f_v\in X^*_{\leq 1}$ and $\epsilon_v > 0$ such that $\vert\,f_v(v)-\phi(f_v)\,\vert > \epsilon_v$. Then we have: 

$$
Q(E) \subseteq \bigcup_{v\in E_{\infty}}\Big\{\psi\in X^{\ast\ast}: \vert\, \psi(f_v) - \phi(f_v)\,\vert > \epsilon_v\Big\}
$$
which implies:

$$
Q(E) \subseteq \overline{Q(E)}^{w\ast} \subseteq \bigcup_{v\in E_{\infty}}\left\{\psi\in X^{\ast\ast}: \vert\, \psi(f_v) - \phi(f_v)\,\vert \geq \epsilon_v \right\} \subseteq \bigcup_{v\in E_{\infty}} \left\{\psi\in X^{\ast\ast}: \vert\, \psi(f_v) - \phi(f_v)\,\vert > \frac{\epsilon_v}{2}\right\}
$$
Since the weak-$\ast$ closure of $E$ is weak-$\ast$ compact,  without losing generality, assume:

$$
Q(E_{\infty}) \subseteq \bigcup_{k\leq n}\left\{\psi\in X^{\ast\ast}: \vert\, \psi(g_k) - \phi(g_k)\,\vert > \frac{\epsilon_k}{2} \right\}
$$
where $\{g_1, g_2, \cdots, g_N\}\subseteq \{f_v\,\vert\,v \in E_{\infty}\}$. Let $\epsilon_0 < \min_{k \leq N}\epsilon_k$ so that we have for each $v \in E$, there is $g\in \{g_i\}_{i\leq N}$ such that $\vert\,g(v)-\phi(g)\,\vert > \epsilon_0$. Let $l$ be smallest integer such that $\dfrac{1}{l} < \epsilon_0$.\\

\noindent
Let $L_n = \left\{f \in X^*_{\leq 1}\,\vert\,\phi(f) > \dfrac{n}{n+1}\|\phi\| \right\}$ and $L_{n+1}\subseteq L_n$ for each $n\in\mathbb{N}$. Again first pick an arbitrary $f_1 \in L_1$ and then find $x_1\in E$ such that $\vert\,\phi(f_1)-f_1(x_1)\,\vert < \dfrac{1}{2}$. Next pick an arbitrary $f_2 \in L_2$. There exists $x_2 \in E$ such that $\vert\,\phi(f_2)-f_2(x_2)\,\vert < \dfrac{1}{3}$ and $\vert\, \phi(f_1)-f_1(x_2)\,\vert < \dfrac{1}{3}$. Repeat steps above until $l$ ($l$ is given above with $\dfrac{1}{l}<\epsilon_0$). For each $k\leq l$, set $B_k = \{f_1, \cdots, f_k\}$ (where $f_i\in L_i$ for each $i\leq k$) and $S_k = \{x_1, \cdots, x_k\}$. Then we have, for each $k\leq l$:

$$
\forall\,i\leq j \leq k, \hspace{0.3cm} \big\vert\, \phi(f_i) - f_i(x_j)\, \big\vert < \frac{1}{j}
$$ 
Since $\phi$ is picked from the weak-$\ast$ closure of $Q(E)$, for each $1\leq i\leq N$, we can find $x_{l+i}\in E$ such that:

$$
\forall\,j\leq i, \hspace{0.3cm} \vert\, \phi(g_j) - g_j(x_{l+i})\,\big\vert < \frac{1}{l+i}
$$
and:

$$
\forall\,k\leq N, \hspace{0.3cm} \big\vert\, \phi(f_k) - f_k(x_{l+i})\,\big\vert < \frac{1}{l+i}
$$
For each $i\leq N$, set $f_{l+i} = g_i$, $B_{l+i} = \{f_1, f_2, \cdots, f_{l+i}\}$ and $S_{l+i} = \{x_1, x_2, \cdots, x_{l+i}\}$. For each $k>l+N$, given an arbitrary $f_k\in L_k$, again we can find $x_k\in E$ such that:

$$
\forall\,j\leq k, \hspace{0.3cm} \big\vert\, \phi(f_j) - f_j(x_k)\,\big\vert < \frac{1}{k}
$$
Then for each $k\in\mathbb{N}$, $B_k = \{x_1, \cdots, x_k\}$ is defined and each $B_k$ is contained in $E$. For each $j\in\mathbb{N}$, define $K_j = \overline{\operatorname{conv}\{x_j, x_{j+1}, \cdots\}}$. By \textbf{5)}, suppose $e\in E\cap \bigcap_{j\geq 1}K_j$. Hence we have:

$$
\forall\,f\in \bigcup_{k\in\mathbb{N}}B_k\, \forall\,n\in\mathbb{N},  \hspace{0.2cm} \vert\, \phi(f) - f(e)\,\vert < \frac{1}{n} \hspace{0.3cm}\implies\hspace{0.3cm} \forall\,f\in \bigcup_{k\in\mathbb{N}}B_k, \hspace{0.3cm} \phi(f) = f(e)
$$
However, since $e\in E$, there exists $g\in \{g_i\}_{i\leq N}$ such that $\vert\,\phi(g) - g(e)\,\vert > 0$ and meanwhile we have $\{g_i\}_{i\leq N} s\subseteq \bigcup_{k\in\mathbb{N}}B_k$. Hence, we obtain a contradiction and we can now conclude that $Q(E)$ is weak-$\ast$ close. Hence, since $E$ is bounded, $Q(E)$ is weak-$\ast$ compact, or $E$ is weakly compact.

\end{proof}

\subsection{Affine sets in the dual space}

\begin{prop}[{\cite[Theorem V.6.2]{3}}]\label{Proposition 3.1}
In a normed linear space, any countable family of norm-closed convex subset, which possesses finite intersection property, has a non-empty intersection if and only if any countable family of weakly closed subset, which possesses finite intersection property, has a non-empty intersection.
\end{prop}

\begin{proof}

$(\Longleftarrow):$ The result follows by the fact that, in a normed linear space, the weak closure of a convex set coincides with its normed closure.\\

\noindent    
$(\Longrightarrow):$ Suppose $\{F_n\}$ is a countable family of weakly-closed subset, which possesses finite intersection property, but has an empty intersection. Therefore, we have:

$$
\emptyset = \operatorname{conv}\left( \bigcap_{n\in\mathbb{N}}F_n \right) = \bigcap_{n\in\mathbb{N}} \operatorname{conv}(F_n) \hspace{0.3cm}\implies\hspace{0.3cm} \overline{\operatorname{conv}\left( \bigcap_{n\in\mathbb{N}}F_n \right)} = \overline{\bigcap_{n\in\mathbb{N}} \operatorname{conv}(F_n)} = \bigcap_{n\in\mathbb{N}} \overline{\operatorname{conv}(F_n)} = \emptyset
$$
Hence $\{\overline{\operatorname{conv}(F_n)}\}$ is a decreasing sequence of normed-closed convex set, which possesses finite intersection property but has empty intersection.

\end{proof}

\begin{theorem}[{\cite[Theorem 1]{17}}, {\cite[Theorem 7]{7}}]\label{Theorem 3.2}
Given a Banach space $X$ the following statements are equivalent:
\begin{enumerate}[label = \arabic*)]
    \item $X$ is not reflexive
    
    \item For each $\theta \in (0, 1)$ there is a sequence $\{x_i\}\subseteq X_{\leq 1}$ and $\{g_n\}\subseteq X^*_{< 1}$ such that:
    $$
    g_n(x_i) = 
    \begin{cases}
    \theta,\hspace{1cm} i \geq n\\
    0,\hspace{1cm} i < n
    \end{cases}
    $$
    
    \item There exists $\theta \in (0, 1)$ there is a sequence $\{x_i\}\subseteq X_{\leq 1}$ and $\{g_n\}\subseteq X^*_{< 1}$ such that:
    $$
    g_n(x_i) = 
    \begin{cases}
    \theta,\hspace{1cm} i \geq n\\
    0,\hspace{1cm} i < n
    \end{cases}
    $$
    
    \item For each $\theta \in (0, 1)$ there is a sequence $\{x_i\} \subseteq X_{\leq 1}$ such that for each $n\in\mathbb{N}$: 
    
    $$
    d[\operatorname{conv}\{x_i\}_{i \leq n}, \operatorname{conv}\{x_i\}_{i > n}] \geq \theta
    $$
    
    \item For some $\theta \in (0, 1)$ there is a sequence $\{x_i\}\subseteq X_{\leq 1}$ such that for each $n\in\mathbb{N}$: 
    
    $$
    d[\operatorname{conv}\{x_i\}_{i \leq n}, \operatorname{conv}\{x_i\}_{i > n}] \geq \theta
    $$
    
\end{enumerate}
\end{theorem}

\begin{proof}$\hspace{0.44cm}\\$
\begin{itemize}
    \item $1)\implies 2)$: Assume $X$ is not reflexive. Then by \textbf{Riesz's Lemma}, for each $\theta\in(0, 1)$ there exists $f\in X^{\ast\ast}$ with $\|f\|=1$ such that $d\big[ f, Q(X)\big] > \theta$. Hence, there exists $\phi\in X^{\ast\ast\ast}_{\leq 1}$ such that $Q(X) \subseteq \operatorname{Ker}\phi$ and $\phi(f) > \theta$. First, find $g_1\in X^{\ast\ast}_{< 1}$ such that $f(g_1) = \theta$. Then, for any $a_1\in\mathbb{C}$, we have:

    \begin{equation}\label{e11}
    \vert\,a_1\theta\,\vert = \vert\, f(a_1g_1)\,\vert \leq \|a_1g_1\|
    \end{equation}
    Hence, by \textbf{Theorem \ref{Theorem 1.19}}, there exists $x_1\in X_{\leq 1}$ with $g_1(x_1) = \theta$. Next, for any $a_1, a_2\in\mathbb{C}$, we have:

    \begin{equation}\label{e12}
    \vert\,a_2\theta\,\vert = \frac{\theta}{\phi(f)} \vert\,\phi(a_2 f)\,\vert = \frac{\theta}{\phi(f)} \Big\vert\, \phi\big[ a_2 f + a_1Q(x_1) \big] \,\Big\vert \leq \frac{\theta}{\phi(f)} \|a_2 f + a_1Q(x_1)\|
    \end{equation}
    Again by \textbf{Theorem \ref{Theorem 1.16}}, there exists $g_2\in X^{\ast}$ with $\|g_2\|<1$ (since $\phi(f)>\theta$) such that $f(g_2) = \theta$ and $g_2(x_1) = 0$. With $g_1, g_2 \in X^{\ast}_{< 1}$ and $x_1\in X_{\leq 1}$, for any $a_1, a_2\in\mathbb{C}$, replace $a_1 g_1$ in (\ref{e11}) by $a_1 g_1 + a_2 g_2$ and then we have:

    $$
    \vert\, a_1\theta + a_2\theta\,\vert = \vert\, f(a_1g_1 + a_2 g_2)\,\vert \leq \|a_1 g_1 + a_2 g_2\|
    $$
    By \textbf{Theorem \ref{Theorem 1.16}}, we can find $x_2 \in X_{\leq 1}$ such that $g_2(x_2) = g_1(x_2) = \theta$. Next, with $a_1, a_2, a_3\in\mathbb{C}$, replace $a_2 f + a_1Q(x_1)$ in (\ref{e12}) by $a_3 f + a_2 Q(x_2) + a_1Q(x_1)$. Then by \textbf{Theorem \ref{Theorem 1.16}}, we can then find $g_3\in X^{\ast}_{< 1}$ such that $f(g_3) = \theta$ and $g_3(x_1) = g_3(x_2) = 0$. By induction, we then can obtain the desired $\{g_n\}\subseteq X^{\ast}_{< 1}$ and $\{x_n\} \subseteq X_{\leq 1}$.

    \item $2)\implies 3)$: Immediate.
    
    \item $3)\implies 4)$: Let $\{x_n\}\subseteq X_{\leq 1}$ and $\{g_n\}\subseteq X^{\ast}_{< 1}$ be given as $g_n(x_i) = \theta$ for all $i\geq n$ and zero for all $i<n$. Now for any $x \in \operatorname{conv}\{x_1, x_2, \cdots, x_n\}, z \in \operatorname{conv}\{x_{n+1}, x_{n+2}, \cdots\}, \|z-x\| \geq \vert\,f_{n+1}(z-x)\,\vert = \theta$.

    \item $4)\implies 5)$: Immediate.
    
    \item $5)\implies 1)$: Let $C_n = \overline{\operatorname{conv}\{x_j\}_{j > n}}$. If $X$ is reflexive, the normed closed unit ball of $X$ is also weakly compact, which, by \textbf{Proposition \ref{Proposition 3.1}}, implies that $\bigcap_{n\in\mathbb{N}}C_n \neq\emptyset$. Let $y\in \bigcap_{n\in\mathbb{N}}C_n$. Then for each $n\in\mathbb{N}$, there exists $x_n \in C_n$ with $\|x_n - y\|< \dfrac{1}{n}$. Therefore, for each $n\in\mathbb{N}$:

    $$
    \|x_n - x_{n+1}\| \leq \|x_n-y\|+\|x_{n+1}-y\| < \frac{2}{n} \hspace{0.3cm}\implies\hspace{0.3cm} d\big[ \operatorname{conv}\{x_i\}_{i\leq n}, C_n\big] < \frac{2}{n} \overset{n\rightarrow\infty}{\longrightarrow} 0
    $$
    
\end{itemize}
\end{proof}

\section{Characterizations by Schauder basis}

This section includes results that describe reflexivity based on properties of a Schauder basis. It is clear that all subspaces of a reflexive Banach space are reflexive and, for instance, $\ell^1$, $c_0$ are rather famous non-reflexive Banach spaces. One way to characterize a non-reflexive Banach space is to show that there exists a subspace that is linearly isomorphic to $\ell^1$ or $c_0$. We will see how the existence of such a subspace is related to certain properties of a Schauder basis. Meanwhile when a subspace is linearly isomorphic to $\ell^1$ or $c_0$, a non-separable subspace will exist in the double dual. As a side result, checking separability of the double dual (see \textbf{Corollary \ref{Corollary 4.20}}) can also help with checking reflexivity.  On the other hand, since a Banach space is reflexive if and only if its dual is reflexive (see \textbf{Theorem \ref{Proposition 4.19}}), techniques introduced above can also be applied to determine when the dual space is reflexive. We will also see even in a Banach space with a Schauder basis, the set of linear functionals defined by giving the $n$-th coefficients is not always a Schauder basis of the dual space.

\subsection{Rosenthal's Condition}

\begin{defn}
In an infinite set $S$, given $X\subseteq S$ and a sequence of subsets $\{A_n\}_{n\in\mathbb{N}}$, we say $\{A_n\}$ \textbf{converges on $X$} if $\{\chi_{A_n}\}$ converges pointwise on $X$. In particular if $X=\emptyset$, we say $\{A_n\}$ \textbf{converges vaguely on} $X$. We call $\{A_n\}$ \textbf{Boolean independent} if for every pair of disjoint finite subsets $N_1, N_2 \subseteq\mathbb{N}$ with $\max N_1 < \min N_2$:

$$
\bigcap_{i\in N_1}\bigcap_{j\in N_2}A_i\backslash A_j \neq\emptyset
$$

\end{defn}

\begin{defn}

In an infinite set $S$, let $\{(A_n, B_n)\}_{n\in\mathbb{N}}$ be a sequence of pair of subsets of $S$ such that $A_n\cap B_n=\emptyset$ for each $n\in\mathbb{N}$. Given $X\subseteq S$, we say $\{(A_n, B_n)\}$ \textbf{converges on} $X$ if $\{A_n\}$ or $\{B_n\}$ converges on $X$. In particular, if $X=\emptyset$, we say $\{(A_n, B_n)\}_{n\in\mathbb{N}}$ \textbf{converges vaguely on} $X$. We call $\{(A_n, B_n)\}$ \textbf{Boolean independent} if for every pair of disjoint finite subsets $N_1, N_2\subseteq\mathbb{N}$ with $\max N_1 < \min N_2$:

$$
\bigcap_{i\in N_1}\bigcap_{j\in N_2} A_i\cap B_j \neq\emptyset
$$
Given two infinite subsets $M_1, M_2\subseteq \mathbb{N}$, we call $\{(A_n, B_n)\}_{n\in M_1}$ is a \textbf{subsequence of} $\{(A_n, B_n)\}_{n\in M_2}$ if $M_1\backslash M_2$ is finite.

\end{defn}

\begin{lem}\label{Lemma 4.3}

Given an infinite set $S$ and $\{ (A_n, B_n) \}_{n\in\mathbb{N}}$ a sequence of pair of subsetes of $S$ with $A_n\cap B_n=\emptyset$ for each $n$, suppose that for each $l\in\mathbb{N}$, there exists $\{X_i\}_{i\leq l}$ a $l$-size family of disjoint subsets of $S$ such that $\{(A_n, B_n)\}_{n\in\mathbb{N}}$ has no convergent subsequences on each $X_i$. Then there exists an infinite subset $M\subseteq\mathbb{N}$ and $K\in M$ such that $\{(A_n, B_n)\}_{n\in M}$ has no convergent subsequences on $X_i\cap A_K$ or $X_i\cap B_K$ for each $i\leq l$.

\end{lem}

\begin{proof}

We will complete the proof by induction. When $l = 1$, by assumption $\{(A_n, B_n)\}$ has no convergent subsequences on $S$. Next, assume by contradiction that for any infinite subset $M\subseteq\mathbb{N}$ and $K\in\mathbb{N}$, $\{(A_n, B_n)\}_{n\in M}$ has a convergent subsequence on $A_K$ or $B_K$. Start from $M_0 = \mathbb{N}$ and $n_0 = 1$. Then we can find an infinite subset $N_1\subseteq \mathbb{N}$ such that $\{(A_n, B_n)\}_{n\in N_1}$ converges on $A_1$ or $B_1$. Set $M_1 = N_1\cap M_0 = N_1$ and hence $\{(A_n, B_n)\}_{n\in M_1}$ converges on $A_1$ or $B_1$. Set $n_1 = \min M_1$. By assumption, we then can find an infinite subset $N_2\subseteq\mathbb{N}$ such that $N_2\backslash M_1$ is finite and $\{(A_n, B_n)\}_{n\in N_2}$ converges on $A_{n_2}$ or $B_{n_2}$. Set $M_2 = N_2\cap M_1$, $n_2 = \min M_2$ and hence $\{(A_n, B_n)\}_{n\in M_2}$ converges on $A_{n_2}$ or $B_{n_2}$, as well as on $A_{n_1}$ or $B_{n_1}$. By induction, we then can find an non-decreasing sequence of integers $\{n_i\}_{i\in\mathbb{N}}$ and $\{M_i\}_{i\in\mathbb{N}}$, an decreasing (in the sense of inclusion) sequence of infinite subsets of $\mathbb{N}$, such that for each $j\in\mathbb{N}$, $\{n_i\}_{i\geq j}\subseteq M_j$ and $\{(A_n, B_n)\}_{n\in M_j}$ converges on $A_i$ or $B_i$ for each $i\leq j$.\\

\noindent
Next, for each $i\in\mathbb{N}$, if $\{(A_n, B_n)\}_{n\in M_i}$ converges on $A_{n_i}$, set $a_i = 1$ and define $M_a = \{i\in\mathbb{N} \,\vert\, a_i = 1\}$; if $\{(A_n, B_n)\}_{n\in M_i}$ converges on $B_{n_i}$, set $b_i = 1$ and define $M_b = \{i\in\mathbb{N} \,\vert\, b_i = 1\}$. Since for each $i\in\mathbb{N}$, $\{(A_n, B_n)\}_{n\in M_i}$ converges on $A_{n_i}$ or $B_i$, we have at least one of $M_a$ or $M_b$ must infinite. Without losing generality, assume $M_a$ be infinite. For each $i\in M_a$ and $x\in A_{n_i}$:

\begin{itemize}

    \item If $\lim_{n\in M_i} \chi_{A_n}(x) = 0$, then there exists $n(i, x) \in\mathbb{N}$ with $n(i, x) \geq n_i$ such that for each $n\in M_i$, $x\notin A_n$ whenever $n\geq n(i, x)$.
    
    \item If $\lim_{n\in M_i} \chi_{A_n}(x) = 1$, then there exists $n(i, x) \in\mathbb{N}$ with $n(i, x) \geq n_i$ such that for each $n\in M_i$, $x\in A_n$ whenever $n\geq n(i, x)$. In this case, since $A_k\cap B_k=\emptyset$ for each $k\in\mathbb{N}$, we will have that $x\notin B_n$ for each $n\in N_i$ whenever $n\geq n(i, x)$
    
\end{itemize}

\noindent
According to our assumption, for any infinite subset $M\subseteq\mathbb{N}$, $\{(A_n, B_n)\}_{n\in M}$ is not convergent on $S$. In particular, for the subsequence $\{(A_{n_i}, B_{n_i})\}_{i\in M_a}$, there exists $y\in S$ such that all of the following sets are infinite.

$$
\{i\in M_a\,\vert\, y\in A_{n_i}\},\hspace{0.2cm} M_a\backslash \{i\in M_a\,\vert\, y\in A_{n_i}\},\hspace{0.2cm} \{i\in M_a\,\vert\, y\in B_{n_i}\},\hspace{0.2cm} M_a\backslash \{i\in M_a\,\vert\, y\in B_{n_i}\} 
$$
However, if we fix $i\in M_a$ such that $y\in A_{n_i}$, we then can find $n(i, y)\geq n_i$ such that either:

\begin{equation}\label{e13}
\forall\, k\in \big\{n\in M_i\,\vert\, n\geq n(i, y)\big\},\hspace{0.3cm} \chi_{A_k}(y) = 0
\end{equation}
or:

\begin{equation}\label{e14}
\forall\, k\in \big\{n\in M_i\,\vert\, n\geq n(i, y)\big\},\hspace{0.3cm} \chi_{B_k}(y) = 0
\end{equation}
Recall that for each $j\in\mathbb{N}$, $\{n_i\}_{i\geq j}\subseteq M_j$. If (\ref{e13}) is true, since the set $\{i\in M_a\,\vert\, y\in A_{n_i}\}$ is infinite, there exists $J\in \{i\in M_a\,\vert\, y\in A_{n_i}\}$ such that $n_J \geq n(i, y)$ and $y\in A_{n_J}$. Then, since $n_J\geq n(i, y) \geq n_i$, $n_J\in M_i$ and hence $y\notin A_{n_J}$ by (\ref{e13}), which is absurd. The case when (\ref{e14}) is true will lead to the same contradiction. Therefore, the base case is proved.\\

\noindent
The base case when $l = 1$ is proved and now by induction assume the statement holds when $l = n \in \mathbb{N}$. Then we will prove the case when $l = n+1$. Let $\{X_i\}_{i\leq n+1}$ be a $(n+1)$-size of disjoint subsets of $S$ such that $\{(A_n, B_n)\}_{n\in\mathbb{N}}$ has no convergent subsequences on each $X_i$. Then by assumption there exists an infinite $M\subseteq\mathbb{N}$ and $K\in M$ such that $\{(A_n, B_n)\}_{n\in M}$ has no convergent subsequences on $X_i\cap A_K$ or $X_i\cap B_K$ for each $i\leq n$. Assume by contradiction that $\{(A_n, B_n)\}_{n\in M}$ has a convergent subsequence on $X_{n+1}\cap A_K$ or $X_{n+1}\cap B_K$. Suppose $M'$ is the infinite set that make $\{(A_n, B_n)\}_{n\in M'}$ convergent on $A_K\cap X_{n+1}$ or $B_K\cap X_{n+1}$ and we can assume $K\in M'$. Let $K_1 = K$ and $M_1 = \{n\in M'\,\vert\, n>K_1\}$. $\{(A_n, B_n)\}_{n\in M_1}$ is a subsequence of $\{(A_n, B_n)\}_{n\in\mathbb{N}}$ and hence does not converge on each $X_i (i\leq n+1)$. By assumption there exists $M_2\subseteq M'$ and $K_2\in M_2$ such that $\{(A_n, B_n)\}_{n\in M_2}$ has no convergent subsequences on $X_i\cap A_{K_2}$ or $X_i\cap B_{K_2}$ for each $i\leq n$. If $\{(A_n, B_n)\}_{n\in M_2}$ has a convergent subsequence on $X_{n+1}\cap A_{K_2}$ or $X_{n+1}\cap B_{K_2}$, suppose $\{(A_n, B_n)\}_{n\in (M_2)'}$ is the desired subsequence and we can assume $K_2\in (M_2)'$.\\

\noindent
Assume the process above will not stop. Then for each $i\in\mathbb{N}$, we have $\{(A_n, B_n)\}_{n\in M_i}$ is convergent on $X_{n+1}\cap A_{K_i}$ or $X_{n+1}\cap B_{K_i}$. Then for each $i\in\mathbb{N}$, set $a_i=1$ if $\{(A_n, B_n)\}_{n\in M_i}$ converges on $X_{n+1}\cap A_{K_i}$; set $b_i=1$ if $\{(A_n, B_n)\}_{n\in M_i}$ converges on $X_{n+1}\cap B_{K_i}$. Again define $N_a = \{i\in\mathbb{N}\,\vert\, a_i=1\}$ and $N_b = \{i\in\mathbb{N}\,\vert\, b_i=1\}$. Without losing generality, assume $N_a$ is infinite. Hence $\{(A_{K_i}, B_{K_i})\}_{i\in N_a}$ is a subsequence and has no convergent subsequences on $X_{n+1}$. Then there exists $y\in X_{n+1}$ such that all of the following sets are infinite:

$$
\{i\in N_a\,\vert\, y\in A_{K_i}\}, \hspace{0.2cm} N_a\backslash \{i\in N_a\,\vert\, y\in A_{K_i}\}, \hspace{0.2cm} \{i\in N_a\,\vert\, y\in B_{K_i}\}, \hspace{0.2cm} N_a\backslash \{i\in N_a\,\vert\, y\in B_{K_i}\}, 
$$
Fix $j\in M_a$ such that $y\in A_{K_j}$. Since $\{(A_n, B_n)\}_{n\in M_j}$ is convergent on $X_{n+1}\cap A_{K_j}$ or $X_{n+1}\cap B_{K_j}$, we then can find $k(j, y)\geq K_i$ such that either:

\begin{equation}\label{e15}
\forall\, k\in\{ m\in M_j\,\vert\, m\geq k(j, y)\}, \hspace{0.3cm} \chi_{A_k}(y)=0
\end{equation}
or

\begin{equation}\label{e16}
\forall\, k\in\{ m\in M_j\,\vert\, m\geq k(j, y)\}, \hspace{0.3cm} \chi_{B_k}(y)=0
\end{equation}
Recall that $\{K_i\}_{i\geq j}\subseteq M_j$. If (\ref{e15}) if true, then there exists $J\in\{i\in N_a\,\vert\, y\in A_{K_i}\}$ such that $K_J\geq k(j, y)$ and $y\in A_{K_J}$. However, by (\ref{e15}) we also have $y\notin A_{K_J}$, which is absurd. The case when (\ref{e16}) is true can also lead to a contradiction similarly.

\end{proof}

\begin{lem}\label{Lemma 4.4}
Let $\{(A_n, B_n)\}_{n \in \mathbb{N}}$ be a sequence of pairs of sets where $A_n\cap B_n=\emptyset\,\forall\,n\in\mathbb{N}$ and suppose $\{(A_n, B_n)\}$ has no convergent subsequence on $S$. Then we can find an infinite set $M \subseteq\mathbb{N}$ so that $\{(A_n, B_n)\}_{n \in M}$ is Boolean Independent.
\end{lem}

\begin{proof}
According to the proof of base case in \textbf{Lemma \ref{Lemma 4.3}}, we can find  $M_1 \subseteq \mathbb{N}$ and $n_1\in M_1$ such that $\{(A_n, B_n)\}_{n\in M_1}$ has no convergent subsequence on $A_{n_1}$ or $B_{n_1}$. With $S$ replaced by $A_{n_1}$ of $B_{n_1}$ (depending on which set is the one where $\{(A_n, B_n)\}_{n\in M_1}$ has no convergent subsequence), apply the proof of the base case in \textbf{Lemma \ref{Lemma 4.3}} again, we then can find an infinite subset $M_2\subseteq M_1$ and $n_2\in M_2$ such that $\{(A_n, B_n)\}_{n\in M_2}$ has no convergent subsequence on each of the following sets:

$$
A_{n_1}\cap A_{n_2}, \hspace{0.2cm} A_{n_1}\cap B_{n_2}, \hspace{0.2cm} B_{n_1}\cap A_{n_2}, \hspace{0.2cm} B_{n_1}\cap B_{n_2}
$$
Repeat the method above and then obtain an infinite sequence of integers $\{n_i\}_{i\in\mathbb{N}}$ and $\{M_i\}_{i\in\mathbb{N}}$ a decreasing sequence of infinite subsets of $\mathbb{N}$. Assume by contradiction that $\{(A_n, B_n)\}_{n\in M}$ is not Boolean independent. Then there exists $F, G$ two disjoint finite subsets of $M$ with $\max F<\min G$ such that:

$$
\bigcap_{i\in F}\bigcap_{j\in G}A_i\cap B_j = \emptyset
$$
Suppose $F\subseteq\{n_1, \cdots, n_F\}$ and $G\subseteq \{n_F+1, \cdots, n_G\}$ for some $n_F, n_G\in\mathbb{N}$. Define $F_M=\{n_1, \cdots, n_F\}$ and $G_M = \{n_F+1, \cdots, n_G\}$. Consider the set $M_{n_G}$. Since $\{(A_n, B_n)\}_{n\in M_{n_G}}$ has no convergent sequences on $\bigcap_{i\in F_M}\bigcap_{j\in G_M}A_i\cap B_j$, there exists $y\in \bigcap_{i\in F_M}\bigcap_{j\in G_M}A_i\cap B_j$ such that all of the following sets are infinite:

$$
\{n\in M_{n_G}\,\vert\, y\in A_n\}, \hspace{0.2cm} M_{n_G}\backslash \{n\in M_{n_G}\,\vert\, y\in A_n\}, \hspace{0.2cm} \{n\in M_{n_G}\,\vert\, y\in B_n\}, \hspace{0.2cm} M_{n_G}\backslash \{n\in M_{n_G}\,\vert\, y\in B_n\}
$$
Hence:

$$
y\in \bigcap_{i\in F_M}\bigcap_{j\in G_M} A_i\cap B_j \subseteq \bigcap_{i\in F}\bigcap_{j\in G}A_i\cap B_j = \emptyset
$$
which is absurd.

\end{proof}

\begin{prop}\label{Proposition 4.5}
Let $\{f_n\}_{n \in \mathbb{N}}$ be a uniformly bounded sequence of real-valued scalar function defined on $S$. Fix $\delta, r > 0$ and for each $n\in\mathbb{N}$, define $A_n = \{x\in S\,\vert\,f_n(x) > \delta+r\}, B_n=\{x\in S\,\vert\,f_n(x) < r\}$. If there exists an infinite set $M\subseteq\mathbb{N}$ such that $\{(A_n, B_n)\}_{n \in M}$ is Boolean Independent, then $\{f_n\}_{n \in M}$ is equivalent, in the $\sup$ norm, to the standard basis of the real $\ell_1$ space.
\end{prop}

\begin{proof}
Without losing generality, we can let $\ell_1$ be indexed by $M$ since $M$ is infinite. Then it suffices to show that for any $v$ in the real $\ell_1$ (indexed by $M$):

$$
\frac{\delta}{2}\|v\|_1 \leq \sup_{x\in S}\left\vert\, \sum_{i\geq 1}v_i f_i(x)\,\right\vert \leq \sup_n\|f_n\| \|v\|_1
$$
The second inequality is immediate. Now fix $v$ in the real $\ell_1$ space with $\|v\|_1=1$. Define $V_+ = \{i\in M\,\vert\, v_i>0\}$ and $V_- = \{i\in M\,\vert\, v_i<0\}$. If we use $\mathcal{F}_+$ to denote the set of all finite subsets of $V_+$, $\mathcal{F}_-$ to denote the set of all finite subsets of $V_-$, we then have:

$$
\lim_{F\in\mathcal{F}_+, G\in\mathcal{F}_-} \sum_{i\in F\cup G}\vert\, v_i\,\vert = \lim_{F\in\mathcal{F}_+, G\in\mathcal{F}_-} \sum_{i\in F}v_i - \sum_{j\in G}v_i = \sum_{i\in M}\vert\,v_i\,\vert
$$
Fix $F\in\mathcal{F}_+$, $G\in\mathcal{F}_-$. By assumption suppose $x\in \bigcap_{i\in F}\bigcap_{j\in G}A_i\cap B_j$. Define $G_0 = \{j\in G\,\vert\, f_j(x)>0\}$. Then we have:

\begin{equation}\label{e17}
\begin{aligned}
&\hspace{1cm} \sum_{i\in G}v_i f_i(x)\geq \sum_{i\in G_0}v_i f_i(x) = \sum_{i\in G_0}\vert\,v_i\,\vert(-f_i(x)) \geq \sum_{i\in G_0} (-r)\vert\,v_i\,\vert \geq \sum_{i\in G}(-r)\vert\,v_i\,\vert\\
&\hspace{1cm} \sum_{i\in F}v_i f_i(x) \geq \sum_{i\in F}v_i (\delta+r)\\
&\implies\, \sum_{i\in F\cup G} v_i f_i(x) \geq \sum_{i\in F}v_i(\delta+r) + \sum_{i\in G}(-r)\vert\,v_i\,\vert
\end{aligned}
\end{equation}
Again by assumption, suppose $y\in \bigcap_{j\in G}\bigcap_{i\in F}A_j\cap B_i$. Then we have:

\begin{equation}\label{e18}
\begin{aligned}
&\hspace{1cm} \sum_{i\in G}v_i f_i(y) = \sum_{i\in G}\vert\,v_i\,\vert (-f_i(y)) \leq \sum_{i\in G}\vert\, v_i\,\vert(-(\delta+r))\\
&\hspace{1cm} \sum_{i\in F}v_i f_i(y) \leq r\sum_{i\in F}v_i\\
&\implies\, -\sum_{i\in F\cup G} v_if_i(y) \geq \sum_{i\in G}\vert\,v_i\,\vert(\delta+r) + \sum_{i\in F}v_i (-r)
\end{aligned}
\end{equation}
Combining (\ref{e17}) and (\ref{e18}) gives us:

$$
\sum_{i\in F\cup G}v_i f_i(x) + \left(-\sum_{i\in F\cup G}v_i f_i(y)\right) \geq \delta\sum_{i\in F\cup G}\vert\,v_i\,\vert
$$
If one of two sums on the left hand side is non-positive, then the other must be greater than or equal to $\delta\sum_{i\in F\cup G}\vert\,v_i\,\vert$. If both sums on the left hand side are positive, then at least one must be greater than or equal to $\dfrac{\delta}{2}\sum_{i\in F\cup G}\vert\,v_i\,\vert$. In general, we have:

$$
\sup_{x\in S}\left\vert\, \sum_{i\in M}v_if_i\,\right\vert \geq \frac{\delta}{2}\sum_{i\in F\cup G}\vert\,v_i\,\vert
$$
Since $F, G$ are arbitrarily picked from $\mathcal{F}_+$ and $\mathcal{F}_-$, we then can conclude:

$$
\sup_{x\in S}\left\vert\, \sum_{i\in M}v_if_i\,\right\vert \geq \lim_{F\in\mathcal{F}_+, G\in\mathcal{F}_-} \frac{\delta}{2}\sum_{i\in F\cup G}\vert\,v_i\,\vert = \frac{\delta}{2}
$$

\end{proof}

\begin{lem}\label{Lemma 4.6}
Given a sequence of uniformly bounded real-valued functions defined on $S$, say $\{f_n\}$, for each $M \subseteq \mathbb{N}$, define:
$$
\delta(M) = \sup_{x\in S}[\limsup_M f_m(x) - \liminf_M f_m(x)]
$$
Then there exists an infinite subset $Q\subseteq\mathbb{N}$ so that $\delta(L)=\delta(Q)$ for all infinite subsets $L\subseteq Q$.
\end{lem}

\begin{proof}
Given any two subsets $L, M\subseteq\mathbb{N}$, if $L\subseteq M$, we will have $\delta(L)\leq\delta(M)$. When $\delta(\mathbb{N}) = 0, \mathbb{N}$ will be the desired set in the conclusion so we from now on assume $\delta(\mathbb{N}) > 0$. Assume by contradiction any infinite subset $M\subseteq\mathbb{N}$ contain an infinite subset $L$ so that $\delta(L)<\delta(M)$. Then by transfinite induction, we can choose $\{N_{\alpha}\,\vert\,\alpha < \Omega\}$ a transfinite family of subsets of $\mathbb{N}$. Such family is indexed by the set of ordinals less than $\Omega$, the first uncountable ordinal, and has the property that for any $\alpha <\beta <\Omega$, $N_{\beta}\subseteq N_{\alpha}$ and $\delta(N_{\beta}) < \delta(N_{\alpha})$. Hence we obtain $\{\delta(N_{\alpha})\,\vert\, \alpha<\Omega\}$ a strictly decreasing transfinite family of positive real numbers, which will not exist since such family must be at most countable. 

\end{proof}

\begin{lem}\label{Lemma 4.7}
Let $\{f_i\}$ be a uniformly bounded sequence of real-valued functions defined on $S$. According to \textbf{Lemma \ref{Lemma 4.6}}, let $Q\subseteq\mathbb{N}$ be the set so that $\delta(Q) = \delta(N)$ for every infinite subset $N \subseteq Q$ and set $\lambda = \dfrac{\delta(Q)}{2}$. Show that there exists an infinite subset $M\subseteq Q$ and a rational number $r$ so that for every infinite subset $L\subseteq M$ there is $s \in S$ so that $\limsup_{i \in L}f_i(s) > \lambda+r$ and $\liminf_{i \in L}f_i(s) < r$
\end{lem}

\begin{proof}
Let $\{r_1, r_2, \cdots\}$ be the enumeration of $\mathbb{Q}$ and assume the conclusion is false. Then we start from $M=Q$ and $r_1$. Then we can find $L_1\subseteq M$ such that for each $s\in S$, the following is true for $r=r_1$:

\begin{equation}\label{e19}
\limsup_{i\in L_1}f_i(s)\leq \lambda+r \hspace{1cm}\textbf{or}\hspace{1cm} \liminf_{i\in L_1}f_i(s)\geq r
\end{equation}
By induction, we then have $\{L_k\}$ a decreasing sequence of infinite subsets of $Q$ such that for each $k\in\mathbb{N}$ and for each $s\in S$, the following is true for all $r\in\{r_i\}_{i\leq k}$:

$$
\limsup_{i\in L_k}f_i(s)\leq \lambda+r \hspace{1cm}\textbf{or}\hspace{1cm} \liminf_{i\in L_k}f_i(s)\geq r
$$
Fix $s\in S$. For each $k\in L_k$, pick $n_k\in L_k$ such that $f_{n_k}(s)\leq \lambda+r_k$ and $f_{n_k}(s) \geq r_k$. Define $L=\{n_k\}_{k\in\mathbb{N}}$. Since $L_0\subseteq L_k$ for each $k$, we then have:

$$
\limsup_{i\in L_0}f_i(s) \leq \limsup_{i\in L_k}f_i(s)\leq \lambda+r_k \hspace{1cm}\textbf{or}\hspace{1cm} \liminf_{i\in L_0}f_i(s)\geq \liminf_{i\in L_k}f_i(s)\geq r_k
$$
for each $k\in\mathbb{N}$. Since $s$ is arbitrarily picked, we then have for each $s\in S$ and each $r\in\mathbb{Q}$:

\begin{equation}\label{e20}
\limsup_{i\in L_0}f_i(s)\leq \lambda+r \hspace{1cm}\textbf{or}\hspace{1cm} \liminf_{i\in L_0}f_i(s)\geq r
\end{equation}
Since $\delta(L_0)=\delta(Q)=2\lambda$, if we let $\epsilon=\dfrac{\lambda}{2}$, we then can find $s_0\in S$ such that $\limsup_{i\in L_0}f_i(s_0) - \liminf_{i\in L_0} f_i(s_0) > 2\lambda-\epsilon$. Set $a = \limsup_{i\in L_0}f_i(s_0)$ and $b = \liminf_{i\in L_0} f_i(s_0)$. Hence $a > 2\lambda-\epsilon + b > b$. Now find $r_0\in\mathbb{Q}$ such that $r-b\in \left( 0, \dfrac{\lambda}{2} \right)$. Thus:

$$
b < r_0 < r_0  + \lambda = (r_0 - b) + (\lambda + b) < \frac{\lambda}{2} + \lambda + b = 2\lambda-\epsilon + b < a
$$
which implies that $a>\lambda+r_0$ and $b<r_0$, but this contradicts our assumption.

\end{proof}

\begin{defn}

A bounded sequence $(x_n)_{n\in \mathbb{N}}$ in a Banach space $(B, \|\cdot\|)$ is \textbf{equivalent to the standard basis of} $\ell^1$ if there exists $\delta>0$ such that for any finite choice of scalars $(c_i)_{i\leq n}$ (depending on if $B$ is an $\mathbb{R}$- or $\mathbb{C}$-vector space):

$$
\delta\sum_{i\leq n}\vert\, c_i\,\vert \leq \left\| \sum_{i\leq n} c_ix_i \right\|
$$
In this case, if we let $(e_n)_{n\in \mathbb{N}}$ be the standard basis of $\ell^1$, the mapping $x_n\mapsto e_n$ can be linearly and continuously extended to be a linear isomorphism between the closure of span of $(x_n)_{n\in \mathbb{N}}$ and $\ell^1$.
    
\end{defn}

\begin{theorem}[Rosenthal's Condition]\label{Theorem 4.9}

Let $S$ be an infinite set and $\{f_n\}_{n\in\mathbb{N}}$ be a uniformly bounded sequence of real-valued functions defined on $S$. Then there exists $\{f_n\}_{n\in M}$ a subsequence of $\{f_n\}$ that satisfies precisely one of the following alternatives:

\begin{enumerate}[label = \arabic*)]

    \item $\{f_n\}_{n\in M}$ converges pointwise
    \item $\{f_n\}_{n\in M}$ is equivalent to the standard basis of a real $\ell_1$ space
    
\end{enumerate}
    
\end{theorem}

\begin{proof}

Suppose $\{f_n\}_{n \in M}$ is a subsequence that does not converge pointwise. Hence $\delta(M)>0$ where $\delta$ is defined in \textbf{Lemma \ref{Lemma 4.6}}. Then by \textbf{Lemma \ref{Lemma 4.7}}, there exists $M_0\subseteq M$, $r\in\mathbb{Q}$ and $x\in S$ such that $\limsup_{i\in M_0}f_i(x) > \delta(M)+r$ and $\liminf_{i\in M_0} f_i(s) < \delta(M)$. For each $i\in M_0$, define $A_i = \{s\in S\,\vert\, f_i(s) > \delta(M)+r\}$ and $B_i = \{s\in S\,\vert\, f_i(s) < r\}$. Because of $x$, the following two sets are infinite:

$$
\{i\in M_0\,\vert\, f_i(x) > \delta(M)+r\}, \hspace{1cm} \{i\in M_0\,\vert\, f_i(x) < r\}
$$
which implies all of the following sets are infinite:

$$
\{i\in M_0\,\vert\, f_i(x) > \delta(M)+r\}, \hspace{0.2cm} M_0\backslash \{i\in M_0\,\vert\, f_i(x) > \delta(M)+r\}, \hspace{0.2cm} \{i\in M_0\,\vert\, f_i(x) <r \}, \hspace{0.2cm} M_0\backslash \{i\in M_0\,\vert\, f_i(x) < r\}
$$
Therefore $\{(A_n, B_n)\}_{n\in M_0}$ does not converge on $S$. By \textbf{Lemma \ref{Lemma 4.4}}, there exists $M_1\subseteq M_0$ such that $\{(A_n, B_n)\}_{n\in M_1}$ is Boolean independent. Hence by \textbf{Proposition \ref{Proposition 4.5}}, $\{f_i\}_{i\in M_1}$ is equivalent to the standard basis of a real $\ell_1$ space.

\end{proof}

\begin{theorem}[Rosenthal-Dor $\ell_1$ Theorem]\label{Theorem 4.10}
In a real Banach space $B$ any bounded sequence $\{x_n\}_{n \in \mathbb{N}}$ will have a subsequence $\{x_n\}_{n \in M}$ (for some infinite subset $M\subseteq \mathbb{N}$) that satisfy one of the following two mutually exclusive conditions:

\begin{enumerate}[label = \arabic*)]
    \item $\{x_n\}_{n \in M}$ is weakly Cauchy.
    \item $\{x_n\}_{n \in M}$ is equivalent to the standard basis of a real $\ell_1$ space.
\end{enumerate}
\end{theorem}

\begin{proof}

Fix a bounded sequence $\{x_n\}_{n \in \mathbb{N}} \subseteq B$ and let $Q$ be the canonical mapping from $B$ to $B^{\ast\ast}$. Since $B$ is a real Banach space, $B^*$ is the space of all real linear bounded functional and same for $B^{\ast\ast}$. Then for any infinite $M\subseteq\mathbb{N}$, $\{x_n\}_{n\in M}$ being weakly Cauchy is equivalent that $\{Q(x_n)\}_{n\in M}$ is pointwise Cauchy, or pointwise convergent. Hence the conclusion follows immediately by \textbf{Theorem \ref{Theorem 4.9}}.

\end{proof}

\begin{theorem}

Given a Banach space $B$, if $B$ is weakly complete, then $B$ is either reflexive or contains a subspace that is linearly isomorphic to $\ell^1$. If weak convergence in $B$ is equivalent to norm convergence (the \textbf{Schur property}), then every infinite-dimensional subspace of $B$ contains a subspace that is isomorphic to $\ell^1$.
    
\end{theorem}

\begin{proof}

If $B$ is reflexive, then all subspaces of $B$ are also reflexive. Hence if $B$ contains a subspace isomorphic to the real $\ell^1$ space, then $B$ cannot be reflexive. Otherwise, given that $B$ is weakly complete, by \textbf{Theorem 4.10}, we have that every bounded sequence has a weakly convergence subsequence. Then by \textbf{Theorem \ref{Theorem 1.19}}, the closed unit ball of $B$ is weakly compact. Together with {\cite[Goldstine's Theorem]{22}}, $B$ is reflexive.\\

\noindent
If $B$ has the Schur property, then by \textbf{Theorem \ref{Theorem 4.10}}, every bounded sequence $(x_n)_{n\in \mathbb{N}}$ either contains a convergent subsequence or a subsequence that is equivalent to the standard basis of $\ell^1$. If $V\leq B$ is an infinite-dimensional subspace, then the closed unit ball of $V$ cannot be compact and hence must contains a $(x_n)_{n\in \mathbb{N}}$ that has a subsequence that is equivalent to the standard basis of $\ell^1$.

\end{proof}

\subsection{Existence of isomorph of \texorpdfstring{$\ell_1$}{l1} or \texorpdfstring{$c_0$}{c0}}

\begin{prop}\label{Proposition 4.10}
Given a Banach space $B$ and a subspace $V \leq B$ that is isomorphic to $\ell_1$, there will be an $f \in B^*$ such that $f$ does not attain its $\sup$ on the unit sphere of $B$
\end{prop}

\begin{proof}
Suppose $\{x_n\}$ is a countable basis of $V$ that is isomorphic to the standard basis in $\ell_1$ and define $f_n \in V^*$ such that:
$$
f_n(x_i) = 
\begin{cases}1,\hspace{1.1cm} i \geq n\\
-1,\hspace{0.8cm} i < n
\end{cases}
$$
WLOG assume each $f_n$ is already continuously extended to the entire $V$. Now define $f$ on $V$ by $f(x) = \lim_n f_n(x)$ so that for each $x = \sum_{i \geq 1}a_nx_n, f(x) = -\sum_{i \geq 1}a_n$. Extend $f$ to the entire $B$ by letting $f$ equal to zero outside of $V$. Define: 

$$
g(x) = -f(x) + \sum_{i \geq 1}\frac{1}{2^i}f_i(x)
$$
Then for each $x \in V_{\leq 1}, \vert\,g(x)\,\vert \leq 2\|x\|$ and we also have:

$$
g(x_n) = -f(e_n) + \sum_{i \leq n}\frac{1}{2^i} - \sum_{i > n}\frac{1}{2^i} = 1 - \frac{1}{2^n} - \frac{1}{2^n} + 1 = 2 - \frac{1}{2^{n-1}} 
$$
Hence $\|g\| = 2$. Assume $\|v\| = 1$ and $g(v) = 2$. Because each $\vert\,f_k(v)\,\vert \leq 1$ we must have $f_k(v) = 1$ in order to let $g(v) = 2$ and hence $f(v) = 1\,\implies\,g(v) = 0$. Contradiction

\end{proof}

\begin{theorem}\label{Theorem 4.11}

Suppose a Banach space $B$ contains a isomorph of $c_0$ (i.e., a closed subspace of $B$ is linearly isomorphic to $c_0$). Then $B^*$ contains a isomoph of $\ell_1$

\end{theorem}

\begin{proof}
Let $M$ is a closed subspace isomorphic to $c_0$ and $\Phi$ be the linear isomorphism. By \textbf{Inverse Mapping Theorem} both $\Phi$ and $\Phi^{-1}$ are bounded. Let $\{x_n\}$ be the pre-image of the standard basis of $c_0$. For each $n\in\mathbb{N}$, define the following linear functional on $M$:

$$
f_n(x_i) = 
\begin{cases}
1, \hspace{1cm} i=n\\
0, \hspace{1cm} i\neq n
\end{cases}
$$
Next, define:

$$
V = \left\{ \sum_{i\geq 1}\alpha_i f_i: \sum_{i\geq 1}\vert\,\alpha_i\,\vert < \infty \right\}
$$
Let $\{e_n\}_{n\in\mathbb{N}}$ be the standard basis of $c_0$. Then we have for each $N\in\mathbb{N}$

$$
\left\| \sum_{n\leq N}x_n \right\| = \left\| \Phi^{-1}\left( \sum_{n\leq N}e_i \right) \right\| \leq \left\| \Phi^{-1}\right\| \left\| \sum_{n\leq N}e_n \right\|_{\infty} = \left\| \Phi^{-1} \right\|
$$
Hence for each $N\in\mathbb{N}$, fix $\{\alpha_i\}$ a absolutely summable sequence and let $\lambda_i$ be the sign of $\alpha_i$. Then:

$$
\begin{aligned}
&\hspace{1cm} \left\| \sum_{i\geq 1}\alpha_i f_i \right\| \geq \left\vert\, \sum_{i\geq 1}\alpha_i f_i \left( \frac{1}{\|\Phi^{-1}\|} \sum_{n\leq N}\lambda_n x_n \right) \, \right\vert = \frac{1}{\|\Phi^{-1}\|}\sum_{i\leq N}\big\vert\, \alpha_i \,\big\vert \\
&\implies\, \frac{1}{\|\Phi^{-1}\|}\sum_{i\geq 1}\big\vert\, \alpha_i \,\big\vert \leq \left\| \sum_{i\geq 1}\alpha_if_i \right\| \leq \sum_{i\geq 1}\big\vert\, \alpha_i\,\big\vert
\end{aligned}
$$
Then we can conclude that $\overline{V}$ is an isomorph of $\ell_1$.

\end{proof}

\subsection{Boundedly complete basis and shrinking basis}

\begin{theorem}[Banach's Condition]\label{Theorem 4.14}
In a Banach space $B$, given a linearly independent sequence $\{z_i\}_{i\in\mathbb{N}}$, $\{z_i\}_{i\in\mathbb{N}}$ is a Schauder basis of a (closed) subspace iff there is $K \geq 1$ such that:

$$
K \left\| \sum_{i \leq n+p}a_i z_i \right\| \geq \left\| \sum_{i \leq n}a_i z_i \right\|
$$
for every $n, p \in \mathbb{N}$ and $\{a_i\}_{i \leq n+p}\subseteq \mathbb{C}$.
\end{theorem}

\begin{proof}
 
$(\Longrightarrow):$ Given $\{z_i\}_{i\in\mathbb{N}}$ a Schauder basis of a closed subspace $V$, define:

$$
P_n:V\rightarrow V, \hspace{0.3cm} \sum_{i\geq 1}a_i z_i \mapsto \sum_{i\leq n}a_i z_i
$$
Since $\{z_i\}$ is a Schauder basis, each $P_n$ is continuous. For each $x\in V$ and each $n\in\mathbb{N}$, we have:

$$
\|P_n x\| = \left\| \left( \sum_{i\geq 1}a_i z_i \right) - \sum_{i>n}a_iz_i \right\| \leq \|x\| + \left\| \sum_{i>n}a_i z_i \right\| \overset{n\rightarrow\infty}{\longrightarrow} \|x\|
$$
Hence, for each $x\in V$, $\sup_n\| P_n x\|<\infty$. By \textbf{Uniform Boundedness Theorem}, we have $K = \sup_n\|P_n\| < \infty$. Given $m, n\in\mathbb{N}$ with $m<n$, we then have:

$$
\left\| \sum_{i\leq m}a_i z_i \right\| = \left\| P_m \left( \sum_{i\leq n}a_i z_i \right) \right\| \leq K \left\| \sum_{i\leq n}a_i z_i \right\|
$$
Notice that for each $n\in\mathbb{N}$:

$$
\left\| P_n \left( \frac{z_1}{\|z_1\|} \right) \right\| = \left\| \frac{z_1}{\|z_1\|} \right\| \leq \sup_n\|P_n\| = K
$$
Hence we have $K\geq 1$.\\

\noindent
$(\Longleftarrow):$ Consider $S = \operatorname{Span}(\{z_i\})$. Clearly each element in $S$ is a finite linear combination of elements from $(z_i)_{i\in \mathbb{N}}$ but without confusion we write each element in $S$ as an infinite sum $\sum_{i\in \mathbb{N}} a_iz_i$ where only finitely many $(a_i)_{i\in \mathbb{N}}$ are non-zero. Next, define:

$$
P_n:S\rightarrow S, \hspace{0.3cm} \sum_{i\geq 1}a_i z_i \mapsto \sum_{i\leq n}a_i z_i
$$
According to the assumption, given $m, n\in\mathbb{N}$ with $m>n$, we have:

$$
\left\| P_n \left( \sum_{i\leq m}a_i z_i \right) \right\| \leq K\left\| \sum_{i\leq m}a_i z_i \right\|
$$
Hence, we have that for each $s\in S$, $\|P_n(s)\|\leq K\|s\|$ and hence each $P_n$ can continuously extended to $\overline{S}$. Suppose $x\in\overline{S}$ and $\{s_k\}_{k\in\mathbb{N}} \subseteq S$ converges to $x$. We then have:

$$
\|x - P_n(x)\| \leq \|x - s_k\| + \|s_k - P_n(s_k)\| + \|P_n(s_k) - P_n(x)\|
$$
Since $P_n$ is continuous on $\overline{S}$ and, for each $s\in S$, $P_n(s)$ is eventually equal to $s$, we then can conclude that $P_n(x) \rightarrow x$ and hence $\{z_i\}$ is a Schauder basis of $\overline{S}$.

\end{proof}

\begin{cor}

In the set-up of \textbf{Theorem \ref{Theorem 4.14}}, $(z_i)_{i\in \mathbb{N}}$ is an unconditional Schauder basis of a subspace if and only if there exists $K\geq 1$ such that:

$$
K\left\| \sum_{i\in A}a_iz_i \right\| \geq \left\| \sum_{i\in B} a_iz_i \right\|
$$
for all finite subsets $B\subseteq A \subseteq \mathbb{N}$ and $(a_i)_{i\in A} \subseteq \mathbb{C}$.
    
\end{cor}

\begin{proof}

Put $S = \operatorname{Span}\big( \{z_i\}_{i\in \mathbb{N}} \big)$. For each finite subset $A\subseteq \mathbb{N}$, define:

$$
P_A: S\rightarrow S, \hspace{0.44cm} \sum_{i\geq 1}a_iz_i\mapsto \sum_{i\in A}a_iz_i
$$
Similar to the reasoning in the proof of \textbf{Theorem \ref{Theorem 4.14}}, $(z_i)_{i\in \mathbb{N}}$ is an unconditional Schauder basis of $\overline{S}$ if and only if the set of operators $\big( P_A: A\subseteq \mathbb{N}, \vert\, A\,\vert < \infty \big)$ is uniformly bounded.
    
\end{proof}

\noindent
Given a Schauder basis $(x_n)_{n\in \mathbb{N}}$ and a $(f_m)_{m\in \mathbb{N}} \subseteq B^*$ such that $f_m(x_n)=\delta_{m,n}$ for all $m, n\in \mathbb{N}$, {\cite[Theorem 3]{6}} provides necessary and sufficient condition for when $(f_m)_{m\in \mathbb{N}}$ is a basis of $B^*$. We will introduce this result later as it contains terminologies we only need in later subsections.

\begin{defn}[\cite{18}]

Given $\{x_i\}_{i\in\mathbb{N}}$ a Schauder basis of a Banach space $B$, $\{x_i\}_{i\in\mathbb{N}}$ is said to be \textbf{boundedly complete} if for any sequence of complex numbers $\{a_i\}_{i\in\mathbb{N}}$:

$$
\sup_n \left\| \sum_{i\leq n}a_i x_i \right\|<\infty \hspace{0.3cm}\implies\hspace{0.3cm} \sum_{i\geq 1}a_i x_i\textbf{  converges}
$$
For each $n\in\mathbb{N}$, let $V_n = \operatorname{Span}\{x_i\}_{i\geq n}$. For each $f\in X^*$, let:

$$
\|f\|_n = \sup_{n\in V_n}\frac{\vert\, f(x)\,\vert}{\|x\|}
$$
Then $\{x_i\}_{i\in\mathbb{N}}$ is said to be \textbf{shrinking} if for each $f\in B^*$, $\lim_n \|f\|_n = 0$.
    
\end{defn}

\begin{rem}\label{Remark 4.16}

Let $X, Y$ be two Banach spaces and $\Phi: X\rightarrow Y$ is a linear isomorphism. According to the \textbf{Inverse Mapping Theorem} and \textbf{Theorem \ref{Theorem 4.14}}, when $\{x_i\}_{i\in\mathbb{N}}\subseteq X$ is a Schauder basis of $X$, $\big\{ \Phi(x_i)\big\}_{i\in\mathbb{N}}$ is a Schauder basis of $Y$. Therefore the boundedly-completeness of a Schauder basis can be preserved by a linear isomorphism. Notice that for each $x\in X$ and $f\in X^*$, we have:

$$
\big\vert\, f\big[ \Phi(x) \big]\,\big\vert \leq \|\Phi\| \big\vert\, f(x)\,\big\vert
$$
Hence $\big\{ \Phi(x_i)\big\}$ is shrinking whenever $\{x_i\}$ is shrinking.
  
\end{rem}

\begin{theorem}[{\cite[Theorem 1]{6}}]\label{Theorem 4.15}
A Banach space $B$ with a Schauder basis $\{x_i\}_{i\in\mathbb{N}}$ is reflexive iff $\{x_i\}_{i\in\mathbb{N}}$ is both bounded complete and shrinking.
\end{theorem}

\begin{proof}

Suppose $B$ is reflexive. Define $u_n = \sum_{i \leq n}a_i x_i$ and suppose $\sup_n\|u_n\|\leq 1$. By reflexivity, without losing generality, assume $u_n$ (or a subsequence of $\{u_n\}$) converges to $u\in B_{\leq 1}$ weakly. Define $f_n$ to be the $n$-th coordinate coefficient of each element in $B$. Since $\{x_i\}$ is a Schauder basis, each $f_n$ is continuous. Suppose $u = \sum_{i\geq 1}\alpha_i x_i$. Then:

$$
\forall\,i\in\mathbb{N}, \hspace{0.3cm} a_i = \lim_n f_i(u_n) = f_i(u) = \alpha_i \hspace{1cm}\implies\hspace{1cm} u = \sum_{i\geq 1}a_i x_i
$$
Hence when $B$ is reflexive, $\{x_i\}$ is boundedly complete. Now suppose the basis is not shrinking when $B$ is reflexive. Then we can find $\epsilon > 0$, $g \in B^*$, an increasing sequence of integers $\{p_n\}$ and a sequence of $\{v_n\}\subseteq B_{\leq 1}$ such that for each $n\in\mathbb{N}$:

$$
v_n = \sum_{i\geq p_n}\alpha(n, i) x_i
$$
and $g(v_n) > \epsilon$. By reflexivity $\{v_n\}$ (or a subsequence of $\{v_n\}$) has a weak limit $v\in B_{\leq 1}$. Notice that for each $i\in\mathbb{N}$, $\lim_n f_i(v_n) = 0$. If $v\neq 0$, then $\lim_n f_i(v_n) = f_i(v) = 0$, which is absurd. If $v = 0$, then $\epsilon < \lim_n g(v_n) = g(v) = 0$, which is also absurd. Hence when $B$ is reflexive the basis is shrinking.\\

\noindent
Conversely, suppose the $\{x_i\}$ is both boundedly complete and shrinking. By \textbf{Theorem 4.12}, there exists $\epsilon\in (0, 1]$ such that for each $x = \sum_{i\geq 1}a_ix_i$ and $n, p\in\mathbb{N}$:

\begin{equation}\label{e21}
\begin{aligned}
&\hspace{1cm}\|x\| \geq \epsilon \left\| \sum_{i < n}a_i x_i \right\| \hspace{0.3cm}\textbf{and}\hspace{0.3cm} \|x\| \geq \epsilon  \left\| \sum_{i \leq n+p}a_i x_i \right\|\\
&\implies\, 2\|x\| \geq \epsilon \left\| \sum_{i \leq n+p}a_i x_i \right\| + \left\| \sum_{i < n}a_i x_i \right\| \geq \epsilon \left\| \sum_{i\leq n+p}a_ix_i - \sum_{i<n}a_i x_i \right\|\\
&\implies\, \left\| \sum_{n\leq i \leq n+p}a_i x_i \right\| \leq \frac{2}{\epsilon}\|x\|
\end{aligned}
\end{equation}
Now fix a sequence $\{y_n\}\subseteq B_{\leq 1}$. Assume $y_n = \sum_{i \geq 1}a(n, i)x_i$. By (\ref{e21}), we have for each $i\in\mathbb{N}$:

$$
\sup_n \|a(n, i)x_i\| \leq \sup_n \frac{2}{\epsilon} \|y_n\| = \frac{2}{\epsilon} \hspace{1cm}\implies\hspace{1cm} \sup_n \vert\, a(n, i)\,\vert \leq \frac{2}{\epsilon\|x_i\|}
$$
Hence for each $i$, let $\alpha_i$ be a clustered point of $\{a(n, i)\}_{n\in\mathbb{N}}$. By diagonal method, we can assume for each $i\in\mathbb{N}$, $\lim_n a(n, i) = \alpha_i$. Fix $N\in\mathbb{N}$. By (\ref{e21}), we have:

$$
\lim_n \left\| \sum_{i\leq N} a(n, i)x_i\right\| \leq \frac{\|y_n\|}{\epsilon} \leq \frac{1}{\epsilon} \hspace{1cm}\implies\hspace{1cm} \left\| \sum_{i\leq N}\alpha_i x_i\right\| \leq \frac{1}{\epsilon}
$$
Since the basis is bounded complete, $y = \sum_{i\geq 1}\alpha_i x_i$ exists and it remains to show that $y_n$ converges to $y$ weakly. Fix $f \in B^*_{\leq 1}$, $\gamma \in (0, 1)$ and $M \in \mathbb{N}$ such that $\|f\|_M < \dfrac{\epsilon\gamma}{8}$. Find $M_1\in\mathbb{N}$ so that for all $n\geq M_1$:

$$
\left\| \sum_{i < M}(\alpha_i - a(n, i)x_i \right\| < \frac{\gamma}{2\|f\|}
$$
By (\ref{e21}), we have:

$$
\lim_p \left\| \sum_{M\leq i\leq M+p}a_i x_i \right\| = \left\| \sum_{i\geq M}a_i x_i \right\| \leq \frac{2}{\epsilon}\|y\| = \frac{2}{\epsilon}
$$
and

$$
\left\| \sum_{i\geq M}\alpha_ix_i \right\| \leq \frac{2}{\epsilon}
$$
Therefore, for any $n \geq \max(M_1, M)$:
    
$$
\vert\,f(y - y_n)\,\vert = \left\vert\, f \left( \sum_{i < M}[\alpha_i - a(n, i)]x_i \right) \,\right\vert + \left\vert\, f \left[ \sum_{i \geq M}\alpha_i x_i - \sum_{i \geq M}a(n, i)x_i \right]\, \right\vert \leq \frac{\gamma}{2\|f\|}\|f\| + \|f\|_M(\frac{2}{\epsilon}+\frac{2}{\epsilon}) = \gamma
$$  
Hence $\lim_n f(y_n) = f(y)$.

\end{proof}

\subsection{Boundedly complete basis and isomorph of \texorpdfstring{$c_0$}{c0}}

\begin{defn}

In a Banach space $B$, a Schauder basis $(x_i)_{i\in \mathbb{N}}$ is \textbf{unconditional} if for any $x\in B$, the series form of $x$:

$$
x = \sum_{i\in \mathbb{N}} \lambda_ix_i
$$
converges unconditionally. 
    
\end{defn}

\begin{prop}[{\cite[Lemma 1]{6}}]\label{Proposition 4.16}
In $B$ a Banach space with a unconditional Schauder basis $\{x_i\}_{i\in\mathbb{N}}$, $\{x_i\}_{i\in\mathbb{N}}$ is not boundedly complete iff $B$ contains an isomorph of $c_0$.
\end{prop}

\begin{proof}

For each finite subset $F\subseteq\mathbb{N}$, define:

$$
P_F: B\rightarrow B, \hspace{0.3cm} x = \sum_{i\geq 1}a_i x_i \mapsto \sum_{i\in F}a_i x_i
$$
Because $\{x_i\}$ is a unconditional Schauder basis, similar to (\ref{e21}),there exists $\epsilon\in (0, 1]$ such that for each finite subset $F\subseteq\mathbb{N}$, we have: 

$$
\big\| P_Fx\big\| \leq \frac{2}{\epsilon}\|x\|
$$
For each $x\in X$, define:

$$
\|x\|_1 = \sup_{F\subseteq\mathbb{N}, \vert\,F\,\vert<\infty} \big\|P_Fx\big\|
$$
Obviously $\|\cdot\|_1$ is a well-defined norm in $B$ and we also have:

$$
\|x\| \leq \sup_n \left\| \sum_{i \leq n}a_i x_i \right\| \leq \|x\|_1 \leq \frac{2}{\epsilon}\|x\|
$$
which implies $\|\cdot\|$ is equivalent to $\|\cdot\|_1$.\\

\noindent
$(\Longrightarrow):$ First suppose $\{x_i\}$ is not boundedly complete and $\left\{\sum_{i \leq n}a_i x_i \right\}_{n\in\mathbb{N}}$ is a bounded sequence that does not converge. For each $n\in\mathbb{N}$ define $s_n= \sum_{i\leq n}a_i x_i$ and $M=\sup_n \|s_n\|$. Since $\{s_n\}$ does not converge with respect to $\|\cdot\|$, there exists $\delta>0$ such that:

$$
\limsup_{n, m}\|s_n - s_m\|>\delta 
$$
Then there exists a strictly increasing sequence $\{n_i\}\subseteq\mathbb{N}$ such that $\|s_{n_{i+2}} - s_{n_{i+1}}\| > \delta$ for each $i\in\mathbb{N}$. For each $i$, define $z_i = s_{n_{i+2}} - s_{n_{i+1}}$ and:

$$
V = \overline{\operatorname{Span}(\{z_i\}_{i\in\mathbb{N}})}^{\|\cdot\|_1}
$$
Notice that for any $N\in\mathbb{N}$ and any $G\subseteq \{1, 2, \cdots, N\}$:

$$
\left\| \sum_{i\in G}z_i \right\|_1 \leq \|s_{n_{N+2}}\|_1 \leq \frac{2M}{\epsilon}
$$
Fix $N\in\mathbb{N}$. First fix $k\in \{1, 2, \cdots, N\}$, a finite set $F \subseteq \{1, 2, \cdots, N\}\backslash \{k\}$, $\theta_k \in[0, 1]$ and $\{\beta_i\}_{i\in F}$ a set of complex numbers indexed by $F$. Then we have:

$$
\left\| \theta_k z_k + \sum_{i \in F}\beta_i z_i \right\| \leq (1-\theta_k) \left\| \sum_{i \in F}\beta_i z_i \right\|_1 + \theta_k \left\| z_k+\sum_{i \in F}\beta_i z_i \right\|_1
$$
Therefore the following function defined on $[0, 1]$:

$$
g_k: [0, 1]\rightarrow (0, \infty), \hspace{0.5cm} r\mapsto \left\|rz_k + \sum_{i\leq n, i\neq k} \beta_i z_i \right\|_1
$$
is a convex function (concave up). Because $g(0) \leq g(1)$, we have $g(\theta)\leq g(1)$ for each $\theta\in[0, 1]$. Similarly, with $z_k$ replaced by $-z_k$, the following function:

$$
g_{-k}: [0, 1]\rightarrow (0, \infty), \hspace{0.5cm} r\mapsto \left\|r(-z_k) + \sum_{i\leq n, i\neq k} \beta_i z_i \right\|_1
$$
is also convex. Hence, for each $r\in [-1, 0]$, we have:

$$
\begin{aligned}
&\hspace{0.46cm} \left\| rz_k + \sum_{i\leq n, i\neq k}\beta_i z_i \right\|_1 = \left\| (-r)(-z_k) + \sum_{i\leq n, i\neq k}\beta_i z_i \right\|_1\\
&\leq \left\| \left(\sum_{i\leq n, i\neq k}\beta_i z_i \right) - z_k \right\|_1 \leq \left\| \left(\sum_{i\leq n, i\neq k}\beta_i z_i \right) \right\|_1 + \|z_k\|_1\\
& \leq \left\| \sum_{i\leq n, i\neq k}\beta_i z_i \right\|_1 + \frac{2M}{\epsilon} \leq \left\| z_k + \sum_{i\leq n, i\neq k}\beta_i z_i \right\|_1 + \frac{2M}{\epsilon}
\end{aligned}
$$
Next fix $i, j\in \{1, 2, \cdots, N\}$ and $\theta_i, \theta_j\in (0, 1)$. Fix $G$ a finite subset in $\{1, 2, \cdots, N\}\backslash \{i, j\}$ and $\{\beta_k\}_{k\in F}$ a finite set of complex numbers indexed by $F$. Then:

$$
\begin{aligned}
\left\| \theta_i z_i + \theta_j z_j + \sum_{k\in F}\beta_k z_k \right\|_1 &\leq  (1-\theta_i)(1-\theta_j) \left\| \sum_{k \in F}\beta_k z_k \right\|_1 + (1-\theta_i)\theta_j \left\| z_j + \sum_{k \in F}\beta_k z_k \right\|_1 \\
& + \theta_i(1-\theta_j) \left\|z_i+\sum_{k \in F}\beta_k z_k \right\|_1 + \theta_i\theta_j \left\|z_i+z_j+\sum_{k \in F}\beta_k z_k \right\|_1
\end{aligned}
$$
Therefore the following function defined on $[0, 1]\times[0, 1]$:

$$
h_{i, j}: [0, 1]\times [0, 1] \rightarrow (0, \infty), \hspace{0.5cm} (r_1, r_2) \mapsto \left\| r_1 z_i + r_2 z_j + \sum_{k\in F}\beta_k z_k \right\|_1
$$
is bi-convex. Since all of $h_{i, j}(0, 0), h_{i, j}(0, 1)$ and $h_{i, j}(1, 0)$ are less than or equal to $h(1, 1)$, we then have for each $r_1, r_2\in[0, 1]$, $h_{i, j}(r_1, r_2)\leq h(1, 1)$. Similar to the previous case, given $r_1, r_2\in [-1, 1]$, if $r_1\leq 0, r_2\geq 0$, then:

\begin{equation}\label{e22}
\begin{aligned}
&\hspace{0.46cm} \left\| r_1 z_i + r_2 z_j + \sum_{k\in F}\beta_k z_k \right\|_1 \leq \left\| \left( z_j + \sum_{k\in F}\beta_k z_k \right) - z_i\right\|_1\\
&\leq \left\| z_j + \sum_{k\in F}\beta_k z_k \right\|_1 + \|z_i\|_1 \leq \left\| z_i + z_j + \sum_{k\in F}\beta_k z_k \right\|_1 + \frac{2M}{\epsilon}
\end{aligned}
\end{equation}
If both $r_1, r_2$ are non-positive, then:

\begin{equation}\label{e23}
\begin{aligned}
&\hspace{0.46cm} \left\| r_1 z_i + r_2 z_j + \sum_{k\in F}\beta_k z_k \right\|_1 \leq \left\| \left(\sum_{k\in F}\beta_k z_k \right) - z_i - z_j \right\|_1\\
& \leq \left\| \sum_{k\in F}\beta_k z_k \right\|_1 + \|z_i+z_j\|_1 \leq \left\| z_i + z_j + \sum_{k\in F}\beta_k z_k \right\|_1 + \frac{2M}{\epsilon}
\end{aligned}
\end{equation}
Combining (\ref{e22}) and (\ref{e23}) gives us:

$$
\forall\, (r_1, r_2)\in [-1, 1]\times[-1, 1], \hspace{0.3cm} \left\| r_1z_i + r_2 z_j + \sum_{k\in F}\beta_k z_k \right\|_1 \leq \left\| z_i + z_j + \sum_{k\in F}\beta_k z_k \right\|_1 + \frac{2M}{\epsilon}
$$
By induction, we then can conclude:

\begin{equation}\label{e24}
\forall\,(\theta_1, \theta_2, \cdots, \theta_N) \in [-1, 1]^N, \hspace{0.3cm} \left\| \sum_{i\leq N}\theta_i z_i \right\|_1 \leq \left\| \sum_{i\leq N}z_i \right\|_1 + \frac{2M}{\epsilon} \leq \frac{4M}{\epsilon}
\end{equation}
Since in $(V, \|\cdot\|_1)$, the sequence $\{z_i\}$ satisfies the condition in \textbf{Theorem \ref{Theorem 4.14}} (with respect to $\|\cdot\|_1$), then $\{z_i\}$ is a Schauder basis of $V$. Pick $v\in V$ and assume $v = \sum_{i\geq 1}\beta_i z_i$. By our assumption, we have $\|z_i\|\in (\delta, 2M)$ for each $i\in\mathbb{N}$ and $\|\beta_iz_i\| \leq \|v\|_1$. Hence $\beta = \sup_n\vert\,\beta_n\,\vert < \infty$. Define:

$$
\begin{aligned}
& S_{1, +} = \{i\in\mathbb{N}, i\leq N \,\vert\, \operatorname{Re}\beta_i \geq 0\}, \hspace{1cm} S_{1, -} = \{i\in\mathbb{N}, i\leq N \,\vert\, \operatorname{Re}\beta_i < 0\}\\
& S_{2, +} = \{i\in\mathbb{N}, i\leq N \,\vert\, \operatorname{Im}\beta_i \geq 0\}, \hspace{1cm} S_{2, -} = \{i\in\mathbb{N}, i\leq N \,\vert\, \operatorname{Im}\beta_i < 0\}
\end{aligned}
$$
and hence by (\ref{e24}):

$$
\begin{aligned}
&\hspace{0.46cm} \left\| \sum_{i \leq N}\frac{\beta_i}{\beta}z_i \right\|_1 \leq \left\| \sum_{i\leq N} \frac{\operatorname{Re}\beta_i}{\beta}z_i \right\|_1 + \left\| \sum_{i\leq N} \frac{\operatorname{Im}\beta_i}{\beta}z_i \right\|_1\\
&\leq \left\| \sum_{i\in S_{1, +}} \frac{\operatorname{Re}\beta_i}{\beta}z_i \right\|_1 + 
\left\| \sum_{i\in S_{1, -}} \frac{\operatorname{Re}\beta_i}{\beta}z_i \right\|_1 +
\left\| \sum_{i\in S_{2, +}} \frac{\operatorname{Im}\beta_i}{\beta}z_i \right\|_1 + 
\left\| \sum_{i\in S_{2, -}} \frac{\operatorname{Im}\beta_i}{\beta}z_i \right\|_1 \leq \frac{8M}{\epsilon}
\end{aligned}
$$
Since $N$ is arbitrarily fixed, now we can conclude that:

$$
\|v\|_1 = \lim_N \left\| \sum_{i\leq N}\beta_i z_i \right\|_1 \leq \beta\frac{8M}{\epsilon}
$$
Since for each $i\in\mathbb{N}$, $\|v\|_1 \geq \|\beta_iz_i\| > \delta\vert\,\beta_i\,\vert$, now we can conclude:

\begin{equation}\label{e25}
\delta\sup_n \vert\,\beta_i\,\vert < \left\| \sum_{i\geq 1}\beta_i z_i \right\|_1 \leq \frac{8M}{\epsilon}\sup_n\vert\,\beta_i\,\vert
\end{equation}
Now define: 

$$
F: V \rightarrow \ell_{\infty}, v=\sum_{i \geq 1}\beta_i z_i \rightarrow (\beta_1, \beta_2, \cdots)
$$ 
By (\ref{e25}) $F(V)$ is closed in $\ell_{\infty}$. Meanwhile for each $N\in\mathbb{N}$, we have:

$$
\sup_{n\geq N}\vert\,\beta_i\,\vert < \frac{1}{\delta} \left\| \sum_{i\geq N}\beta_i z_i\right\|_1 \overset{N\rightarrow\infty}{\longrightarrow} 0
$$
Hence $F(V)\subseteq c_0$. On the other hand, given $(\alpha_i)_{i\in\mathbb{N}}\in c_0$, for each $N\in\mathbb{N}$ and a finite subset $G\subseteq \{N, N+1, \cdots\}$, again by (\ref{e25}):

$$
\Big\| \sum_{i\in G}\alpha_i z_i\Big\|_1\leq \frac{8M}{\epsilon}\sup_{n\geq N}\vert\,\alpha_i\,\vert \overset{N\rightarrow\infty}{\longrightarrow} 0 \hspace{0.3cm}\implies\hspace{0.3cm} \left\{ \sum_{i\leq N}\alpha_i z_i \right\}_{N\in\mathbb{N}} \textbf{  converges  in  }\|\cdot\|_1
$$
Hence $F(V) = c_0$. Since $(V, \|\cdot\|_1)$ is linearly isomorphic to $(V, \|\cdot\|)$, we can now conclude $B$ contains an isomorph of $c_0$.\\

\noindent
$(\Longleftarrow):$ Conversely, suppose $S$ is a closed subspace of $B$, which is linearly isomorphic to $c_0$ through $\Phi: S\rightarrow c_0$. Then by \textbf{Inverse Mapping Theorem} both $\Phi$ and $\Phi^{-1}$ are bounded. Let $\{s_i\}_{i\in\mathbb{N}}\subseteq S$ be the pre-image of $\{e_i\}_{i\in\mathbb{N}}$ the standard basis of $c_0$. Then for each $N\in\mathbb{N}$:

$$
\left\| \sum_{i\leq N}s_i \right\| = \left\| \Phi^{-1} \left( \sum_{i\leq N}e_i \right) \right\| \leq \left\| \sum_{i\leq N}e_i \right\|_{\infty}\|\Phi^{-1}\| = \|\Phi^{-1}\|
$$
Meanwhile, for each $N\in\mathbb{N}$ and $F$ a finite subset of $\{N, N+1, \cdots\}$, we have:

$$
1 = \left\| \sum_{i\in F}e_i \right\|_{\infty} = \left\| \Phi \left( \sum_{i\in F}s_i \right)\right\|_{\infty} \leq \left\| \sum_{i\in F}s_i \right\| \|\Phi\| \hspace{0.3cm}\implies\hspace{0.3cm} \left\| \sum_{i\in F}s_i \right\| \geq \frac{1}{\|\Phi\|}.
$$
which implies that the sequence $\{\sum_{i\leq N}s_i\}_{N\in\mathbb{N}}$ does not converge in $\|\cdot\|_1$, nor in $\|\cdot\|$. Thus $\{s_i\}$ is not boundedly complete, which implies that $\{z_i\}$ is not either.

\end{proof}

\subsection{Shrinking basis and isomorph of \texorpdfstring{$\ell_1$}{l1}}

\begin{prop}[{\cite[Lemma 2]{6}}]\label{Proposition 4.17}
Let $B$ be a Banach space with a unconditional Schauder basis $\{z_i\}_{i\in\mathbb{N}}$. $\{z_i\}_{i\in\mathbb{N}}$ is not shrinking iff $B$ contains an isomorphs of $\ell_1$.
\end{prop}

\begin{proof}

$(\Longrightarrow):$ Assume by contradiction that there is $f \in B^*$ such that $\lim_n\|f\|_n > 0$. WLOG assume $\lim_n\|f\|_n = 1$. For each $n$ define $V_n = \operatorname{Span}\{z_i\}_{i \geq n}$. Hence there exists a strictly increasing sequence of integers $\{n_i\}$, $\{y_i\}\subseteq B$ where for each $i\in\mathbb{N}$, $y_i \in V_i$ and $\|y_i\|=1$, such that for each $i\in\mathbb{N}$, $f(y_i) > \dfrac{1}{2}$. Next define $Y_0 = \operatorname{Span}(\{y_i\}_{i\in\mathbb{N}})$ and in $Y$ define:

$$
\left\| \sum_{i\leq N}\alpha_i y_i \right\|_1 = \sum_{i\leq N}\vert\, \alpha_i\,\vert
$$
Clearly $\|\cdot\|_1$ defines a norm in $Y_0$ and for each $y\in Y_0$, $\|y\|\leq \|y\|_1$. Let $Y$ be the completion of $Y_0$ with respect to $\|\cdot\|_1$ and we have $Y \subseteq \overline{Y_0}^{\|\cdot\|}$. According to \textbf{Theorem \ref{Theorem 4.14}}, we have $\{y_i\}_{i\in\mathbb{N}}$ is a Schauder basis of $(Y, \|\cdot\|_1)$. Next, since in $\mathbb{R}^2$, $1$-norm and $2$-norm are equivalent, there exists $0 < r \leq R$ such that for any $(\alpha, \beta)\in\mathbb{R}^2$:

\begin{equation}\label{e26}
r\sqrt{\alpha^2 + \beta^2} \leq \vert\,\alpha\,\vert + \vert\,\beta\,\vert \leq R\sqrt{\alpha^2 + \beta^2}
\end{equation}
Since $\{y_i\}$ is a Schauder basis of $Y$, each $y\in Y$ can be uniquely written as $y = \sum_{i\geq 1}\alpha_i y_i$. Fix $y=\sum_{i\geq 1}\alpha_i y_i$ in $Y$ and define:

$$
\begin{aligned}
& S_{1, +} = \{i\in\mathbb{N} \,\vert\, \operatorname{Re}\alpha_i \geq 0\}, \hspace{1cm} S_{1, -} = \{i\in\mathbb{N} \,\vert\, \operatorname{Re}\alpha_i < 0\}\\
& S_{2, +} = \{i\in\mathbb{N} \,\vert\, \operatorname{Im}\alpha_i \geq 0\}, \hspace{1cm} S_{2, -} = \{i\in\mathbb{N} \,\vert\, \operatorname{Im}\alpha_i < 0\}
\end{aligned}
$$
Next, define:

$$
y_{1, +} = \sum_{i\in S_{1, +}} \operatorname{Re}(\alpha_i) y_i\hspace{1cm} y_{2, +} = \sum_{i\in S_{2, +}}\operatorname{Im}(\alpha_i) y_i \hspace{1cm} y_{1, -} = \sum_{i\in S_{1, -}} \operatorname{Re}(\alpha_i) y_i\hspace{1cm} y_{2, -} = \sum_{i\in S_{2, -}}\operatorname{Im}(\alpha_i) y_i
$$
Since all these four points converges with respect to $\|\cdot\|_1$, they also converge with respect to $\|\cdot\|$. Since each $y_i\in V_i$ and $\{z_i\}_{i\in\mathbb{N}}$ is a unconditional basis, according to \textbf{Theorem \ref{Theorem 4.14}}, there exists $\epsilon\in(0, 1]$ such that:

\begin{equation}\label{e27}
\max \big( \|y_{1, +}\|, \|y_{1, -}\|, \|y_{2, +}\|, \|y_{2, -}\| \big) \leq \frac{1}{\epsilon}\|y\|
\end{equation}
By assumption, we have $f(y_i)>\dfrac{1}{2}$ for each $i\in\mathbb{N}$. Therefore:

\begin{equation}\label{e28}
\begin{aligned}
& \frac{1}{2}\sum_{i\in S_{1, +}}\operatorname{Re}(\alpha_i) < f(y_{1, +})\leq \|f\|\|y_{1, +}\|\hspace{1cm} \frac{1}{2} \sum_{i\in S_{1, -}}-\operatorname{Re}(\alpha_i) < -f(y_{1, -}) = \vert\,f(y_{1, -})\,\vert \leq \|f\|\|y_{1, -}\| \\
& \frac{1}{2}\sum_{i\in S_{2, +}}\operatorname{Im}(\alpha_i) < f(y_{2, +})\leq \|f\|\|y_{2, +}\|\hspace{1cm} \frac{1}{2} \sum_{i\in S_{2, -}}-\operatorname{Im}(\alpha_i) < -f(y_{2, -}) = \vert\,f(y_{2, -})\,\vert \leq \|f\|\|y_{2, -}\| \\
\end{aligned}
\end{equation}
Hence combining (\ref{e26}), (\ref{e27}) and (\ref{e28}) gives us:

$$
\begin{aligned}
&\hspace{0.46cm} \|y\|_1 = \sum_{i\geq 1}\vert\,\alpha_i\,\vert = \sum_{i\geq 1} \sqrt{\big(\operatorname{Re}\alpha_i\big) ^2 + \big(\operatorname{Im}\alpha_i\big)^2} \leq \frac{1}{r}\sum_{i\geq 1}\big( \vert\, \operatorname{Re}\alpha_i\,\vert + \vert\,\operatorname{Im}\alpha_i\,\vert \big) \\
& = \frac{1}{r}\sum_{i\geq 1}\vert\,\operatorname{Re}\alpha_i\,\vert + \frac{1}{r} \sum_{i\geq 1}\vert\, \operatorname{Im}\alpha_i \,\vert = \frac{1}{r}\sum_{i\in S_{1, +}} \operatorname{Re}\alpha_i - \frac{1}{r}\sum_{i\in S_{1, -}}\operatorname{Re}\alpha_i + \frac{1}{r} \sum_{i\in S_{2, +}}\operatorname{Im}\alpha_i - \frac{1}{r} \sum_{i\in S_{2, -}} \operatorname{Im}\alpha_i\\
& \leq \frac{2\|f\|}{r}\big( \|y_{1, +}\| + \|y_{1, -}\| + \|y_{2, +}\| + \|y_{2, -}\| \big) \leq \frac{8\|f\|}{r\epsilon}\|y\|
\end{aligned}
$$
Since $y$ is arbitrarily picked from $Y$, we then can conclude for each $y\in Y$:

$$
\|y\| \leq \|y\|_1 \leq \frac{8\|f\|}{r\epsilon}\|y\|.
$$
Hence $\|\cdot\|$ is equivalent to $\|\cdot\|_1$, or $Y$ is linearly isomorphic to $\overline{Y_0}^{\|\cdot\|}$. Since $Y$ is obviously linear isomorphic to $\ell_1$, $\overline{Y_0}^{\|\cdot\|}$ is linearly isomorphic to $\ell_1$.\\

\noindent
$(\Longleftarrow):$ Conversely, suppose $S$, a closed subspace of $B$, is linearly isomorphic to $B$ and $\{s_i\}_{i\in\mathbb{N}}$ is the pre-image of the standard basis of $\ell_1$. By \textbf{Remark \ref{Remark 4.16}} $\{s_i\}$ is a Schauder basis of $S$. Since there exists $F\in B^*$ such that $F(s_i)=1$ for each $i\in\mathbb{N}$ and zero else, we could conclude that $\{s_i\}$ is not shrinking so that $\{z_i\}$ is not either.

\end{proof}

\begin{theorem}[{\cite[Theorem 1]{6}}]\label{Theorem 4.22}

Given a Banach space $B$ with a unconditional Schauder basis $\{z_i\}_{i\in\mathbb{N}}$, $B$ is non-reflexive iff $B$ contains an isomorph of $c_0$ or $\ell_1$.

\end{theorem}

\begin{proof}
Immediate by \textbf{Theorem \ref{Theorem 4.15}}, \textbf{Proposition \ref{Proposition 4.16}} and \textbf{Proposition \ref{Proposition 4.17}}.
\end{proof}

\subsection{Bases in the dual space}

\begin{prop}\label{Proposition 4.19}
A Banach space $B$ is reflexive iff $B^*$ is reflexive
\end{prop}

\begin{proof}
When $B$ is reflexive, then obviously $B^* = (B^{\ast\ast})^* = (B^*)^{\ast\ast}$. Now when $B^*$ is reflexive, we have $B^{\ast\ast}$ is reflexive. Hence the canonical embedding of $X$ into $X^{\ast\ast}$ as a closed subspace of $X^{\ast\ast}$ is also reflexive, which implies that $X$ is reflexive.
\end{proof}

\begin{cor}[\cite{6}]\label{Corollary 4.20}
Let $B$ be a Banach space with a unconditional basis. If $B^{\ast\ast}$ is separable, then $B$ is reflexive. 
\end{cor}

\begin{proof}

By \textbf{Theorem 4.\ref{Theorem 4.22}}, it suffices to show that $B$ contains no isomorphs of $c_0$ or $\ell_1$. If $B$ contains an isomorph of $c_0$, by \textbf{Theorem \ref{Theorem 4.11}}, $B^*$ contains an isomorph of $\ell_1$ and suppose $\{s_i\}_{i\in\mathbb{N}}$ is the pre-image of the standard basis of $\ell_1$. Then for each subset $M\subseteq\mathbb{N}$, there exists $f_M\in B^{\ast\ast}$ such that $f(e_i)=1$ for each $i\in M$ and zero elsewhere. Then $\{f_M\}_{M\in\mathcal{P}(\mathbb{N})}$ is a non-separable subset, which contradicts our assumption. When $B$ contains an isomorph of $\ell_1$, similarly we can define $\{f_M\}_{M\in\mathcal{P}(\mathbb{N})} \subseteq B^*$ as above. Then for each $M\subseteq\mathbb{N}$, there exists $F_M\in B^{\ast\ast}$ such that for each $N\subseteq\mathbb{N}$:

$$
F_M(f_N) = 
\begin{cases}
1, \hspace{1cm}N\subseteq M\\
0, \hspace{1cm}N\backslash M\neq\emptyset
\end{cases}
$$
and $F_M$ is zero outside of $\operatorname{Span}(\{f_M\}_{M\in \mathcal{P}(\mathbb{N})})$. Thus $\{F_M\}_{M\in \mathcal{P}(\mathbb{N})}$ is another non-separable set in $B^{\ast\ast}$, which is absurd.

\end{proof}

\begin{theorem}[{\cite[Theorem 3]{6}}]\label{Theorem 4.21}
Let $B$ be a Banach space with a Schauder basis $\{x_i\}_{i\in\mathbb{N}}$ and $f_n$ be the $n$-th coordinate coefficient linear functional. Then:

\begin{enumerate}[label = \alph*)]

    \item If $\{x_i\}$ is unconditional and shrinking, then $\{f_n\}$ is a unconditional Schauder basis for $B^*$.
    \item $\{f_n\}$ is a Schauder basis of $B^*$ iff $\{x_i\}$ is shrinking.
    \item $\{f_n\}$ is boundedly complete if it is a Schauder basis of $B^*$.
    
\end{enumerate}

\end{theorem}

\begin{proof}

Let $\epsilon\in (0, 1]$ be the constant provided by \textbf{Theorem \ref{Theorem 4.14}} such that for any $N, M\in\mathbb{N}$ with $N<M$:
        
$$
\left\| \sum_{i \leq M}\alpha_i x_i \right\| \geq \epsilon \left\| \sum_{i \leq N}\alpha_i x_i \right\|.
$$

\begin{enumerate}[label = \alph*)]
    
    \item Fix $\gamma \in (0, 1)$, $f \in B^*$ and find $N \in \mathbb{N}$ such that $\|f\|_N < \dfrac{\gamma}{1+K}$. Now let $\{x_{p_i}\}_{i \in \mathbb{N}}$ be a permutation of $\{x_i\}$. Then there exists $M \in \mathbb{N}$ such that $p_n \geq N$ iff $n > M$. Since $\{x_i\}$ is assumed unconditional, fix $x\in X$ with $\|x\|=1$ and we have:

    $$
    x = \sum_{i\geq 1}a_i x_i = \sum_{i\geq 1}a_{p_i}x_i
    $$
    Hence, for each $n\geq M$:
        
    $$
    \begin{aligned}
    &\hspace{1cm} \left\vert\,f(x) - \sum_{i \leq n}a_{p_i} f_{p_i}(x)\, \right\vert = \left\vert\,f \left( \sum_{i > n}a_{p_i} x_i \right)\, \right\vert \leq \|f\|_N \left\| \sum_{i > n}a_{p_i}x_i \right\|\\
    &\hspace{0.54cm} \leq \|f\|_N\left( \|x\| + \left\|\sum_{i \leq n}a_{p_i}x_i \right\| \right) \leq \|f\|_N(\|x\|+K\|x\|) < \gamma\\
    &\implies\, \left\| f-\sum_{i\leq n} a_{p_i}f_{p_i} \right\| < \gamma
    \end{aligned}
    $$
    since $x$ is arbitrarily fixed. 
    
    \item $(\Longrightarrow):$ For each $n\in\mathbb{N}$ let $V_n = \operatorname{Span}\{x_i\}_{i\geq n}$. Fix $f\in B^*$ and suppose $f = \sum_{i\geq 1} \alpha_i f_i$. Fix $N\in\mathbb{N}$. Then for each $k\in\mathbb{N}$, there exists $v_k\in V_N$ with $\|v_k\|=1$ such that:

    $$
    \|f\|_N - \frac{1}{k} \leq f(v_k) = \sum_{i\geq 1}\alpha_i f_i(v_k) = \sum_{i\geq N}\alpha_i f_i(v_k) \leq \left\| \sum_{i\geq N}\alpha_i f_i \right\| \overset{N\rightarrow\infty}{\longrightarrow} 0
    $$
    Hence $\{x_i\}$ is shrinking.\\

    \noindent
    $(\Longleftarrow):$ The proof to this direction is a special case of the proof to \textbf{a)} where the permutation is the identity.

    \item Fix a sequence of complex numbers $\{a_i\}$ such that $R = \sup_n\|\sum_{i \leq n}a_i f_i\| < \infty$. Then we want to show $\sum_{i \geq 1}a_i f_i$ converges. Now fix $x = \sum_{i \geq 1}\alpha_i x_i$, $N, M\in\mathbb{N}$ with $N<M$. Then:
    
    \begin{equation}\label{e29}
    \begin{aligned}
    &\hspace{0.46cm} \left\vert\, \sum_{N \leq i \leq M}a_i f_i(x) \,\right\vert 
    = \left\vert\, \sum_{N \leq i \leq M}a_i f_i \left( \sum_{N \leq i \leq M}\alpha_i x_i \right) \,\right\vert \\
    &\leq \left\| \sum_{N \leq i \leq M}a_i f_i\right\|\, \left\| \sum_{N \leq i \leq M}\alpha_i x_i \right\| < 2R \left\| \sum_{N \leq i \leq M}\alpha_i x_i \right\|
    \end{aligned}
    \end{equation}
    Since $\{x_i\}$ is a Schauder basis, according to the proof of \textbf{Proposition \ref{Proposition 4.16}}, the norm $\|\cdot\|_1$ (defined in the proof of \textbf{Proposition \ref{Proposition 4.16}}) is equivalent to $\|\cdot\|$. Hence we have:

    $$
    \left\vert\, \sum_{N\leq i \leq M}a_i f_i(x)\, \right\vert < 2R\left\| \sum_{N\leq i \leq M}\alpha_i x_i \right\| \leq 2R \left\| \sum_{i\geq N}\alpha_i x_i \right\|_1 \overset{N\rightarrow\infty}{\longrightarrow} 0
    $$
    Hence $\sum_{i \geq 1}a_i f_i(x)$ converges for any $x \in B$. Since $\{x_i\}$ is unconditional. by (\ref{e21}) (\ref{e29}) and the equivalence between $\|\cdot\|_1$ and $\|\cdot\|$, we have for each $M\in\mathbb{N}$:

    $$
    \left\vert\, \sum_{i\leq M}a_if_i(x)\,\right\vert < 2R\left\| \sum_{i\leq M}\alpha_i x_i \right\| \leq 2R\|x\|_1 \leq \frac{4R}{\epsilon}\|x\| \hspace{0.3cm}\implies\hspace{0.3cm} \left\vert\, \sum_{i\geq 1}a_i f_i(x)\, \right\vert \leq \frac{4R}{\epsilon}\|x\|
    $$
    Since $x$ is arbitrarily fixed, we then can conclude $\sum_{i\geq 1}a_i f_i\in B^*$ and by definition of each $f_i$, $\{a_i\}_{i\in\mathbb{N}}$ uniquely corresponds to $\sum_{i\geq 1}a_i f_i$.
        
\end{enumerate}
\end{proof}

\begin{theorem}[{\cite[Proposition 5]{18}}]

In the set-up of \textbf{Theorem \ref{Theorem 4.21}}, $(f_m)_{m\in \mathbb{N}}$ is shrinking if and only if $(x_n)_{n\in \mathbb{N}}$ is boundedly complete; $(f_m)_{m\in \mathbb{N}}$ is boundedly complete if and only if $(x_n)_{n\in \mathbb{N}}$ is shrinking.
    
\end{theorem}

\begin{proof}

The second part follows immediately by \textbf{Theorem \ref{Theorem 4.21}}. For the first part, we first let $Q$ denote the canonical mapping from $B$ to $B^{\ast\ast}$. Then again apply \textbf{Theorem \ref{Theorem 4.21}} to the bi-orthogonal system $(f_m)_{m\in \mathbb{N}}$ and $\big( Q(x_n) \big)_{n\in \mathbb{N}}$.

\end{proof}

\section{Characterizations by bi-orthogonal System}

In this section, we will study results mainly from \cite{15}, which are about a linearly independent sequence $(e_n)_{n\in\mathbb{N}}$ and a sequence of linear functionals $(f_n)_{n\in\mathbb{n}}$ where $f_n$ is defined to output the $n$-th coefficient of each element in $\operatorname{Span}(e_n)_{n\in\mathbb{N}}$. The idea of proofs in this section seem to be similar to the one in \textbf{Theorem \ref{Theorem 3.2}} or results in \textbf{Section 2}, but is not based on a certain affine set but the boundedness of $(e_n)_{n\in \mathbb{N}}$ and $(f_n)_{n\in \mathbb{N}}$. Throughout this section, given a Banach space $X$, we will use $Q$ to denote the canonical embedding from $X$ to $X^{\ast\ast}$, and $Q'$ to denote the canonical embedding from $X^*$ to $X^{\ast\ast\ast}$.

\begin{defn}\label{Definition 5.1}
Let $E$ be a Banach space. A \textbf{bi-orthogonal System} is a sequence $\{e_n\}\subseteq B$ and $\{f_n\}\subseteq B^*$ so that for each $i, j\in\mathbb{N}$:

$$
f_i(e_j)=
\begin{cases}
1, \hspace{1cm}i=j\\
0, \hspace{1cm}i\neq j
\end{cases}
$$
A bi-orthogonal System is bounded if both sequences are bounded in their own norms. In this section, whenever we are given $S = \{(e_i, f_j)\}_{i,j\in\mathbb{N}}$ a bi-orthogonal system, we define: 

$$
A(S) = \overline{\operatorname{Span}\{e_i\}}^{\|\cdot\|}, \hspace{1cm}
B_1(S) = \overline{\operatorname{Span}\{f_i\}}^{\|\cdot\|}, \hspace{1cm}
B(S) = \overline{\operatorname{Span}\{f_i\}}^{w\ast}
$$

\end{defn}

\begin{prop}\label{Proposition 5.2}
Given $X$ a non-reflexive Banach space, $Q$ the canonical mapping from $X$ to $X^{\ast\ast}, r \in X^{\ast\ast}\backslash Q(X)$ with $\|r\|\leq 1$ and $\sigma > 0$, show that for every finite sequence $\{f_1, f_2, \cdots, f_n\} \subseteq X^*$ and any $\sigma\in(0, 1)$, there exists $b \in X$ with $\|b\|\leq 1+\sigma$ such that $f_i(b) = r(f_i)$ for each $i\leq n$.
\end{prop}

\begin{proof}

Assume that $\{f_i\}_{i\leq n}\subseteq X^*$ is a finite set such that for any $b\in X_{\leq 1}$, there exists $i\leq n$ with $f_i(b)\neq r(f_i)$. Then in $\mathbb{C}^n$ define:  

$$
W = \big\{ (f_1(x), f_2(x), \cdots, f_n(x)\,\vert\,x \in X_{\leq 1} \big\}
$$
Obviously $W$ contains an open neighborhood of zero. Now let $z = (r(f_1), r(f_2), \cdots, r(f_n))$. By assumption and the fact that $\overline{W}$ is compact, $z \notin \overline{W}$. Since $\overline{W}$ compact we can find a bounded linear functional that strictly separates $\overline{W}$ and $z_0$. Namely we can find $(\lambda_1, \lambda_2, \cdots, \lambda_n)\in\mathbb{C}^n$ such that:

$$
\sum_{i \leq n}\lambda_i r(f_i) > 1 \hspace{1cm}\textbf{ 
 and  }\hspace{1cm} \forall\,(x_1, \cdots, x_n)\in \overline{W}, \hspace{0.3cm} \sum_{i\leq n}\lambda_i x_i \leq 1
$$
Therefore $\left\| \sum_{i\leq n}\lambda_i f_i \right\|\leq 1$ and hence:

$$
1 < r\left( \sum_{i\leq n}\lambda_i f_i \right) \leq \|r\| \left\| \sum_{i\leq n}\lambda_i f_i \right\| \leq 1
$$
which is absurd. Hence for any $\{f_i\}_{i\leq n}\subseteq X^*$ and $\sigma\in(0, 1)$:

$$
(r(f_1), \cdots, r(f_n)) \in  \overline{ \{ (f_1(x), \cdots, f_n(x)) \,\vert\, x\in X_{\leq 1}\big\}} \subseteq (1+\sigma)\big\{ (f_1(x), \cdots, f_n(x)) \,\vert\, x\in X_{\leq 1}\big\}
$$
or 

$$
(r(f_1), \cdots, r(f_n)) \in \big\{ (f_1(x), \cdots, f_n(x)) \,\vert\, x\in X_{\leq 1+\sigma}\big\}
$$

\end{proof}

\begin{prop}\label{Proposition 5.3}
Any non-reflexive Banach space $X$ will have a bounded bi-orthogonal system $(e_i, f_j)_{i, j\in\mathbb{N}}$ such that $\sup_n\|\sum_{i \leq n}e_i\| < \infty$ or $\sup_n\|\sum_{i \leq n}f_i\| < \infty$
\end{prop}

\begin{proof}

Let $r \in  X^{\ast\ast}\backslash Q(X)$ and suppose $H$ is an arbitrary finite dimensional subspace of $X$. Define: 

$$
\rho = \sup \big\{ \vert\,r(y)\,\vert: y \in X^*_{= 1}\cap H^{\perp} \big\}.
$$
Suppose $H = \operatorname{Span}\{h_1, h_2, \cdots, h_n\}$ and $H^* = \operatorname{Span}\{f_1, f_2, \cdots, f_n\}$ where $f_i(h_j) = 1$ iff $i = j$. If $\rho=0$, then $r\in (H^{\perp})^{\perp}$. Observe that $Q(H)^{\perp} = Q'(H^{\perp})$ and $(H^{\perp})^{\perp} \subseteq Q'(H^{\perp})_{\perp}$. Then $r\in \big[ Q(H)^{\perp}\big]_{\perp} = Q(H)\subseteq Q(X)$, which is absurd. Hence $\rho>0$. Below we will start constructing the desired bi-orthogonal system.\\

\noindent
Suppose $\|r\| = 1$. Fix $\sigma\in(0, 1)$. $y_1 \in X^*_{= 1}$ so that $\beta_1 = r(y_1) > \frac{1}{2}$. By \textbf{Proposition \ref{Proposition 5.2}} we can find $b_1\in X_{\leq 1+\sigma}$ so that $y_1(b_1) = \beta_1$. Define $E_1 = \operatorname{Span}\{b_1\}$ and: 

$$
\rho_1 = \sup\big\{ r(y)\,\vert\,y \in X^*_{= 1}\cap E_1^{\perp} \big\}
$$
By the previous remark, we have $\rho_1>0$. Next find $y_2 \in X^*_{\leq 1}\cap E_1^{\perp}$ so that $\beta_2 = r(y_2) > \frac{1}{2}\rho_1$. Again by \textbf{Proposition \ref{Proposition 5.2}}, find $b_2 \in X_{\leq 1+\sigma}$ so that $y_1(b_2) = \beta_1, y_2(b_2) = \beta_2$. Define $E_2 = \operatorname{Span}\{b_1, b_2\}$ and:

$$
\rho_2 = \sup \big\{r(y)\,\vert\,y \in X^*_{= 1}\cap E_2^{\perp} \big\}
$$
and similarly $\rho_2>0$. By induction for each $n \in \mathbb{N}$ we will have $\{b_i\}_{i\leq n} \subseteq X_{\leq 1+\sigma}$, $E_n = \operatorname{Span}\{b_i\}_{i\leq n}$, $y_{i+1}\in X^{\ast}_{\leq 1}\cap E_i^{\perp}$, $\{\beta_i\}_{i\leq n}$ and:

$$
\rho_n = \sup\big\{ r(y)\,\vert\, y\in X^{\ast}_{= 1}\cap E_n^{\perp}\big\}
$$
such that $\beta_1 = r(y_1) >\dfrac{1}{2}$, $\beta_{n+1} = r(y_{n+1}) \in \left(\dfrac{1}{2}\rho_n, \rho_n \right]$ and for each $i, j\in\{1, 2, \cdots, n\}$:
    
$$ 
y_i(b_j) = 
\begin{cases}r(y_i),\hspace{0.5cm} 1 \leq i \leq j \leq n\\
0,\hspace{1.06cm} 1 \leq j < i \leq n
\end{cases}
$$
Since $H_n^{\perp}$ is decreasing as $n$ increases, we have $\{\rho_n\}$ is non-increasing sequence in $(0, 1]$ and hence convergent. Assume that $\inf_n\rho_n = \lim_n\rho_n = 0$. For each $\epsilon\in(0, 1)$ and then suppose $\rho_n < \epsilon$ for all $n\geq N$. Fix $f\in X^{\ast}_{\leq 1}$ and for each $n\in\mathbb{N}$, define $\tau_n(f) = \max_{i\leq n}\vert\, f(b_i)\,\vert$ and:
    
$$
z_n(f) = \frac{1}{\beta_1}f(b_1)y_1+\sum_{1 < i \leq n}\frac{1}{\beta_i} \big[ f(b_i)-f(b_{i-1}) \big]y_i
$$
Since for each $n\in\mathbb{N}$:

$$
0 < \frac{1}{2}\rho_{n+1} < \beta_n \leq \rho_n
$$
we then have $\|z_n(f)\| \leq \dfrac{4\tau_n(f)}{\rho_n}$
and $r\big[z_n(f)\big] = f(b_n)$. Hence:

$$
\left\vert\, \rho_n r \big[ f - z_n(f) \big]\, \right\vert \leq \rho_n \| f - z_n(f) \| \leq \left[ 1+\dfrac{4n\tau_n(f)}{\rho_n} \right]\rho_n = \rho_n+4n\tau_n(f)
$$
Hence, for each $n\geq N$, whenever $\tau_n(f)< \dfrac{\epsilon}{4n}$, we then have:

$$
\big\vert\, \rho_n r(f) \,\big\vert \leq \Big\vert\, \rho_n r\big[ z_n(f)\big] \,\Big\vert +\rho_n + 4n\tau_n(f) < \rho_n\epsilon + \rho_n + \epsilon < 3\epsilon
$$
Hence:

\begin{equation}\label{e30}
\Big\{ f\in X^*: \max_{i\leq N} \vert\, f(b_i)\,\vert <\epsilon\Big\} \subseteq \Big\{ f\in X^*: \vert\,r(f)\,\vert < \epsilon\Big\}
\end{equation}
Since we assume that $\lim_n\rho_n=0$, for each $\epsilon\in(0, 1)$, we can then find $N\in\mathbb{N}$ so that (\ref{e30}) is true, which implies that $r$ is continuous with respect to the weak topology (on $X^*$) generated by $Q(X)$. Since $Q(X)$ separates points on $X^*$, by \textbf{Theorem \ref{Theorem 1.25}} we have $r\in Q(X)$ and obtain a contradiction. Hence there exists $\beta\in(0, 1)$ such that $\inf_n\rho_n > \beta$ and $\beta_i>\beta$ for each $i\in\mathbb{N}$.\\

\noindent
Next, define $e_1 = b_1$ and for each $i\in\mathbb{N}$ with $i>1$:

$$
e_i = b_i - b_{i-1}.
$$
For each $i\in\mathbb{N}$, define $f_i = \dfrac{1}{\beta}y_i$. Hence for each $i\in\mathbb{N}$, $\|e_i\|\leq 2(1+\sigma)$, $\|f_i\|\leq \frac{1}{\beta}$ and for each $n\in\mathbb{N}$:

$$
\left\| \sum_{i\leq n}e_i \right\| = \|b_n\|\leq 1+\sigma
$$
Then $\{(e_i, f_j)\}_{i, j\in\mathbb{N}}$ is a bounded bi-orthogonal system where $\sup_n\big\|\sum_{i\leq n}e_i\big\| < \infty$. On the other hand, for each $i\in\mathbb{N}$, define:

$$
h_i = \frac{1}{\beta_i}y_i - \frac{1}{\beta_{i+1}}y_{i+1}.
$$
Hence for each $i\in\mathbb{N}$, $\|h_i\|\leq \dfrac{2}{\beta}$ and for each $n\in\mathbb{N}$:

$$
\left\| \sum_{i\leq n}h_i \right\| = \left\| \frac{1}{\beta_1}y_1 - \frac{1}{\beta_{n+1}}y_{n+1} \right\| \leq \frac{2}{\beta}.
$$
In this case we have $\{(b_i, h_j)\}_{i, j\in\mathbb{N}}$ a bounded bi-orthogonal system  where $\sup_n\big\| \sum_{i\leq n}h_i\big\| < \infty$.

\end{proof}

\begin{prop}\label{Proposition 5.4}

Given a doubly indexed bounded sequence $\{a_{m, n}\}_{m, n\in\mathbb{N}} \subseteq\mathbb{R}$, define $c_{1, n} = a_{1, n}$ and for each $m\in\mathbb{N}$ with $m>1$, define  $c_{m, n} = a_{m, n} - a_{m-1, n}$. If, for each $n\in\mathbb{N}$, $\{a_{m, n}\}_{m\in\mathbb{N}}$ is monotonely increasing and, for each $m\in\mathbb{N}$, $\{c_{m, n}\}_{n\in\mathbb{N}}$ is also monotonely increasing, then:

$$
\lim_m \lim_n a_{m, n} = \lim_n \lim_m a_{m, n}
$$

\end{prop}

\begin{proof}

By assumption and \textbf{Monotone Convergence Theorem}:

$$
\lim_n \sum_{m\in\mathbb{N}} c_{m, n} = \sum_{m\in\mathbb{N}} \lim_n c_{m, n}
$$
Hence:

$$
\lim_n\lim_m a_{m, n} = \lim_n \lim_m \sum_{i\leq m}c_{m, n} = \lim_m \sum_{i\leq m}\lim_nc_{m, n} = \lim_m \lim_n \sum_{i\leq m}c_{m, n} = \lim_m \lim_n a_{m, n}
$$
    
\end{proof}

\begin{theorem}\label{Theorem 5.5}

A Banach space $E$ is reflexive iff for every bounded bi-orthogonal system $\{(e_i, f_j)\}_{i, j\in\mathbb{N}}$, $\big\{ \sum_{i\leq n}e_i \big\}_{n\in\mathbb{N}}$ is not uniformly bounded and $\big\{ \sum_{j\leq n}f_j \big\}_{n\in\mathbb{N}}$ is not uniformly bounded.
    
\end{theorem}

\begin{proof}

One direction is proved by \textbf{Proposition \ref{Proposition 5.3}}. When $E$ is reflexive, assume $\{(e_i, f_j)\}_{i,j\in\mathbb{N}}$ is a bi-orthogonal system where $\sup_n\|\sum_{i \leq n}e_i\| < \infty$. Therefore $\big\{ \sum_{i \leq n}e_i \big\}_{n \in\mathbb{N}}$ (or restricted to a subsequence) will converge weakly to a point, say $e$. Then for each $n\in\mathbb{N}$, $f_n(e) = \lim_k f_n\big( \sum_{i\leq k}e_i\big) = 1$. For each $k\in\mathbb{N}$, $\lim_n f_n\big( \sum_{i\leq k}e_i \big) = 0$. However, the doubly indexed sequence $\Big\{ f_n\big( \sum_{i\leq k}e_i\big)\Big\}_{n, k\in\mathbb{N}}$ satisfies the condition in \textbf{Proposition \ref{Proposition 5.4}} and hence:

$$
1 = \lim_n \lim_k f_n\big( \sum_{i\leq k}e_i\big) = \lim_k \lim_n f_n\big( \sum_{i\leq k}e_i\big) = 0
$$
which is absurd. On the other hand, when $E$ is reflexive, $E^*$ is also reflexive by \textbf{Proposition \ref{Proposition 4.19}}. Therefore if $\sup_n \big\| \sum_{i\leq n}f_i\big\| < \infty$, the sequence $\{\sum_{i\leq n}f_i\}_{n\in\mathbb{N}}$ (or a subsequence) converges weakly to a point, say $f$, which will lead to a contradiction similar to above. 

\end{proof}

\begin{theorem}

Given a Banach space $E$, the following are equivalent:

\begin{enumerate}[label = \arabic*)]

    \item $E$ is non-reflexive.
    
    \item There exists a bounded bi-orthogonal system $\{(e_i, f_j)\}_{i, j\in\mathbb{N}}$ and $\omega>0$ so that for every strictly decreasing sequence $\{\alpha_i\}_{i\in\mathbb{N}}\subseteq (0, \infty)$ with $\lim_i\alpha_i = 0$, we have $\sum_{i\geq 1}\alpha_i e_i$ converges and $\Big \|\sum_{i\geq 1}\alpha_i e_i\Big\| \leq \omega \alpha_1$.

    \item There exists a bounded bi-orthogonal system $S = \{(e_i, f_j)\}_{i, j\in\mathbb{N}}$ and $\gamma>0$ such that for every summable sequence $\{\lambda_i\}_{i\in\mathbb{N}}\subseteq (0,\infty)$:

    $$
    \inf\left\{ \left\|x - \sum_{i\geq 1}\lambda_i e_i \right\|: x\in B(S)_{\perp}\right\} \geq \gamma\sum_{i\geq 1}\lambda_i
    $$
    
\end{enumerate}
    
\end{theorem}

\begin{proof}$\hspace{1cm}\\$
\begin{itemize}

    \item $1)\Longrightarrow 2):$ Suppose $E$ is non-reflexive. Then according to the proof of \textbf{Proposition \ref{Proposition 5.3}}, for an arbitrary $\sigma\in(0, 1)$, we will have $\{b_i\}_{i \geq 1}\subseteq X_{\leq 1+\sigma}$, $\{y_i\}_{i \geq 1}\subseteq X_{\leq 1}$, $\{\beta_i\}_{i \geq 1}\subseteq (0,\infty)$ such that $\inf_n\beta_n = \beta > 0$ and:
    
    $$
    y_i(b_j) = 
    \begin{cases}
    \beta_i,\hspace{0.8cm} 1 \leq i \leq j \leq n\\
    0,\hspace{0.93cm} 1 \leq j < i \leq n
    \end{cases}
    $$
    Define $e_1 = b_1$ and for each $i\in\mathbb{N}$ with $i>1$, define $e_i = b_i - b_{i-1}$. For each $i\in\mathbb{N}$, define $f_i = \dfrac{1}{\beta_i}y_i$. Then $\{(e_i, f_j)\}_{i, j\in\mathbb{N}}$ is a bounded bi-orthogonal system. Given a strictly decreasing sequence $\{\alpha_i\}_{i\in\mathbb{N}} \subseteq (0, 1)$ with $\lim_i\alpha_i = 0$, for each $n, m\in\mathbb{N}$ with $m>n$, we have:
    
    $$
    \begin{aligned}
    &\hspace{1cm} \sum_{n < i \leq m}\alpha_i e_i = \alpha_{n+1} (b_{n+1} - b_n) + \cdots + \alpha_m (b_m - b_{m-1})\\
    &\hspace{0.52cm} = \alpha_m b_m + (\alpha_{m-1}-\alpha_m) b_{m-1} + \cdots + (\alpha_{n+1} - \alpha_{n+2}) b_{n+1} - \alpha_{n+1}b_n\\
    &\implies\, \left\| \sum_{n < i\leq m}\alpha_i e_i \right\| = \left\| \alpha_m b_m - \alpha_{n+1} b_{n+1} + \sum_{n < i < m} \big( \alpha_i - \alpha_{i+1} \big)b_i \right\|\\
    &\hspace{3cm} \leq (1+\sigma) \left( \alpha_m + \alpha_{n+1} + \sum_{n < i < m}\big( \alpha_i - \alpha_{i+1}\big) \right) = 2(1+\sigma)\alpha_{n+1} \overset{n\rightarrow\infty}{\longrightarrow} 0
    \end{aligned}
    $$
    Hence, we have:

    $$
    \lim_n \sum_{i\leq n}\alpha_i e_i = \sum_{i\geq 1}\alpha_i e_i \hspace{1cm}\textbf{and}\hspace{1cm} \left\| \sum_{i\geq 1}\alpha_i e_i \right\| \leq (1+\sigma)\alpha_1
    $$
    
    \item $1)\Longrightarrow 3):$ Let $\{\lambda_i\}_{i\in\mathbb{N}}$ be a summable sequence and for an arbitrary $\sigma\in(0, 1)$, let $\{b_i\}_{i\geq 1}$, $\{y_i\}_{i\geq 1}$, $\{\beta_i\}_{i\geq 1}$ be given in the proof of \textbf{Proposition \ref{Proposition 5.3}} such that $\inf_n\beta_n = \beta>0$ and:

    $$
    y_i(b_j) = 
    \begin{cases}
    \beta_i,\hspace{0.8cm} 1 \leq i \leq j \leq n\\
    0,\hspace{0.93cm} 1 \leq j < i \leq n
    \end{cases}
    $$
    For each $i\in\mathbb{N}$ define:

    $$
    h_i = \frac{1}{\beta_i}f_i - \frac{1}{\beta_{i+1}}f_{i+1}.
    $$
    Again by the proof of \textbf{Proposition \ref{Proposition 5.3}}, $\mathcal{S} = \{(b_i, h_j)\}_{i, j\in\mathbb{N}}$ is a bounded bi-orthogonal system and $\sup_n\Big\| \sum_{i\leq n}h_i\Big\| \leq \dfrac{2}{\beta}$. Then for any $z\in B(\mathcal{S})_{\perp}$ and $n\in\mathbb{N}$
    
    $$
    \left\| z+\sum_{i \leq n}\lambda_i b_i \right\| \geq \left\vert\, \frac{\beta}{2}\sum_{j\leq n}h_j \left( z + \sum_{i\leq n}\lambda_i b_i \right)\, \right\vert = \frac{\beta}{2}\sum_{i\leq n}\lambda_i.
    $$
    Hence:

    $$
    \begin{aligned}
    &\hspace{1cm} \forall\,n\in\mathbb{N}, \hspace{0.3cm} \frac{\beta}{2}\sum_{i\leq n}\lambda_i \leq \inf\left\{ \left\| x - \sum_{i\geq 1}\lambda_i e_i \right\|: x\in B(S)_{\perp}\right\} \\
    & \implies\, \frac{\beta}{2} \sum_{i\geq 1}\lambda_i \leq \inf\left\{ \left\|x - \sum_{i\geq 1}\lambda_i e_i \right\|: x\in B(S)_{\perp}\right\}
    \end{aligned}
    $$
    
    \item $2)\Longrightarrow 1):$ Suppose $S = \{(e_i, f_j)\}_{i, j\in\mathbb{N}}$ is a bounded bi-orthogonal system and $\omega>0$ such that for any sequence $\{\alpha_i\}_{i\in\mathbb{N}} \subseteq (0, \infty)$ with $\lim_i\alpha_i = 0$, we have $\sum_{i\geq 1}\alpha_i e_i$ converges and $\Big\| \sum_{i\geq 1}\alpha_i e_i\Big\| \leq \omega\alpha_1$. Hence:

    $$
    \forall\,n\in\mathbb{N}, \hspace{0.3cm} \left\| \sum_{i\leq n}e_i \right\| \leq \omega \hspace{1cm}\implies\hspace{1cm} \sup_n \left\| \sum_{i\leq n}e_i \right\| \leq \omega
    $$
    which implies that $E$ is non-reflexive by \textbf{Theorem \ref{Theorem 5.5}}.
    
    \item $3)\Longleftarrow 1):$ Let $S = \{(e_i, f_j)\}_{i, j\in\mathbb{N}}$ be the bounded bi-orthogonal system and $\gamma>0$ be given such that for any summable sequence of positive real numbers $\{\lambda_i\}_{i\in\mathbb{N}}$:

    $$
    \inf\left\{ \left\|x - \sum_{i\geq 1}\lambda_i e_i \right\|: x\in B(S)_{\perp}\right\} \geq \gamma\sum_{i\geq 1}\lambda_i
    $$
    Define $M = \sup_j\|f_j\|$. Without losing generality (or replace $\gamma$ by a smaller positive real number), suppose $\gamma\in(0, \dfrac{1}{M})$. For each $n\in\mathbb{N}$, define:

    $$
    V_n = \overline{\operatorname{Span}\big( B(S)_{\perp}, \{e_i\}_{i>n}\big)}^{\|\cdot\|}
    $$
    and

    $$
    V_0 = \overline{\operatorname{Span}\big( B(S)_{\perp}, \{e_i\}_{i\geq 1}\big)}^{\|\cdot\|}
    $$
    In the subspace $V_0$, for each different pair $i, j\in\mathbb{N}$, we have:

    $$
    \|e_i-e_j\| = \sup\Big\{ \big\vert\, f(e_i-e_j)\,\big\vert: f\in V_0^*, \|f\| \leq 1\Big\} \geq \frac{1}{M}\big\vert\, f_i(e_i-e_j)\,\big\vert > \gamma
    $$
    which implies that for each $n\in\mathbb{N}$ and $i\leq n$:

    $$
    \inf_{v\in V_n}\|e_i - v\| > \gamma
    $$
    By \textbf{Hahn-Banach Theorem}, when $n=1$, there exists $g_1\in V_0^*$ such that $\|g_1\| = 1$, $g_1\Big|_{V_1}=0$ and $g_1(e_1) > \gamma$. By induction assume that when $n=k$, there exists $g_k\in V_0^*$ such that $g_k\Big|_{V_k} = 0$ and $g_k(e_i)=0$ for each $i\leq k$. Hence for any $i \leq k+1$, all of the following hold:

    $$
    \left\vert\, \left( f_{k+1}\Big|_{V_0} g_k\right)(e_i) \, \right\vert > \gamma,
    \hspace{0.5cm} 
    \left\| f_{k+1}\Big|_{V_0} + g_k\right\| = 1, \hspace{0.5cm}
    \left( g_k + f_{k+1}\Big|_{V_0^*} \right)\Big|_{V_{k+1}} = 0
    $$
    Hence, for each $k\in\mathbb{N}$, there exists $g_k\in V_0^*$ such that $g_k(e_i)>\gamma$ for each $i\leq k$, $\|g_k\|=1$ and $g_k\Big|_{V_k} = 0$. For each $k\in\mathbb{N}$, let $\overline{g_k}$ be the norm-preserving extension of $g_k$ to the entire $E$. Hence $\{ \overline{g_k} \}_{k\in\mathbb{N}} \subseteq \big[ B(S)_{\perp} \big]^{\perp} = B(S)$ since $B(S)$ is weak-$\ast$ closed. Since $\{\overline{g_k}\}_{k\in\mathbb{N}}$ is bounded and, there exists $g\in B(S)$ such that $\{\overline{g_k}\}_{k\in\mathbb{N}}$ (or a proper subsequence) that converges to $g$ in weak-$\ast$ topology. Then we have $g(e_k)>\gamma$ for all $k\in\mathbb{N}$. Assume by contradiction that $g\in B_1(S)$ and $\{F_m\}_{m\in\mathbb{N}} \subseteq \operatorname{Span}\{f_i\}_{i\in\mathbb{N}}$ converges to $g$ in norm. By diagonal method to find a proper subsequence, we can assume for each $n\in\mathbb{N}$, $\Big\{ \big\vert\, F_m(e_n)\, \big\vert\Big\}_{m\in\mathbb{N}}$ is increasing. Since, for each $m\in\mathbb{N}$, the sequence $\Big\{ \big\vert\, F_{m+1}(e_n)\,\big\vert\, - \big\vert\, F_m(e_n)\,\big\vert \Big\}_{n\in\mathbb{N}}$ is eventually equal to zero, the sequence $\Big\{ \big\vert\, F_m(e_n)\, \big\vert\Big\}_{m\in\mathbb{N}}$ satisfies the condition in \textbf{Proposition \ref{Proposition 5.4}} and hence we have:

    $$
    \lim_n\lim_m \big\vert\, F_m(e_n)\, \big\vert = \lim_m \lim_n \big\vert\, F_m(e_n)\, \big\vert = 0
    $$
    Since $\{F_m\}$ also converges to $g$ pointwise and $g(e_k)>\gamma$ for each $k\in\mathbb{N}$, we then have:

    $$
    0 = \lim_n\lim_m \big\vert\, F_m(e_n)\,\big\vert = \lim_n \vert\, g(e_n)\,\vert > \gamma
    $$
    which is absurd. Hence $g\notin B_1(S)$. If $E$ is reflexive, so is $E^*$ by \textbf{Proposition \ref{Proposition 4.19}} and hence so is $B_1(S)$. In this case, $B(S)$ is equal to the weak closure of $\operatorname{Span}\{f_i\}_{i\in\mathbb{N}}$, which is again equal to $B_1(S)$. However, we then have $g\in B(S)\backslash B_1(S)$ and obtain a contradiction. We can now conclude $E$ is not reflexive.
    
\end{itemize}
\end{proof}

\section{Characterizations by weakly compact subsets}

The main theorems in this section are from \cite{7} and \cite{8}. In the source, the main focus is on a complete \textbf{locally convex topological vector space} (in short LCTVS) and $E\subseteq X$ a bounded weakly closed subset, and a list of conditions on $E$, which are equivalent to that $X$ is reflexive, will be given. By \textbf{Remark \ref{Remark 2.8}}, it suffices to only consider the case where $X$ is a Banach space. As it is mentioned in \textbf{Section 3}, $X$ being reflexive is equivalent to the closed unit ball of $X$ being weakly compact, which is an immediate consequence of {\cite[Goldstine's Theorem]{22}}. Also, given $S$ a weakly closed convex subset of $X$, the first two theorems from \cite{8} state precisely when $S$ is weakly compact, which generalize the \textbf{Proposition \ref{Proposition 4.10}}. 

\begin{theorem}[{\cite[Theorem 5]{8}}]

A weakly closed subset $S\subseteq X$ in a Banach space is weakly compact if and only if every element in $X^*$ attains its norm on $S$.
    
\end{theorem}

\noindent
Several precise descriptions are given in \cite{8}, and then later summarized by main theorems in \cite{7}. Here we start from the first main theorem in \cite{7}

\begin{defn}

Given a vector space $V$, a sequence of points $\{v_i\}_{i\in\mathbb{N}}$ and $u_1, u_2\in \operatorname{conv}\{v_i\}_{i\in\mathbb{N}}$, $u_1, u_2$ are said to be \textbf{non-overlapping} if thee exists $N\in\mathbb{N}$ such that $u_1\in \operatorname{conv}\{v_i\}_{i\leq N}$ and $u_2\in \operatorname{conv}\{v_i\}_{i>N}$. For each $v\in\operatorname{conv}\{v_i\}_{i\geq 1}$, define:

$$
l(v) = \min\big\{ n\in\mathbb{N}: v\notin \operatorname{conv}\{v_i\}_{i>n} \big\}
$$
    
\end{defn}

\begin{theorem}[{\cite[Theorem 1]{8}}]\label{Theorem 6.2}
    A bounded closed convex subset $C$ of a Banach space is not weakly compact iff there exists $\mu, \sigma > 0$ such that for any $\delta, \Delta > 0$ with $\delta < \mu < \Delta$, there exists a sequence of elements $\{z_i\}\subseteq C$ such that both of the following hold:
    \begin{enumerate}[label = \arabic*)]
    
        \item $\forall\,\xi\in \operatorname{conv}\{z_i\}_{i\geq 1}$, $\|\xi\|\in (\delta, \Delta)$
        
        \item $\forall\,n\in\mathbb{N}$, $d\big[ \operatorname{conv}\{z_i\}_{i \leq n}, \operatorname{conv}\{z_i\}_{i > n} \big] \geq \sigma$
        
    \end{enumerate}
\end{theorem}
    
\begin{proof}$\hspace{0.44cm}\\$
\begin{itemize}

    \item $\Longleftarrow):$ Let $\{z_i\}\subseteq C, \delta, \Delta, \sigma, \mu$ be chosen so that conditions $1), 2)$ are satisfied. For each $n\in\mathbb{N}$, $K_n = \operatorname{conv}\{z_i\}_{i\geq n}$. Assume by contradiction $C$ is weakly compact. Hence by weakly compactness and \textbf{Proposition \ref{Proposition 3.1}}, $\bigcap_{n\in\mathbb{N}} \overline{K_n}\neq\emptyset$. Pick $w\in \bigcap_{n\in\mathbb{N}} \overline{K_n}$. Then we can find two non-overlapping points $\xi, \gamma\in K_1$ such that $\|w-\xi\| < \dfrac{\sigma}{2}$ and $\|w-\gamma\| < \dfrac{\sigma}{2}$. Hence $\|\xi-\gamma\| < \sigma$, which contradicts $2)$.
        
    \item $\Longrightarrow):$ Assume $C$ is not weakly compact. By \textbf{Proposition \ref{Proposition 3.1}}, we can find a strictly decreasing sequence of closed convex subsets of $C$, say $\{T_n\}$, such that $\bigcap_{n \geq 1}T_n = \emptyset$. For each $n\in\mathbb{N}$, $t_n \in T_n\backslash T_{n+1}$. Then we want to assume by contradiction that: for any $\mu, \sigma > 0$ there exists $\delta, \Delta > 0$ with $\delta < \mu < \Delta$, so that for any sequence $\{z_i\}\subseteq C$, there exists $\xi \in \operatorname{conv}\{z_i\}_{i \geq 1}$ so that $\|\xi\| \geq \Delta$ or $\|\xi\| \leq \delta$, or there exists $k\in\mathbb{N}$ so that 
    
    $$
    d\big[ \operatorname{conv}\{z_i\}_{i \leq k}, \operatorname{conv}\{z_i\}_{i > k} \big] < \sigma
    $$
    Since for each $n\in\mathbb{N}$ we have $T_{n+1}\subseteq T_n$, we then have:

    $$
    \sup_{N\in\mathbb{N}} \inf_{c\in T_N}\|c\| = \lim_N \inf_{c\in T_N}\|c\|
    $$
    If the limit above is equal to $0$, we then have $\inf_{c\in T_N}\|c\| = 0$ for all $N\in\mathbb{N}$, which implies that $0\in T_N$ for all $N\in\mathbb{N}$ since each $T_N$ is closed. However, this contradicts our assumption that $\bigcap_{n\in\mathbb{N}}T_n = \emptyset$. Hence we can conclude there exists $\mu > 0$ such that starting from some $N_1\in\mathbb{N}$, $\inf_{c\in T_n}\|c\| > \mu$ for all $n\geq N_1$. For each $n\in\mathbb{N}$ define $V_{\geq n} = \operatorname{conv}\{t_i\}_{i\geq n}$. Next we claim:

    \begin{equation}\label{e31}
    \exists\,t_1^1\in V_{\geq 1} \textbf{   such that   }\forall\,n\geq l(t_1^1),\hspace{0.3cm} d\big[ t_1^1, V_{\geq n} \big] \leq \frac{1}{2}
    \end{equation}
    Assume that the claim above is false. Then first let $z_1 = t_1$ and then we can find $n_1\geq l(t_1)$ such that $d\big[ t_1, V_{\geq n_1}\big] > \dfrac{1}{2}$. Next let $z_2 = t_{n_1}$ and then find $n_2\geq \max \big( l(z_1), l(z_2) \big)$ such that both $d\big[ z_1, V_{\geq n_2} \big]$ and $d\big[ z_2, V_{\geq n_2} \big]$ are greater than or equal to $\dfrac{1}{2}$. For each $z\in \operatorname{conv}\{z_1, z_2\}$, again find $n(z)$ such that $d\big[z, V_{\geq n(z)} \big] > \dfrac{1}{2}$. Since, for each $N\in\mathbb{N}$, the following function is continuous in $X$:

    $$
    d_N: C \rightarrow(0, \infty), \hspace{0.3cm} x\mapsto d\big[x, V_{\geq N} \big]
    $$
    Hence:

    $$
    \operatorname{conv}\{z_1, z_2\} \subseteq \bigcup_{z\in \operatorname{conv}\{z_1, z_2\}} d_{n(z)}^{-1} \left( \frac{1}{2}, \infty \right)
    $$
    Since $\operatorname{conv}\{z_1, z_2\}$ is compact, then there exists $N_2\in\mathbb{N}$ such that $d\big[z, V_{\geq N_2}\big] > \frac{1}{2}$ for each $z\in\operatorname{conv}\{z_1, z_2\}$. By induction, we can find $\{z_i\}_{i\in\mathbb{N}} \subseteq \{t_i\}_{i\in\mathbb{N}}$ so that for any $k\in\mathbb{N}$ with $k>1$:

    $$
    d\big[ \operatorname{conv}\{z_i\}_{i\leq k}, V_{\geq N_k} \big] > \frac{1}{2}
    $$
    Recall that $\inf_{c\in T_n}\|c\|>\mu$ for all $n\geq N_1$. Without losing generality, assume $\{z_i\}_{i\in\mathbb{N}} \subseteq \{t_i\}_{i\geq N_1}$ (or remove finitely many points) and $\{N_i\}$ is strictly increasing. For each $\delta, \Delta$ with $\mu\in(\delta, \Delta)$ and each $j\in\mathbb{N}$, pick $z_{n_j} \in \{z_i\}_{i\geq N_j}$ such that $\|z_{n_j}\|\in (\delta, \Delta)$ for all $j\in\mathbb{N}$ and for any $k\in\mathbb{N}$ there exists $K\in\mathbb{N}$ with:

    $$
    \begin{aligned}
    & \hspace{1cm} \{z_{n_j}\}_{j\leq k} \subseteq \{z_i\}_{i\leq K}, \hspace{1cm} \{z_{n_j}\}_{j>k} \subseteq V_{\geq N_K}\\
    &\implies\, d\big[ \operatorname{conv}\{z_{n_j}\}_{j\leq k}, \operatorname{conv}\{z_{n_j}\}_{j>k} \big] \geq d\big[ \operatorname{conv}\{z_i\}_{i\leq K}, V_{\geq N_k} \big] > \frac{1}{2}
    \end{aligned}
    $$
    which contradicts our assumption. Hence let $t_1^1 \in V_{\geq 1}$ such that for all $n\geq l(t_1^1)$, $d\big[ t_1^1, V_{\geq n} \big] < \dfrac{1}{2}$. Let $n_1 = \max(N_1, l(t_1^1)$, $\mu_1 = \inf_{c \in T_{n_1}}\|c\|$, $\sigma_1 = \dfrac{1}{2}$. Let $\{t_n^1\}_{n\in\mathbb{N}} \subseteq V_{\geq n_1}$ such that $\|t_1^1 - t_n^1\| < \sigma_1$ for all $n\in\mathbb{N}$. Let $\delta_1,\Delta_1$ be given by our assumption. Without losing generality (or restrict to a subsequence) suppose $\|t_n^1\|\in (\delta_1, \Delta_1)$ for all $n\in\mathbb{N}$. Since $\operatorname{conv}\{t_n^1\}_{n\in\mathbb{N}} \subseteq T_{n_1}$, for each $\xi\in \operatorname{conv}\{t_n^1\}$, $\|\xi\| \in  [\mu_1, \Delta_1)$. Hence by our assumption there exists $m_1\in\mathbb{N}$ such that:

    \begin{equation}\label{e32}
    d\big[ \operatorname{conv}\{t_n^1\}_{n\leq m_1}, \operatorname{conv}\{t_n^1\}_{n>m_1} \big] < \sigma_1
    \end{equation}
    For our next claim:

    \begin{equation}\label{e33}
    \exists\, t_1^2\in \operatorname{conv}\{t_n^1\}_{n > m_1}\textbf{   such that   }\forall\,n > \max(l(t_1^2), m_1), \hspace{0.3cm} d\big[ t_1^2, V_{\geq n}\big] \leq \frac{1}{2^2}
    \end{equation}
    similar to the how we prove (\ref{e31}), the claim above is true. Let $n_2 = \max(m_1, l(t_1^2))$, $\sigma_2 = \dfrac{1}{2^2}$ and $\mu_2 = \inf_{c\in T_{n_2}} \|c\|$. Then there exists $\{t_n^2\}_{n\in\mathbb{N}} \subseteq \{t_n^1\}_{n > m_1}$ such that, with $\delta_2, \Delta_2$ given by our assumption and $\mu_2\in(\delta_2, \Delta_2)$, $\|t_n^2\|\in(\delta_2, \Delta_2)$ and $\|t_n^2-t_1^2\| < \sigma_2$ for all $n\in\mathbb{N}$. Also since $\operatorname{conv}\{t_n^2\} \subseteq T_{n_2}$, by our assumption there exists $m_2\in\mathbb{N}$ such that:

    $$
    d\big[ \operatorname{conv}\{t_n^2\}_{n\leq m_2}, \operatorname{conv}\{t_n^2\}_{n>m_2} \big] < \sigma_2
    $$
    Since $\{t_n^2\}_{n\in\mathbb{N}} \subseteq \{t_n^1\}_{n>m_1}$, we have $\|t_{m_1}^1 - t_{m_2}^2\| < \sigma_1$ By induction, for each $k\in\mathbb{N}$, we have $\{t_n^k\}$, $m_k$ and $\sigma_k = \dfrac{1}{2^k}$ such that:

    $$
    d\big[ \operatorname{conv}\{t_n^k\}_{n\leq m_k}, \operatorname{conv}\{t_n^k\}_{n>m_k} \big] < \sigma_k
    $$ 
    and for each $i, j\in\mathbb{N}$ with $i\leq j$, we have $\|t_{m_i}^i - t_{m_j}^j\| < \sigma_i$. Hence $\{t_{m_i}^i\}_{i\in\mathbb{N}}$ is a Cauchy sequence in $T_{n_1}$ and suppose $t_{m_i}^i \rightarrow t$. Since for each $k\in\mathbb{N}$, there exists $M_K\in\mathbb{N}$ such that $\{t_{m_i}^i\}_{i\geq k} \subseteq T_{M_K}$, by the assumption that $\{T_n\}$ is strictly decreasing:

    $$
    t\in \bigcap_{K\in\mathbb{N}}T_{M_K} = \bigcap_{n\in\mathbb{N}}T_n = \emptyset
    $$
    which is absurd.
        
\end{itemize}
\end{proof}

\begin{theorem}[{\cite[Theorem 1]{7}}]\label{Theorem 6.3}
Let $X$ be a Banach space and $E$ a bounded weakly closed subset. Then the following statements are equivalent:

\begin{enumerate}[label = \arabic*)]
    \item $E$ is weakly compact
    \item $E$ is weakly countably compact
    \item For each sequence $\{x_n\}\subseteq E$, we can find $x \in E$ such that $\liminf f(x_n)\leq f(x)\leq \limsup f(x_n)\,\forall\,f \in X^*$
    
    \item If $\{K_n\}$ is a decreasing sequence of closed convex sets and $E\cap K_n\neq\emptyset\,\forall\,n\in\mathbb{N}$, then $E\cap\bigcap_n K_n\neq\emptyset$
    
    \item For each sequence of $\{x_n\}\subseteq E$ and bounded sequence of $\{f_n\}\subseteq X^*$ we have:
    $$
    \lim_n\lim_k f_n(x_k) = \lim_k\lim_n f_n(x_k)
    $$
    whenever these limits exists
    
    \item For each sequence $\{x_n\}\subseteq E$ and bounded sequence of $\{f_n\}\subseteq X^*$ we have:
    $$
    \inf\{f_n(x_k)\,\vert\,n < k\} \leq \sup\{f_n(x_k)\,\vert\,n > k\}
    $$
    
    \item For each sequence $\{x_n\}\subseteq E$, $0 \in \overline{\bigcup_{n \geq 1}[\operatorname{conv}\{x_i\}_{i \leq n} - \operatorname{conv}\{x_i\}_{i > n}]}$
    
    \item For each sequence $\{x_n\}\subseteq E$, $0 \in \overline{\bigcup_{n \geq 1}[\operatorname{Span}\{x_i\}_{i \leq n} - \operatorname{conv}\{x_i\}_{i > n}]}$
    
    \item There will \textbf{NOT} exist $\theta > 0$, a sequence $\{z_n\}\subseteq E$ and a bounded sequence $\{g_n\}\subseteq X^*$ such that
    $$
    g_n(x_k) 
    \begin{cases}
    > \theta, \hspace{1cm} n \leq k\\ 
    = 0, \hspace{1cm} n > k\end{cases}
    $$

    \item $E$ is weakly sequentially compact.
    
\end{enumerate}
\end{theorem}

\begin{proof}$\hspace{0.44cm}\\$
\begin{itemize}

    \item $1)\implies 2)\implies 3):$ Immediate.
    
    \item $3)\implies 4):$ Follows immediately by \textbf{Corollary \ref{Corollary 1.23}}.
    
    \item $1)\implies 5):$ While the equivalence between $1)$ and $2)$ is proven in \textbf{Theorem \ref{Theorem 1.19}}, the conclusion follows immediately the existence of the weak clustered points of $(x_n)_{n\in \mathbb{N}}$ and $(f_n)_{n\in \mathbb{N}}$.
        
    \item $5)\implies 6):$ Given $\{x_n\}\subseteq E$ and a bounded sequence of continuous linear functionals $\{f_n\}$, we have the set $\{f_n(x_k)\}$ is bounded in $\mathbb{C}$. Therefore, we can find subsequences of $\{x_{k_j}\}\subseteq \{x_k\}$ and $\{f_{n_i}\}\subseteq \{f_n\}$ such that both $\lim_i\lim_j f_{n_i}x_{k_j}, \lim_j\lim_i f_{n_i}x_{k_j}$ exists. If $\inf\{f_n(x_k)\,\vert\,n < k\} > \sup\{f_n(x_k)\,\vert\,n > k\}$, then:
    
    $$
    \begin{aligned}
    &\hspace{1cm}\forall\,n, k\in\mathbb{N}, \hspace{0.3cm} \liminf_j f_n(x_{k_j}) = \lim_j f_n(x_{k_j}) > \limsup_i f_{n_i}(x_k) = \lim_i f_{n_i}(x_k)\\
    &\implies\,\forall\, i, j\in\mathbb{N}, \hspace{0.3cm} \lim_j f_{n_i}(x_{k_j})> \lim_i f_{n_i}(x_{k_j})\\
    &\implies\, \limsup_i \big( \lim_j f_{n_i}(x_{k_j}) \big) > \liminf_j \big( \lim_i f_{n_i}(x_{k_j}) \big)\\
    &\implies\, \lim_i\lim_j f_{n_i}(x_{k_j}) > \lim_j\lim_i f_{n_i}(x_{k_j})
    \end{aligned}
    $$
    
    which contradicts our assumption. 
        
    \item $6)\implies 9):$ If $9)$ is false, for the sequence $\{z_n\}$ and $\{g_n\}$ given by the assumption we have:

    $$
    0 = \sup\big\{ g_n(z_k) \,\vert\, n>k \big\} < \theta\leq \inf\big\{ g_n(z_k)\,\vert\, n<k\big\}
    $$
    
    \item $4)\implies 7):$ Given a sequence $\{x_n\}\subseteq E$, by assumption let $x \in \bigcap_n \overline{\operatorname{conv}\{x_i\}_{i \geq n}}$. Fix an open ball centred at $0$, say $W$ and then find a open convex neighborhood $U$ so that $U-U\subseteq W$. Since $x\in \overline{\operatorname{conv}\{x_i\}_{i\geq 1}}$, there $N\in\mathbb{N}$ and $v_N \in \operatorname{conv}\{x_i\}_{i\leq N}$ such that $x-v_N\in U$. Since $x\in \overline{\operatorname{conv}\{x_i\}_{i>N}}$, let $s_N\in \operatorname{conv}\{x_i\}_{i>N}$ such that $x-s_N \ in U$. Hence:

    $$
    v_N - s_N \in W\cap \operatorname{conv}\{x_i\}_{i\leq N} - \operatorname{conv}\{x_i\}_{i>N} \hspace{0.3cm}\implies\hspace{0.3cm} W\cap \bigcup_{n\geq 1} \big[ \operatorname{conv}\{x_i\}_{i\leq n} - \operatorname{conv}\{x_i\}_{i>n} \big]\neq\emptyset
    $$
    Since $W$ is an arbitrary open ball centred at $0$, the conclusion follows immediately.

    \item $7)\implies 8):$ Immediate.
        
    \item $8)\implies 9):$ Suppose $9)$ is false, Let $\{g_n\}\subseteq X^*$, $\{z_n\}\subseteq E$ and $\theta$ be given such that:

    $$
    g_n(z_k)
    \begin{cases}
    >\theta, \hspace{1cm} n\leq k\\
    =0, \hspace{1cm} n>k
    \end{cases}
    $$
    Without losing generality, suppose $\|g_n\|\leq 1$ for each $n\in\mathbb{N}$. Then for each $N\in\mathbb{N}$, $\{a_i\}_{i\leq N}\subseteq\mathbb{C}$ and a finite set of convex coefficients $\{\alpha_j\}$, we have:

    $$
    \left\| \sum_{i\leq N}a_i z_i - \sum_{j>N}\alpha_j z_j \right\| \geq \left\vert\, g_{N+1} \left( \sum_{i\leq N}a_i z_i - \sum_{j>N}\alpha_j z_j \right)\, \right\vert = \left\vert\, \sum_{j>N}\alpha_j g_{N+1}(z_j)\, \right\vert > \theta
    $$
    
    \item $1) \Longleftrightarrow 7):$ Follows immediately by \textbf{Theorem \ref{Theorem 6.2}}.

    \item $9)\Longrightarrow 1):$ Since $1)$ is equivalent to $7)$ according to \textbf{Theorem \ref{Theorem 6.2}}, it suffices to prove that $9)$ implies $7)$. Suppose $7)$ is false and let $\{z_i\}_{i\in\mathbb{N}} \subseteq E$ be a sequence such that for some $\delta\in(0, 1)$:

    \begin{equation}\label{e34}
    \inf_{n\in\mathbb{N}} d\big[ \operatorname{conv}\{z_i\}_{i\leq n}, \operatorname{conv}\{z_i\}_{i>n} \big] = \delta
    \end{equation}
    Without losing generality (or remove $z_1$ from $\{z_i\}_{i\in\mathbb{N}}$), we can assume:

    $$
    \inf\big\{ \|z\|: z\in\operatorname{conv}\{z_i\}_{i\in\mathbb{N}} \big\} = \delta
    $$
    Hence, by \textbf{Corollary \ref{Corollary 1.22}} there exists $f\in X^*_{\leq 1}$ such that:

    $$
    \lambda = \inf\big\{ f(z): z\in\operatorname{conv}\{z_i\}_{i\geq 1} \big\} > f(0) = 0
    $$
    Observe that for each $n\in\mathbb{N}$, $\operatorname{conv}\{z_i\}_{i\leq n}$ is compact and has non-empty interior. Then by (\ref{e34}) and \textbf{Theorem \ref{Theorem 3.11}}, for each $n\in\mathbb{N}$ there exists $g_n\in X^*_{\leq 1}$, $c_n\in\mathbb{R}$ such that:

    $$
    \sup\big\{ g_n(c): c\in \operatorname{conv}\{z_i\}_{i\leq n} \big\} < c_n < \inf\big\{ g_n(c): c\in \operatorname{conv}\{z_i\}_{i>n} \big\}
    $$
    In the case where $\{c_n\}$ is not bounded, we can then find $\{\lambda_n\}_{n\in\mathbb{N}} \subseteq (0, 1)$ such that $\{ \lambda_n c_n\}$ is bounded and $\inf_n \vert\, \lambda_n c_n\,\vert = 0$. Then replace each $g_n$ by $\lambda_n g_n$ and $c_n$ by $\lambda_n c_n$. Therefore without losing generality, we can assume the set $\{c_n\}$ is bounded and $\inf_n \vert\,c_n\,\vert = 0$. Since $\inf_n c_n=0$, we can further assume $\lim_n c_n=0$ (or replace $\{c_n\}$ by a proper subsequence) and then assume that $\sup_n \vert\,c_n\,\vert < \lambda$ (or remove the first $N$ elements). Next define: 

    $$
    h_n: X\rightarrow \mathbb{C}, \hspace{0.3cm} z\mapsto
    \begin{cases}
    0, \hspace{1cm} z\in\operatorname{Span}\{z_i\}_{i\leq n}\\
    g_n+f, \hspace{0.19cm} z\notin \operatorname{Span}\{z_i\}_{i>n}
    \end{cases}
    $$
    By our assumption, fix $\theta\in \big( 0, \lambda-\sup_n\vert\,c_n\,\vert \big)$. Then for each $n, k\in\mathbb{N}$

    $$
    h_n(z_k)
    \begin{cases}
    =0, \hspace{1cm} k\leq n \\
    >\theta, \hspace{1cm} k>n \\
    \end{cases}
    $$

    \item $9)\Longrightarrow 10):$ Suppose $E$ is not sequentially compact. Then $E$ is not weakly compact. By \textbf{Theorem \ref{Theorem 6.2}}, there exists $\mu, \sigma>0$ such that for any $\delta, \Delta>0$ with $\mu\in (\delta, \Delta)$, there exists a sequence $\{z_i\}\subseteq E$ such that:

    $$
    \forall\,\xi \in \operatorname{conv}\{z_n\}, \hspace{0.3cm} \|\xi\| \in (\delta, \Delta)
    $$
    and:

    $$
    \forall\,n\in\mathbb{N},\hspace{0.3cm} d\big[ \operatorname{conv}\{z_i\}_{i\leq n}, \operatorname{conv}\{z_i\}_{i>n} \big] \geq \sigma
    $$
    Assume by contradiction that $3)$ is true. Then there exists $z\in E$ such that for each $g\in X^*$:

    $$
    \liminf_i g(z_i) \leq g(z) \leq \limsup_i g(z_i)
    $$
    By \textbf{Corollary \ref{Corollary 1.23}}, we have $z\in \overline{\operatorname{conv}\{z_i\}_{i\geq 1}}$. Then there exists $z'\in \operatorname{conv}\{z_i\}_{i\leq N}$ for some $N\in\mathbb{N}$ such that $\|z-z'\|< \dfrac{\sigma}{2}$, which implies for any $x\in \operatorname{conv}\{z_i\}_{i>N}$:

    $$
    \|z-x\| \geq \big\vert\, \|z-z'\| - \|z'-x\| \,\big\vert > \frac{\sigma}{2}
    $$
    Hence $z\notin \overline{\operatorname{conv}\{z_i\}_{i>N}}$. By \textbf{Corollary \ref{Corollary 1.22}}, there exists $f\in X^*$ such that:

    \begin{equation}\label{e37}
    f(z) > \sup\big\{ f(v)\,\vert\, v\in\operatorname{conv}\{z_i\}_{i>N}\big\}
    \end{equation}
    Meanwhile, by assumption we have:

    $$
    f(z) \leq \limsup_i f(z_i) \leq \lim_n\sup\big\{ f(v) \,\vert\, v\in \operatorname{conv}\{z_i\}_{i\geq n} \big\} \leq \sup\big\{ f(v) \,\vert\, v\in \operatorname{conv}\{z_i\}_{i>N}\big\}
    $$
    which contradicts (\ref{e37}).
    
    \item $10)\Longrightarrow 9):$ Suppose $9)$ is false and let $\{z_i\}\subseteq E$, $\{g_n\}\subseteq X^*$ and $\theta>0$ be given such that:

    $$
    g_n(z_i) =
    \begin{cases}
    \theta, \hspace{1cm} n\leq i\\
    0, \hspace{1cm} n>i\\
    \end{cases}
    $$
    Assume by contradiction that $24)$ is correct. Then suppose $\{z_{n_i}\}_{i\in\mathbb{N}} \subseteq \{z_i\}$ is a subsequence that weakly converges to $z\in E$. Then for each $n\in\mathbb{N}$, $\{g_n(z_{n_i})\}_{i\in\mathbb{N}}$ is eventually equal to $\theta$ and for each $i\in\mathbb{N}$, $\{g_n(z_{n_i})\}_{n\in\mathbb{N}}$ is eventually equal to $0$. Then by \textbf{Proposition \ref{Proposition 5.4}}:

    $$
    \theta = \lim_n g_n(z) = \lim_n\lim_i g_n(z_{n_i}) = \lim_i\lim_n g_n(z_{n_i}) = 0
    $$
    which is absurd.

\end{itemize}
\end{proof}

\begin{theorem}[{\cite[Theorem 2]{7}}]\label{Theorem 6.4}
Let $X$ be a Banach space and $E$ a weakly closed bounded subset. Then the following statement are equivalent and each is equivalent to any one of statement listed under \textbf{Theorem \ref{Theorem 6.3}}.

\begin{enumerate}[label = \arabic*)]\addtocounter{enumi}{10}
    
    \item For each complex-valued function $f$ that is weakly continuous on $E$, $\sup_{e\in E}\vert\, f(e)\,\vert < \infty$.
    
    \item For each complex-valued function $f$ that is weakly continuous on $E$, there exists $e_0\in E$ such that $\vert\,f(e_0)\,\vert = \sup_{e\in E} \vert\,f(e)\,\vert$.
    
    \item For each $f\in X^*$, there exists $e_0\in E$ such that $\vert\, f(e_0)\,\vert = \sup_{e\in E} \vert\,f(e)\,\vert$.
    
    \item $\overline{\operatorname{circ}E}$ is weakly compact.
    
    \item $\overline{\operatorname{conv}E}$ is weakly compact.
    
    \item For an arbitrary sequence $\{x_n\}\subseteq E$, if $g$ is a weakly continuous linear functional such that $\lim_n g(x_n)$ exists, then there exists $e\in E$ with $g(e) = \lim_n g(x_n)$.
    
    \item For an arbitrary sequence $\{x_n\}\subseteq E$, if there exists $g\in X^*$ such that $\lim_n g(x_n)$ exists, then there exists $e\in E$ with $g(e) = \lim_n g(x_n)$.
    
\end{enumerate}
\end{theorem}

\begin{proof}

From now on statements $1)\sim 9)$ will be corresponding statements in \textbf{Theorem \ref{Theorem 6.3}}.

\begin{itemize}
    
    \item $1)\Longrightarrow 11):$ Suppose $\pi$ is a weakly continuous complex-valued function and unbounded on $E$. Then if we let $\mathbb{D}$ be the unit disk on $\mathbb{C}$, then the weakly open cover, $\big\{ \pi^{-1}(n\mathbb{D}) \big\}_{n\in\mathbb{N}}$, does not have any finite sub-covers. Hence $E$ is not weakly compact.
    
    \item $11)\Longrightarrow 12):$ Suppose $\pi$ is a weakly continuous complex-valued function such that for all $e_0\in E$, $\vert\,\pi(e_0)\,\vert < \sup_{e\in E}\vert\, \pi(e)\,\vert$. Then define:

    $$
    \pi^*: E\rightarrow\mathbb{R}, \hspace{0.3cm} x\mapsto \frac{1}{\Big\vert\, \big[ \sup_{e\in E}\vert\, \pi(e) \,\vert \big] - \pi(x)\, \Big\vert}
    $$
    Clearly $\pi^*$ is a weakly continuous but unbounded on $E$.

    \item $12) \Longrightarrow 13):$ Immediate.
    
    \item $13)\Longrightarrow 14):$ By \textbf{Theorem \ref{Theorem 2.4}}, since $\overline{\operatorname{circ}E}$ is weakly closed and bounded, if $\overline{\operatorname{circ}E}$ is not weakly compact, there exists $g\in X^*$ such that for all $e_0\in E$, $\vert\, g(e_0)\,\vert < \sup_{e\in E} \vert\,g(e)\,\vert$.
    
    \item $14)\implies 15)\implies 1):$ Immediate.
    
    \item $1)\implies 16):$ Let $L = \lim_n g(x_n)$. Suppose $L\notin g(E)$, then the following weakly open cover of $E$:

    $$
    \left\{ \Big\{ x\in E: \big\vert\, L-g(x)\, \big\vert > \frac{1}{n} \Big\} \right\}_{n\in\mathbb{N}}
    $$
    does not have a finite sub-cover and hence $E$ cannot be weakly compact.
    
    \item $16)\implies 17)\implies 13):$ Immediate.
    
\end{itemize}
\end{proof}

\noindent
To prove the next list of conditions is equivalent to previous one (from \textbf{Theorem \ref{Theorem 6.3}} and \textbf{Theorem \ref{Theorem 6.4}}), we need the following result:

\begin{theorem}[{\cite[Schauder-Tychonoff Theorem]{19}}]\label{ST-Theorem}
In a LCTVS $V$, if $K\subseteq X$ is a convex subset and $K_0\subseteq K$ is compact, then every continuous mapping from $K$ to $K_0$ has a fixed point in $K_0$. 
    
\end{theorem}

\begin{theorem}[{\cite[Theorem 3]{7}}]\label{Theorem 6.6}

Let $X$ be a Banach space and $E$ a convex bounded closed subset of $X$. Then the following statements are equivalent to each of statements $1)\sim 17)$ in \textbf{Theorem \ref{Theorem 6.3}} and \textbf{Theorem \ref{Theorem 6.4}}.

\begin{enumerate}[label = \arabic*)]\addtocounter{enumi}{17}

    \item If $\big\{ (x_n, f_i) \big\}_{n, i\in\mathbb{N}}$ is a bi-orthogonal system where $\{f_i\}_{i\in\mathbb{N}}$ is bounded, then there exists $N\in\mathbb{N}$ such that $\sum_{i\leq N}x_i \notin E$.
    
    \item If $\{(x_n, f_i)\}$ is a bounded bi-orthogonal system, then there exists $N\in\mathbb{N}$ such that $\sum_{i\leq N}x_i\notin E$.
    
    \item For any sequence of hyperplanes $\{H_i\}_{i\in\mathbb{N}}$, we have:

    $$
    \forall\,n\in\mathbb{N}, \hspace{0.3cm} \bigcap_{i\leq N}E\cap H_i \neq\emptyset \hspace{1cm}\implies\hspace{1cm} \bigcap_{i\in\mathbb{N}} E\cap H_i \neq\emptyset
    $$
    
    \item Recall for each subset $A\subseteq X$, we use $\operatorname{flat}(A)$ to denote the smallest flat set that contains $A$ (check \textbf{Definition \ref{Definition 3.1}}). For each sequence $\{x_n\}\subseteq E$, we have:
    
    $$
    0 \in\overline{\bigcup_{n \geq 1}\operatorname{Span}\{x_i\}_{i \leq n}-\operatorname{flat}\{x_i\}_{i > n}}
    $$
    
    \item Each affine continuous mapping of a non-empty closed convex subset of $E$ into $E$ itself has a fixed point
    
    \item There will \textbf{NOT} exist $\theta > 0$, a sequence $\{z_i\}\subseteq E$ and a bounded sequence $\{g_n\}_{n\in\mathbb{N}} \subseteq X^*$ such that:
    
    \begin{equation}\label{e35}
    g_i(z_n) = 
    \begin{cases}
    \theta,\hspace{1cm} i \leq n\\
    0,\hspace{1cm} i > n
    \end{cases}
    \end{equation}
    
\end{enumerate}
\end{theorem}

\begin{proof}$\hspace{1cm}\\$
\begin{itemize}

    \item $5) \Longrightarrow 19):$ Suppose $19)$ is false. Since we have $\{\sum_{i\leq n}x_i\}_{n\in\mathbb{N}} \subseteq E$, if $5)$ is true then:

    $$
    1 = \lim_n \lim_k f_n \left( \sum_{i\leq k}x_i \right) = \lim_k \lim_n f_n \left( \sum_{i\leq k}x_i \right) = 0
    $$
    which is absurd.

    \item $18) \Longrightarrow 19):$ Immediate.
    
    \item $4) \Longrightarrow 20):$ Immediate.

    \item $8) \Longrightarrow 21):$ Since a flat set is clear convex, the conclusion follows immediately.
    
    \item $1)\Longrightarrow 22):$ When $E$ is weakly compact, the conclusion follows by \textbf{Theorem \ref{ST-Theorem}}.
    
    \item $19) \Longrightarrow 23):$ Suppose $23)$ is false. Then let $\{z_n\}_{n\in\mathbb{N}}$, $\{g_i\}_{i\in\mathbb{N}}$ and $\theta>0$ be given by the assumption such that for each $i, n\in\mathbb{N}$:

    $$
    g_i(z_n) = 
    \begin{cases}
    \theta, \hspace{1cm}i\leq n\\
    0, \hspace{1cm} i>n\\
    \end{cases}
    $$
    Then let $x_1 = z_1$ and for each $n\in\mathbb{N} (n>1)$, define $x_n = z_n - z_{n-1}$. Then $\{x_i\}_{i\in\mathbb{N}}$ is bounded and hence $\big\{ (x_i, \dfrac{g_n}{\theta}) \big\}_{i, n\in\mathbb{N}}$ is a bounded bi-orthogonal system such that for each $N\in\mathbb{N}$, $\sum_{i\leq N}x_i = z_N\in E$.

    \item $20) \Longrightarrow 23):$ Suppose $23)$ is false and then let $\{z_n\}\subseteq E$, $\{g_i\}\subseteq X^*$ and $\theta>0$ be given such that (\ref{e35}) holds. Define:

    $$
    V = \left\{ v\in X \,\vert\, \lim_i g_i(v) \textbf{  exists} \right\}
    $$
    Clearly $V$ is a closed subspace and contains $\{z_n\}$. Then there exists $g_0\in X^*$ such that for each $v\in V$, $g_0(v) = \lim_i g_i(v)$. Define $H_0 = \operatorname{Ker}g_0$ and for each $i\in\mathbb{N}$, $H_i = g_i^{-1}\{\theta\}$. Then for each $n\in\mathbb{N}$:

    $$
    z_n \in \bigcap_{0\leq i \leq n}H_i
    $$
    Assume by contradiction that $\bigcap_{i\geq 0}E\cap H_i$ is non-empty and contains $s$ for some $s\in X$. Then $g_n(s) = \theta$ for each $n$ and hence $g_0(s) = \theta$, which contradicts that $s\in H_0$.\\

    \item $21) \Longrightarrow 23):$ Assume $23)$ is false and let $\{z_n\}\subseteq E$, $\{g_i\}\subseteq X^*$ and $\theta>0$ be given by the assumption. Since $\{g_i\}$ is bounded, suppose $\vert\,g_i(x)\,\vert < \theta$ for all $i$ whenever $\|x\|<\epsilon$ for some $\epsilon>0$. Fix an arbitrary $N\in\mathbb{N}$, $z\in \operatorname{Span}\{z_n\}_{n\leq N}$, $u\in \operatorname{flat}\{z_n\}_{n>N}$. Then $g_{N+1}(z-u) = -\theta$ which implies that $\|z-u\| \geq\epsilon$ and hence:

    $$
    \begin{aligned}
    & \hspace{1cm} \left\{ x\in X: \|x\|<\epsilon \right\} \cap \left[ \operatorname{Span}\{z_n\}_{n\leq N} - \operatorname{flat}\{z_n\}_{n>N} \right] = \emptyset \\
    & \implies\, \left\{ x\in X: \|x\| < \frac{\epsilon}{2} \right\} \cap \bigcup_{N\in\mathbb{N}} \overline{ \left[ \operatorname{Span}\{z_n\}_{n\leq N} - \operatorname{flat}\{z_n\}_{n>N} \right]} = \emptyset\\
    & \implies\, \left\{ x\in X: \|x\|<\frac{\epsilon}{2} \right\} \cap \overline{ \bigcup_{N\in\mathbb{N}} \left[\operatorname{Span}\{z_n\}_{n\leq N} \cap \operatorname{flat}\{z_n\}_{n>N} \right]} = \emptyset
    \end{aligned}
    $$
    
    \item $22)\Longrightarrow 23:$ Suppose $23)$ is false and let $\{z_n\}\subseteq E$, $\{g_i\}\subseteq X^*$ and $\theta>0$ be given by the assumption. For each $i\in\mathbb{N}$, define $\alpha_i = \dfrac{1}{\theta}\big[ g_i - g_{i+1} \big]$. Since, for each $z\in \operatorname{conv}\{z_n\}_{n\geq 1}$, $g_1(z)=\theta$ and for each $i\in\mathbb{N}$, $g_i(z) - g_{i+1}(z)\geq 0$, we have for each $x\in \overline{\operatorname{conv}\{z_i\}_{i\geq 1}}$, $g_1(x) = \theta$ and $\alpha_i(x)\geq 0$ for each $i$. Also, for each $z\in \operatorname{conv}\{z_n\}_{n\geq 1}$, $g_1(z) - g_k(z)$ is eventually equal to $\theta$. Therefore:

    $$
    \begin{aligned}
    &\hspace{1cm} \forall\,z\in \operatorname{conv}\{z_n\}_{n\geq 1}, \hspace{0.3cm} \lim_k\sum_{i<k}\alpha_i(z) = \sum_{i\geq 1}\alpha_i(z) = \lim_k \frac{1}{\theta}\big[ g_1(z) - g_k(z) \big] = 1\\
    &\implies\, \forall\, x\in \overline{ \operatorname{conv}\{z_n\}_{n\geq 1} }, \hspace{0.3cm} \sum_{i\geq 1}\alpha_i(x) = 1
    \end{aligned}
    $$
    Fix $x\in \overline{\operatorname{conv}\{z_n\}_{n\geq 1}}$ and suppose $\{x_n\}\subseteq \operatorname{conv}\{z_n\}_{n\geq 1}$ converges to $x$. Then consider $\{g_i(x_n)\}_{i, n\in\mathbb{N}}$. For each $n\in\mathbb{N}$, $\big\{ g_i(z_n) \big\}$ is eventually equal to zero. For each $i\in\mathbb{N}$, $\{ g_{i+1}(z_n) - g_i(z_n) \}_{n\in\mathbb{N}}$ is eventually constant. Hence by \textbf{Proposition \ref{Proposition 5.4}}, we have:

    \begin{equation}\label{e36}
    \forall\,x\in \overline{\operatorname{conv}\{z_i\}_{i\geq 1}}, \hspace{0.3cm} \lim_i g_i(x) = \lim_i\lim_n g_i(x_n) = \lim_n\lim_i g_i(x_n) = 0
    \end{equation}
    For each $N\in\mathbb{N}$ define:

    $$
    w_N = \frac{g_N(x)}{\theta}z_N + \sum_{i<N}\alpha_i(x) z_i
    $$
    Then:

    $$
    \begin{aligned}
    &\hspace{1cm} \left\| w_N - \sum_{i\geq 1}\alpha_i(x) z_i \right\| \leq \left\| \sum_{i\geq N}\alpha_i(x)z_i \right\| + \left\vert\, \frac{g_N(x)}{\theta}\, \right\vert \|z_N\| \overset{N\rightarrow\infty}{\longrightarrow} 0\\
    &\implies\, \lim_N w_N = \sum_{i\geq 1}\alpha_i(x) z_i
    \end{aligned}
    $$
    For each $k, N\in\mathbb{N}$ with $N>k$:
    
    $$
    \begin{aligned}
    &\hspace{1.36cm} g_k(w_N - x) = \frac{g_N(x)}{\theta}g_k(z_N) + \sum_{i<N}\alpha_i(x) g_k(z_i)  - g_k(x)\\
    &\hspace{0.9cm} = \sum_{k\leq i < N}\alpha_i(x)\theta + g_N(x) - g_k(x) = g_k(x) - g_N(x) + g_N(x) - g_k(x) = 0\\
    &\implies\, \lim_N g_k(w_N - x) = g_k \left( \sum_{i\geq 1}\alpha_i(x)z_i \right) - g_k(x) = 0
    \end{aligned}
    $$
    Since $\{g_n\}_{n\in\mathbb{N}}$ separates points in $\overline{ \operatorname{Span}\{z_i\}_{i\in\mathbb{N}} }$, we then can conclude $x = \sum_{i\geq 1}\alpha_i(x) z_i$ and, for $x$ is arbitrarily fixed, we have for each $v\in \overline{\operatorname{conv}\{z_i\}_{i\geq 1}}$, $v = \sum_{i\geq 1}\alpha_i(v) z_i$. Next define:

    $$
    \pi: \overline{\operatorname{conv}\{z_i\}_{i\geq 1}} \rightarrow \overline{\operatorname{conv}\{z_i\}_{i\geq 1}}, \hspace{0.3cm} v = \sum_{i\geq 1}\alpha_i(v) z_i \mapsto \sum_{i\geq 1}\alpha_i(v) z_{i+1}
    $$
    Clearly $\pi$ is continuous and, since each $\alpha_i$ is linear, $\pi$ is also linear. If $u$ is a fixed point of $\pi$, we then have:

    $$
    \begin{aligned}
    &\hspace{0.92cm} \begin{cases}
    \alpha_1(u) = 0\\
    \forall\,i\in\mathbb{N}, \hspace{0.3cm} \alpha_i(u) = \alpha_{i+1}(u)\\
    \end{cases}
    \,\implies \hspace{0.3cm}
    \begin{cases}
    g_1(u) = g_2(u) = \theta\\
    \forall\,i\in\mathbb{N}, \hspace{0.3cm} g_i(u) - g_{i+1}(u) = g_{i+1}(u) - g_{i+2}(u)\\
    \end{cases}\\
    &\implies\, \forall\,i\in\mathbb{N}, \hspace{0.3cm} g_i(u) = \theta
    \end{aligned}
    $$
    which contradicts (\ref{e36}). Hence $\pi$ has no fixed points.

    \item $23)\Longrightarrow 9):$ Suppose $9)$ is false and let $\{z_n\}\subseteq E$, $\{g_i\}\subseteq X^*$ and $\theta>0$ be given by the assumption such that for each $i, n\in\mathbb{N}$:

    $$
    g_i(z_n)
    \begin{cases}
    >\theta, \hspace{1cm} i\geq n\\
    =0, \hspace{1cm} i>n
    \end{cases}
    $$
    Define $\theta_1 = \liminf_n g_1(z_n)$.
    
    \begin{itemize}
    
        \item If $\{g_1(z_n)\}_{n\in\mathbb{N}}$ has a subsequence that is eventually equal to $\theta_1$, let $\{z_n^1\}_{n\in\mathbb{N}}$ be that subsequence and define $\sigma_1 = \theta_1$.

        \item Otherwise, let $\{z_{k_n}\}_{n\in\mathbb{N}}$ be a proper subsequence such that $\lim_n g_1(z_{k_n}) = \theta_1$. Since $g_1(z_{k_n})>\theta$ for each $n$, without losing generality (or restrict to a subsequence), assume $\{g_1(z_{k_n})\}$ is strictly decreasing. Fix $\sigma_1\in \big( g_1(z_{k_2}), g_1(z_{k_1}) \big)$. Then for each $n\in\mathbb{N} (n>1)$, there exists $\alpha_n^1\in(0, 1)$ such that:

        $$
        \alpha_n^1 g_1(z_{k_n}) + (1-\alpha_n^1) g_1(z_{k_1}) = \sigma_1
        $$

        Then define $z_1^1 = z_{k_1}$ and for each $n\in\mathbb{N}$ with $n>1$, define:

        $$
        z_n^1 = \alpha_n^1 z_{k_n} + (1-\alpha_n^1) z_{k_1}
        $$
        
    \end{itemize}

     \noindent
     In either case, we have a subsequence $\{z_n^1\}\subseteq \operatorname{conv}\{z_n\}_{n\geq 1}$ and $\sigma_1\geq\theta$ such that $g_1(z_n^1) = \sigma_1$ for all $n\in\mathbb{N}$. Next find $K_2\in\mathbb{N}$ that is large enough so that $g_{K_2}(z_1^1)=0$. Define $\theta_2 = \liminf_n g_{K_2}(z_n^1)$. Similarly there exists $\sigma_2\geq \theta_2$, $\{z_n^2\}\subseteq \operatorname{conv}\{z_n^1\}$ such that for each $n\in\mathbb{N}$, $g_{K_2}(z_n^2) = \sigma_2$ and $g_1(z_n^2) = \sigma_1$.By induction, for each $j\in\mathbb{N} (j>1)$, there is $\{z_n^j\}_{n\in\mathbb{N}} \subseteq \operatorname{conv}\{z_n^{j-1}\}_{n\in\mathbb{N}}$ and $K_j\in\mathbb{N}$ such that for each $i\in\mathbb{N}$:

     $$
     g_{K_j}(z_n^i)=
     \begin{cases}
     \sigma_j, \hspace{0.85cm}j\leq i\\
     0, \hspace{1cm} j>i
     \end{cases}
     $$
     For each $j\in\mathbb{N}$, define $h_j = \dfrac{\theta}{\sigma_j}g_{K_j}$. Since $\{g_i\}$ is bounded and $\sigma_j\geq\theta$ for each $j$, $\{h_j\}$ is also a bounded sequence. Also $\{z_n^n\}_{j\in\mathbb{N}}$ is contained in $E$ for $E$ is convex. Then for each $j, n\in\mathbb{N}$:

     $$
     h_j(z_n^n)=
     \begin{cases}
     \theta, \hspace{1cm}j\leq n\\
     0, \hspace{1cm}j>n
     \end{cases}
     $$
    
\end{itemize}
\end{proof}

\begin{theorem}[{\cite[Theorem 5]{7}}]\label{Theorem 6.7}
Let $X$ be a Banach space and $E$ a convex closed bounded subset. Then the following theorems are equivalent to each of $1)\sim 23)$ from \textbf{Theorem \ref{Theorem 6.3}, \ref{Theorem 6.4}} and \textbf{Theorem \ref{Theorem 6.6}}.

\begin{enumerate}[label = \arabic*)]\addtocounter{enumi}{23}
    
    \item If two convex closed subsets $A_1, A_2$ of $E$ are disjoint, then there exists $f\in X^*$ such that:
    
    $$
    \sup_{a_1 \in A_1}f(a_1) < \inf_{a_2\in A_2}f(a_2)
    $$
    
    \item If two convex closed subsets $A_1, A_2$ of $E$ are disjoint, then we can find $r > 0$ and $f \in X^*$ such that: 
    
    $$
    \sup_{a_1\in A_1}f(a_1) < r < \inf_{a_2\in A_2}f(a_2)
    $$
    
\end{enumerate}
\end{theorem}

\begin{proof}$\hspace{0.44cm}\\$
\begin{enumerate}[label = \alph*)]

    \item $1) \Longrightarrow 24):$ Assume $E$ is weakly compact. Then both $A_1, A_2$ are also weakly compact and hence weakly sequential compact by the equivalence of $1)$ and $10)$. Assume by contradiction that $0\in \overline{A_1-A_2}$. Since $A_1-A_2$ is convex, $0$ is also in the weak closure of $A_1-A_2$. Then there exists two sequences, $\{x_n\}\subseteq A_1$ and $\{y_n\}\subseteq A_2$ such that $x_n-y_n$ converges to $0$ weakly. By weakly sequential compactness, let $a_1$ be a weak cluster point of $\{x_n\}$ and $a_2$ a weak cluster point of $\{y_n\}$. Then we have $a_1=a_2$, which contradicts that $A_1\cap A_2 = \emptyset$.

    \item $24)\Longrightarrow 25):$ Immediate.

    \item $25) \Longrightarrow 23):$ Suppose $23)$ is false. Then similar to the proof to $24)\Longrightarrow 23)$, we will first define the following two disjoint closed convex subsets of $E$ where $\{z_i\}$ is given by the assumption:

    $$
    \begin{aligned}
    & B_1 = \left\{ \sum_{i\in\mathbb{N}} \gamma_i \Big[ \frac{i}{i+2}z_{2i+1} + \frac{2}{i+2}z_{2i}\Big]: \{\gamma_i\}\subseteq [0, 1] \textbf{   and   } \sum_{i\in\mathbb{N}}\gamma_i\in \Big[ \frac{1}{2}, 1\Big] \right\}\\
    & B_2 = \left\{ \sum_{i\in\mathbb{N}} \gamma_i \Big[ \frac{i+1}{i+2}z_{2i+1} + \frac{1}{i+2}z_{2i}\Big]: \{\gamma_i\}\subseteq [0, 1] \textbf{   and   } \sum_{i\in\mathbb{N}}\gamma_i= \Big[ \frac{1}{2}, 1\Big] \right\}
    \end{aligned}
    $$
    Then assume by contradiction that there exists $f\in X^*$ and $r>0$ such that:

    $$
    \sup_{b_1\in B_1}f(b_1) \leq r \leq \inf_{b_2\in B_2}f(b_2)
    $$
    Therefore, for each $N\in\mathbb{N}$, we have:

    $$
    \begin{aligned}
    &\hspace{1cm} \begin{cases}
    \dfrac{1}{2}f\left[ \dfrac{N+1}{N+2}z_{2N+1} + \dfrac{1}{N+2}z_{2N}\right] \geq r\\\\
    f\left[ \dfrac{N}{N+}z_{2N+1} + \dfrac{2}{N+2}z_{2N} \right]\leq r
    \end{cases}
     \hspace{0.2cm}\implies\hspace{0.2cm}
    \begin{cases}
    (N+2)f(z_{2N+1}) + f(z_{2N}) \geq 2(N+2)r\\
    Nf(z_{2N+1}) + 2f(z_{2N}) \leq (N+2)r
    \end{cases}\\
    &\implies\, 2f(z_{2N+1}) - f(z_{2N}) \geq (N+2)r \hspace{0.3cm}\implies\hspace{0.3cm} r\leq \frac{3}{N+2}\sup_{i\in\mathbb{N}}\|z_i\|\|f\| \overset{N\rightarrow\infty}{\longrightarrow} 0
    \end{aligned}
    $$
    Since $N$ is arbitrarily picked, we then conclude that $r\leq 0$, which contradicts our assumption.
    
\end{enumerate}
\end{proof}

\begin{theorem}[{\cite[Theorem 4]{7}}]\label{Theorem 6.8}
In a Banach space $X$, the following statements are equivalent to statements $1)\sim 25)$ from \textbf{Theorem \ref{Theorem 6.3}, \ref{Theorem 6.4}, \ref{Theorem 6.6}, \ref{Theorem 6.7}} with $E$ replaced by $X_{\leq 1}$ the closed unit ball of $X$.

\begin{enumerate}[label = \arabic*)]\addtocounter{enumi}{25}

    \item $X$ is not reflexive
    
    \item For each $\sigma\in(0, 1)$ there exists a sequence $\{z_i\}_{i\geq 1}\subseteq X_{\leq 1}$ such that for each $\xi\in \operatorname{conv}\{z_i\}_{i\geq 1}$, $\sigma < \|\xi\| \leq 1$ and:

    $$
    \inf_n d\big[ \operatorname{conv}\{z_i\}_{i\leq n}, \operatorname{conv}\{z_i\}_{i>n} \big] >\sigma
    $$
    
    \item For each $\sigma\in(0, 1)$ there exists a sequence $\{z_i\}_{i\geq 1}\subseteq X_{\leq 1}$ such that for each $\xi\in \operatorname{conv}\{z_i\}_{i\geq 1}$, $\sigma < \|\xi\| \leq 1$ and:

    $$
    \inf_n d\big[ \operatorname{Span}\{z_i\}_{i\leq n}, \operatorname{conv}\{z_i\}_{i>n} \big] >\sigma
    $$
    
    \item For each $\sigma\in(0, 1)$ there exists a sequence $\{z_i\}_{i\geq 1}\subseteq X_{\leq 1}$ such that for each $\xi\in \operatorname{conv}\{z_i\}_{i\geq 1}$, $\sigma < \|\xi\| \leq 1$ and:

    $$
    \inf_n d\big[ \operatorname{Span}\{z_i\}_{i\leq n}, \operatorname{flat}\{z_i\}_{i>n} \big] >\sigma
    $$
    
    \item For each $\sigma\in(0, 1)$ there exists a sequence $\{z_i\}_{i\geq 1}\subseteq X_{\leq 1}$ such that for each $\xi\in \operatorname{conv}\{z_i\}_{i\geq 1}$, $\sigma < \|\xi\| \leq 1$ and for any $n, p\in\mathbb{N}$ and $\{a_i\}_{i\geq 1}\subseteq  \mathbb{C}$:

    $$
    \left\| \sum_{i\leq n+p}a_iz_i \right\| \geq \frac{\sigma}{4} \left\| \sum_{i\leq n}a_iz_i \right\|
    $$
    
\end{enumerate}
\end{theorem}

\begin{proof}$\hspace{0.44cm}\\$

\begin{itemize}

    \item $1)\Longleftrightarrow 26) \Longleftrightarrow 27):$ The equivalence between $27)$ and that $X_{\leq 1}$ is weakly compact is proved in \textbf{Theorem \ref{Theorem 6.2}} and $X_{\leq 1}$ is weakly compact iff $X$ is reflexive.
    
    \item $29)\Longrightarrow 28) \Longrightarrow 27):$ Immediate.

    \item $26)\Longrightarrow 29):$ By \textbf{Theorem \ref{Theorem 3.2}}, for each $\sigma\in(0, 1)$, there exists $\{z_i\}_{i\in \mathbb{N}}\subseteq X_{\leq 1}$ and $\{g_n\}\subseteq X^*_{\leq 1}$ such that for each $n, i\in\mathbb{N}$:

    $$
    g_n(z_i)
    \begin{cases}
    >\sigma, \hspace{1cm}n\leq i\\
    =0, \hspace{1cm}n>i
    \end{cases}
    $$
    Therefore, for each $n\in\mathbb{N}$, given $x\in\operatorname{Span}\{z_i\}_{i\leq n}$ and $y\in \operatorname{flat}\{z_i\}_{i>n}$, we have:

    $$
    \|x-y\| \geq \big\vert\, g_{n+1}(x-y)\,\big\vert > \theta
    $$
    The conclusion follows since $x$ and $y$ are arbitrarily picked.

    \item $30)\Longrightarrow 26):$ By assumption, for an arbitrary $n\in\mathbb{N}$, $x\in \operatorname{conv}\{z_i\}_{i\leq n}$ and $y\in \operatorname{conv}\{z_i\}_{i>n}$, we have:

    $$
    \|x-y\| \geq \frac{\sigma}{4}\|x\| > \frac{\sigma^2}{4}.
    $$
    which implies for each $n\in\mathbb{N}$:

    $$
    d\big[ \operatorname{conv}\{z_i\}_{i\leq n}, \operatorname{conv}\{z_i\}_{i>n} \big] > \frac{\sigma^2}{4}.
    $$
    By \textbf{Theorem \ref{Theorem 3.2}}, $X$ is not reflexive.
    
    \item $26) \Longrightarrow 30):$ Suppose $X$ is not reflexive. Let $Q$ be the canonical mapping from $X$ to $X^{\ast\ast}$. Then fix an arbitrary $\Delta\in (0, 1)$. There exists $F\in X^{\ast\ast}$ such that $\|F\|<1$ and $d\big[ F, Q(X) \big] > \Delta$. Then $\|F\|\in (\Delta, 1)$. If:

    \begin{equation}\label{e38}
    \inf\Big\{ \big\vert\, F(f)\, \big\vert: f\in X^*, \|f\|=1 \Big\}\geq\Delta
    \end{equation}
    then pick an arbitrary $f\in X^*$ with $\|f\|=1$ and define $\Phi = \dfrac{\Delta}{F(f)}f$. If the left hand side in $(\ref{e38})$ is strictly less than $\Delta$, since $\|F\|>\Delta$, there must exist $f\in X^*$ with $\|f\|=1$ such that $F(f)=\Delta$. In general, there exists $\Phi\in X^*$ with $\|\Phi\|\leq 1$ such that $F(\Phi)=\Delta$. Next we will show there exists a sequence $\{z_i\}_{i\in\mathbb{N}}\subseteq X_{\leq 1}$ such that $\Phi(z_i)=\Delta$ for each $i\in\mathbb{N}$ and a sequence $\{h_n\}\subseteq X^*$ such that for each $n\in\mathbb{N}$, $\|h_n\|<\dfrac{2}{\Delta}$, $F(h_n)=0$, and for each $i, n\in\mathbb{N}$:

    $$
    h_n(z_i)=
    \begin{cases}
    \Delta, \hspace{0.88cm}n\geq i\\
    0, \hspace{1cm} n<i
    \end{cases}
    $$
    Since $\|F\|<1$ and $F(\Phi)=\Delta$, if $\|\Phi\|\leq\Delta$, we then have $F\Big(\dfrac{\Phi}{\Delta}\Big) = 1$, which implies that $\|F\|\geq 1$ and contradicts our assumption. Therefore $\|\Phi\|>\Delta$ and there exists $z_1\in X_{\leq 1}$ such that $\Phi(z_1)=\Delta$. Then for each $a_1\in\mathbb{C}$:

    \begin{equation}\label{e39}
    \big\vert\, a_1\Phi(z_1)\, \big\vert = \big\vert\, a_1\Delta\, \big\vert \leq \big\vert\, F(\Phi) + a_1\Phi(z_1)\, \big\vert + \big\vert\, F(\Phi)\, \big\vert \leq  \big\| \Phi\big\| \big\|F + a_1Q(z_1) \big\| +  \big\| F \big\|.
    \end{equation}
    Since $d\big[ F, Q(X) \big]>\Delta$ and $\|F\|\in(\Delta, 1)$, we then have $\Delta\|F\| < d\big[ F, Q(X)\big] \leq \|F+a_1Q(z_1)\|$ and together with (\ref{e39}) we have that for each $a_1\in\mathbb{C}$:

    $$
    \big\vert\, a_1\Delta\,\big\vert \leq \big\| F+a_1Q(z_1) \big\| + \|F\| < \big\| F+a_1Q(z_1) \big\| + \frac{1}{\Delta}\big\| F+a_1Q(z_1) \big\| = \left( 1 + \frac{1}{\Delta}\right) \big\| F+a_1Q(z_1) \big\|
    $$
    Therefore, for every $a_1, a_2\in\mathbb{C}$ with $a_2\neq 0$, we have:

    $$
    \Big\vert\, \frac{a_1}{a_2}\Delta\, \Big\vert \leq \left( 1+\frac{1}{\Delta} \right) \Big\| F+\frac{a_1}{a_2}Q(z_1) \Big\| \hspace{0.3cm}\implies\hspace{0.3cm} \big\vert\, a_1\Delta + a_2\,0 \,\big\vert \leq \left( 1+\frac{1}{\Delta}\right) \Big\| a_2F + a_1Q(z_1) \Big\|
    $$
    Then by \textbf{Theorem \ref{Theorem 1.19}}, there exists $h_1\in X^*$ with $\|h_1\|\leq\dfrac{2}{\Delta}$ such that $F(h_1)=0$ and $h_1(z_1)=\Delta$. Hence for every $a_1, a_2\in\mathbb{C}$:

    \begin{equation}\label{e40}
    \big\vert\, a_1\Delta+a_2\,0 \,\big\vert = \big\vert\, a_1F(\Phi) + a_2F(h_1)\, \big\vert \leq \|F\|\big\| a_1\Phi+ a_2h_1 \big\|
    \end{equation}
    Again by \textbf{Theorem \ref{Theorem 1.19}}, there exists $z_2\in X$ with $\|z_2\|<1$ such that $\Phi(z_2)=\Delta$ and $h_1(z_2)=0$. Next with $z_1, z_2\in X_{\leq 1}$ and $h_1\in X^*$, we will find $h_2$ by the following inequality: for every $a_1, a_2\in\mathbb{C}$:

    $$
    \begin{aligned}
    &\hspace{0.46cm} \big\vert\, a_1\Delta + a_2\Delta \,\big\vert \\
    &\leq \big\vert\, a_1\Phi(z_1) + a_2\Phi(z_2) \,\big\vert \leq \big\vert\, a_1\Phi(z_1) + a_2\Phi(z_2) + F(\Phi) \,\big\vert + \big\vert\, F(\Phi) \,\big\vert\\
    &\leq \Big\| F+ \big( a_1z_1 + a_2z_2 \big)\Big\| + \|F\| \leq  \left( 1+\frac{1}{\Delta} \right) \Big\| F+ \big( a_1z_1 + a_2z_2\big) \Big\| 
    \end{aligned}
    $$
    By \textbf{Theorem \ref{Theorem 1.19}} there exists $h_2\in X^*$ with $\|h_2\|<\dfrac{2}{\Delta}$ such that $F(h_2)=0$ and $h_2(z_1)=h_2(z_2)=\Delta$. Then by (\ref{e40}) we can find $z_3\in X$ with $\|z_3\|<1$ such that $\Phi(z_3)=\Delta$ and $h_1(z_3)=h_2(z_3)=0$. Repeat the process above and eventually we have $\{h_n\}_{n\in\mathbb{N}}$ with $\|h_n\|<\dfrac{2}{\Delta}$ for each $n$, and $\{z_i\}_{i\in\mathbb{N}} \subseteq X_{\leq 1}$ such that $F(h_n)=0$, $\Phi(z_i)=\Delta$ for each $i$ and:

    $$
    h_n(z_i)=
    \begin{cases}
    \Delta, \hspace{0.88cm}n\geq i\\
    0, \hspace{1cm} n<i
    \end{cases}
    $$
    For each $n\in\mathbb{N}$ with $n>1$, define:

    $$
    H_n = \big\{ h_i-h_{i-1}: i\in\{ 1, 2, \cdots, n\} \big\}
    $$
    Fix $n, p\in\mathbb{N}$ with $n>1$ and $z\in \operatorname{Span}\{z_i\}_{i\leq n}$ with $\|z\|=1$. Since $\|z_i\|\leq 1$ for each $i\in\mathbb{N}$, without losing generality, suppose $\vert\, \alpha_n \,\vert \geq 1$. Then:

    \begin{equation}\label{e41}
    \big\vert\, h_n(z) - h_{n-1}(z) \,\big\vert = \vert\, \alpha_n h_n(z_n)\,\vert  = \vert\,\alpha_n\,\vert\Delta \geq \Delta = \Delta\|z\|.
    \end{equation}
    We can then conclude for each $z\in \operatorname{Span}\{z_i\}_{i\leq n}$, (\ref{e41}) holds and hence, for each $z\in \operatorname{Span}\{z_i\}_{i\leq n}$ and $\{\alpha_i\}_{n < i \leq n+p} \subseteq \mathbb{C}$:

    $$
    \left\| z + \sum_{n < i \leq n+p}\alpha_iz_i \right\| \geq \frac{\Delta}{4} \left\vert\, \big(h_n - h_{n-1} \big)\, \left( z + \sum_{n < i\leq n+p}\alpha_iz_i \right) \, \right\vert = \frac{\Delta}{4}\big\vert\, (h_n-h_{n-1})(z)\, \big\vert \geq \frac{\Delta^2}{4}\|z\|
    $$
    Since $\Delta$ is arbitrarily picked, we can assume $\Delta^2$ is strictly greater than the $\sigma$ that is picked from $(0, 1)$ and then the conclusion follows.
    
\end{itemize}

\end{proof}

\section{Characterizations of super-properties}

In this section we will introduce properties stronger than reflexivity and give an overview of why they are stronger and how they can be characterized. As mentioned in \textbf{Introduction}, the \textbf{von Neumann-Jordan constant} (see \cite{20}) is one of the main tools to characterize super-reflexivity and uniformly non-squareness. The characterizations of uniformly non-squareness is purely geometric while the characterizations of super-reflexivity will involve representability of other Banach spaces and definitions of equivalent norms. We can also see from both \textbf{Section 1} and the current section that uniform convexity is the strongest among all other properties that have been covered. However, the existence of an equivalent uniform convex norm is a comparatively weak restriction. 

\subsection{Uniformly non-square Banach space}

\begin{defn}[\cite{20}]

In a Banach space $(X, \|\cdot\|)$, the \textbf{von Neumann-Jordan Constant for $(X, \|\cdot\|)$} is defined as:

$$
\inf\left\{ C>0: \forall\,x, y\in X \,\, (\|x\|^2 + \|y\|^2 > 0), \hspace{0.3cm} \frac{1}{C} \leq \frac{\|x+y\|^2 + \|x-y\|^2}{2\|x\|^2 + 2\|y\|^2} \leq C\right\}
$$
and is denoted by $CJ(X, \|\cdot\|)$. 

\end{defn}

\begin{rem}

In any Banach space $(X, \|\cdot\|)$, for each $x, y\in X$ we have:

$$
\|x+y\|^2 + \|x-y\|^2 \leq 2\|x\|^2 + 2\|y\|^2
$$
When $\|x\|=1$:

$$
\frac{\|x+x\|^2 + \|x-x\|^2}{2\|x\|^2 + 2\|x\|^2} = \frac{1}{2}
$$
Hence, we have $CJ(X, \|\cdot\|) \in [1, 2]$ for any Banach space $(X, \|\cdot\|)$. In particular, when $(X, \|\cdot\|)$ is a Hilbert space, we have $CJ(X, \|\cdot\|) =1$.
    
\end{rem}

\begin{defn}

A normed linear space $(X, \|\cdot\|)$ is said to be \textbf{uniformly non-square} if there exists $\delta>0$ such that for any $x, y\in X_{\leq 1}$ we have $\|x+y\| \leq 2(1-\delta)$ whenever $\|x-y\| > 2(1-\delta)$.

\end{defn}

\begin{theorem}[{\cite[Theorem 1]{12}}]\label{Theorem 7.4}

Fix $p>1$. Given $(X, \|\cdot\|)$ a Banach space, if we use $\ell_p^2(X)$ to denote the vector space $X\times X$ equipped with the following norm:

$$
\big\| (x, y) \big\|_p = \Big( \|x\|^p + \|y\|^p \Big)^{\frac{1}{p}}
$$
then the following statements are equivalent:

\begin{enumerate}[label = \arabic*)]

    \item $X$ is uniformly non-square.

    \item There exists $\epsilon, \delta\in (0, 1)$ such that for any $x, y\in X_{\leq 1}$, if $\|x-y\| > 2(1-\epsilon)$, then:

    \begin{equation}\label{e42}
    \left\| \frac{x+y}{2} \right\|^p \leq (1-\delta)\frac{\|x\|^p + \|y\|^p}{2}
    \end{equation}

    \item There exists $\delta\in(0, 1)$ such that for any $x, y\in X_{\leq 1}$, (\ref{e42}) holds if $\|x-y\| > 2(1-\delta)$,

    \item There exists $\delta\in (0, 2)$ such that for any $x, y\in X$:

    $$
    \left\| \frac{x-y}{2} \right\|^p + \left\| \frac{x+y}{2} \right\|^p \leq (2-\delta)\frac{\|x\|^p + \|y\|^p}{2}
    $$

    \item The following linear function:

    \begin{equation}\label{e43}
    A: X\times X\rightarrow X\times X, \hspace{0.3cm} (x, y) \mapsto (x+y, x-y)
    \end{equation}
    is continuous and $\|A\|<2$ when both $X\times X$ are equipped with $\|\cdot\|_p$.
    
\end{enumerate}
    
\end{theorem}

\begin{proof}$\hspace{1cm}\\$
\begin{itemize}

    \item $1)\Longrightarrow 2)$: Assume by contradiction that $ii)$ is false when $X$ is uniformly non-square. Then for each $n\in\mathbb{N}$, there exists $x_n, y_n\in X_{\leq 1}$ with $\|x_n-y_n\| \geq 2(1-n^{-1})$ such that:

    \begin{equation}\label{e44}
    \left\| \frac{x_n + y_n}{2} \right\|^p > \left(1-\frac{1}{n} \right) \frac{\|x_n\|^p + \|y_n\|^p}{2}
    \end{equation}
    By dividing $\max(\|x_n\|, \|y_n\|)$ in (\ref{e44}), we can assume that $\|x_n\|=1$ and $\|y_n\|\leq 1$. Therefore we have:

    \begin{equation}\label{e45}
    2\left( 1-\frac{1}{n} \right) \leq \|x_n-y_n\| \leq \|x_n\|+\|y_n\| = 1+\|y_n\| \leq 2
    \end{equation}
    and from (\ref{e44}):

    \begin{equation}\label{e46}
    \left(1-\frac{1}{n}\right)\frac{1+\|y_n\|^p}{2}  =\left(1-\frac{1}{n}\right) \frac{\|x_n\|^p + \|y_n\|^p}{2} < \Big\| \frac{x_n+y_n}{2} \Big\|^p \leq \left( \frac{1+\|y_n\|}{2} \right)^p \leq 1
    \end{equation}
    By (\ref{e45}) we have $\lim_n \|y_n\| = 1$. Hence by (\ref{e45}) and (\ref{e46}) we have:

    $$
    \lim_n \left\| \frac{x_n+y_n}{2} \right\| = \lim_n \left\| \frac{x_n-y_n}{2} \right\| = 1
    $$
    Hence, for any $\delta\in(0, 1)$ we have both $\|x_n-y_n\| \geq 2(1-\delta)$ and $\|x_n+y_n\| \geq 2(1-\delta)$ for large enough $n$, which contradicts that $X$ is uniformly non-square.

    \item $2)\Longrightarrow 3):$ Immediate.

    \item $3) \Longrightarrow 1):$ Let $\delta\in(0, 1)$ be given. Then for two arbitrary $x, y\in X$, suppose $\|x-y\|<2(1-\delta)$. By assumption we then have:

    $$
    \left\| \frac{x+y}{2} \right\|^p \leq (1-\delta) \hspace{0.3cm}\implies\hspace{0.3cm} \left\| \frac{x+y}{2} \right\| \leq (1-\delta)^{\frac{1}{p}}
    $$
    Since $p>1$ and $\delta\in(0, 1)$, we have $(1-\delta)^{\frac{1}{p}} \in (1-\delta, 1)$ and hence there exists $\delta_0 \in (\delta, 1)$ such that $1-\delta_0 = (1-\delta)^{\frac{1}{p}}$. Therefore, whenever $\|x-y\| > 2(1-\delta_0)$, we also have $\|x-y\| > 2(1-\delta)$ and then $\|x+y\| \leq 2(1-\delta_0)$. 

    \item $1) \Longrightarrow 4):$ Suppose $4)$ fails. Then similar to the proof of $1)\Longrightarrow 2)$, for each $n\in\mathbb{N}$, there exists $\|x_n\|=1$ and $y_n\in X_{\leq 1}$ such that:

    \begin{equation}\label{e47}
    \left\| \frac{x_n - y_n}{2} \right\|^p + \left\| \frac{x_n + y_n}{2} \right\|^p > \left( 2-\frac{1}{n}\right) \frac{\|x_n\|^p + \|y_n\|^p}{2} = \left( 2-\frac{1}{n}\right) \frac{1 + \|y_n\|^p}{2}
    \end{equation}
    Since $\{\|y_n\|\}_{n\in\mathbb{N}}$ is bounded in $[0, 1]$, without losing generality suppose $\{\|y_n\|\}$ is convergent (or replace it by a convergent subsequence) and suppose $\alpha\in [0, 1]$ is its limit. Then by (\ref{e47}) we have:

    $$
    1+\alpha^p = \lim_n \left(2-\frac{1}{n}\right) \frac{1+\|y_n\|^p}{2} \leq \lim_n 2\left( \frac{1+\|y_n\|}{2} \right)^p = \frac{(1+\alpha)^p}{2^{p-1}}
    $$
    Since for each $r\in[0, 1]$, according to the binomial expansion of $(1+r)^p$, we have $(1+r)^p \geq 1+r^p$ and the function $\dfrac{(1+r)^p}{1+r^p}$ is increasing in $[0, 1]$. Therefore the maximum of $\dfrac{(1+r)^p}{1+r^p}$ in $[0, 1]$ is obtained when $r=1$ and is equal to $2^{p-1}$. Therefore $\alpha=1$ and from (\ref{e47}) we have:

    $$
    \begin{aligned}
    &\hspace{1cm} 2 = \lim_n \left(2-\frac{1}{n}\right)\frac{1+\|y_n\|^p}{2} \leq \lim_n \left\| \frac{x_n-y_n}{2} \right\|^p + \left\| \frac{x_n+y_n}{2} \right\|^p \leq \lim_n 2\left( \frac{1+\|y_n\|}{2}\right)^p =2\\
    &\implies\, \lim_n \left\| \frac{x_n-y_n}{2} \right\|^p + \left\| \frac{x_n+y_n}{2} \right\|^p = 2
    \end{aligned}
    $$
    Since both $\left\{ \left\| \dfrac{x_n-y_n}{2} \right\|^p \right\}$ and $\left\{ \left\| \dfrac{x_n+y_n}{2} \right\|^p \right\}$ are bounded above by $1$, we then can conclude:

    $$
    \lim_n \left\| \frac{x_n - y_n}{2} \right\|^p = \lim_n \left\| \frac{x_n + y_n}{2} \right\|^p = 1
    $$
    or:

    $$
    \lim_n \|x_n-y_n\| = \lim_n \|x_n+y_n\| = 2
    $$
    Hence, for any $\epsilon\in(0, 1)$ we will have both $\|x_n-y_n\|>2(1-\epsilon)$ and $\|x_n+y_n\| > 2(1-\epsilon)$ whenever $n$ is large enough, which implies that $X$ is not uniformly non-square.

    \item $4)\Longrightarrow 1):$ Let $\delta\in(0, 2)$ be given by $4)$ and hence for every $x, y\in X_{\leq 1}$, we have:

    $$
    \left\| \frac{x-y}{2} \right\|^p + \left\| \frac{x+y}{2} \right\|^p \leq 1-\frac{\delta}{2} \hspace{0.3cm}\implies\hspace{0.3cm} \max\left( \left\| \frac{x-y}{2} \right\|^p, \left\| \frac{x+y}{2} \right\|^p \right) \leq 1-\frac{\delta}{2}
    $$
    Since $p>1$ and $\delta\in(0, 2)$, there exists $\delta_0\in \left(0, \dfrac{\delta}{2} \right)$ such that:

    $$
    \left( 1-\frac{\delta}{2}\right)^{\frac{1}{p}} = 1-\delta_0
    $$
    Hence for any $x, y\in X_{\leq 1}$, we have $\max\big( \|x-y\|, \|x+y\| \big) \leq 2(1-\delta_0)$, which implies $X$ is uniformly non-square.

    \item $4)\Longleftrightarrow 5):$ By definition of the operator norm we have:

    $$
    \|A\|^p = \sup_{x,y\in X}\frac{\|A(x, y)\|_p^p}{\|(x, y)\|_p^p}  = \sup_{x, y\in X}\frac{\|x+y\|^p + \|x-y\|^p}{\|x\|^p + \|y\|^p}
    $$
    If $4)$ is true, since $p>1$, then:

    $$
    \|A\| \leq 2^{1-\frac{1}{p}}\left(1-\frac{\delta}{2}\right) < 2\left( 1-\frac{\delta}{2}\right) < 2
    $$
    If $5)$ is true, then there exists $\epsilon\in(0, 2)$ such that $\|A\|=2-\epsilon$ and hence, for any $x, y\in X$:

    $$
    \begin{aligned}
    &\hspace{1cm} \frac{\|x+y\|^p + \|x-y\|^p}{\|x\|^p + \|y\|^p} \leq (2-\epsilon)^p < (2-\epsilon)2^{p-1} \\
    & \implies\, \|x+y\|^p + \|x-y\|^p < 2^p(2-\epsilon) \frac{\|x\|^p + \|y\|^p}{2}\\
    & \implies\, \left\| \frac{x+y}{2} \right\|^p + \left\| \frac{x-y}{2} \right\|^2 < (2-\epsilon)\frac{\|x\|^p + \|y\|^p}{2}
    \end{aligned}
    $$
    
\end{itemize}
\end{proof}

\begin{lem}[{\cite[Lemma 2]{12}}]\label{Lemma 7.5}
Let $A: X\times X\rightarrow X\times X$ be given in (\ref{e43}) and both $X\times X$ are equipped with $\|\cdot\|_2$. Then $\|A\|^2 = 2CJ(X, \|\cdot\|)$

\end{lem}

\begin{proof}

By definition of $\|A\|$ and $CJ(X, \|\cdot\|)$ we have:

$$
\sup_{\substack{x, y\in X \\ \|x\|^2+\|y\|^2>0}} \frac{\|x+y\|^2 + \|x-y\|^2}{\|x\|^2 + \|y\|^2} = \|A\|^2 = 2^{\frac{2}{t}} \leq 2CJ(X, \|\cdot\|)
$$
If $2CJ(X, \|\cdot\|) > \|A\|^2$, then there exists $x_0, y_0\in X$ such that:

$$
\frac{1}{2}\|A\|^2 \geq \frac{\|x_0 + y_0\|^2 + \|x_0 - y_0\|^2}{2\|x_0\|^2 + 2\|y_0\|^2} > CJ(X, \|\cdot\|) > \frac{1}{2}\|A\|^2
$$
which is absurd. Hence $\|A\|^2 = 2CJ(X, \|\cdot\|)$.
    
\end{proof}

\begin{theorem}[{\cite[Theorem 2]{12}}]\label{Theorem 7.6}

For a Banach space $(X, \|\cdot\|)$, $CJ(X, \|\cdot\|)<2$ iff $X$ is uniformly non-square.
    
\end{theorem}

\begin{proof}

Let $A:X\times X\rightarrow X\times X$ be define by (\ref{e43}). By \textbf{Theorem \ref{Theorem 7.4}}, $X$ is uniformly non-square iff $\|A\|<2$ and by \textbf{Lemma \ref{Lemma 7.5}}, we have $2\|A\|^2 = CJ(X, \|\cdot\|)$ and hence $\|A\|<2$ iff $CJ(X, \|\cdot\|) < 2$.
    
\end{proof}

\subsection{Uniformly convexifiable \texorpdfstring{$\,\Longrightarrow\,$}{ implies } no finite-tree property}

\begin{defn}[\cite{5}]

In a Banach space $X$, given an $\epsilon > 0$, we call an ordered pair $(x_1, x_2)$ a $(1, \epsilon)$-\textbf{part of a tree} if $\|x_1-x_2\| \geq \epsilon$. For each $n\in\mathbb{N}$, given a $2^{n+1}$-tuple, $(x_1, x_2, \cdots, x_{2^{n+1}})$,  we call $(x_1,\cdots x_{2^{n+1}})$ an $(n+1, \epsilon)$-part of a tree if for each $1 \leq j \leq 2^n$, $\|x_{2j-1} - x_{2j}\| \geq\epsilon$ and the following $2^n$-tuple:

$$
\left( \frac{x_1+x_2}{2}, \frac{x_3+x_4}{2}, \cdots, \frac{x_{2^{n+1}-1}+x_{2^{n+1}}}{2} \right)
$$ 
is an $(n, \epsilon)$-part of a tree. $X$ is said to have the \textbf{finite tree property} if there is an $\epsilon> 0$ such that for each $n\in\mathbb{N}$ and each $\delta \in (0, 1)$, there is an $(n, \epsilon)$-part of a tree where the norm of each element in that $(n, \epsilon)$-part of a tree is bounded above by $1+\delta$.

\end{defn}

\begin{defn}

Let $X$ be a Banach space and $x \in X$. Given $x_1, x_2, x\in X$, the ordered pair $(x_1, x_2)$ is a $(1, \epsilon)$ \textbf{partition of} $x$ if $x_1+x_2 = x, \|x_1\|=\|x_2\|$ and $\left\| \dfrac{x_1}{\|x_1\|} - \dfrac{x_2}{\|x_2\|} \right\| \geq \epsilon$. For each $n\in\mathbb{N}$ with $n>1$, a $2^{n+1}$-tuple $(y_1, y_2, \cdots, y_{2^{n+1}})$ is an $(n+1, \epsilon)$-\textbf{partition of} $x$ if for each $1\leq j \leq 2^n$, $\|y_{2j-1}\| = \|y_{2j}\|$ and $\left\| \dfrac{y_{2j-1}}{\|y_{2j-1}\|} - \dfrac{y_{2j}}{\|y_{2j}\|} \right\| \geq \epsilon$, and the following $2^n$-tuple:

$$
\left( y_1+y_2, y_3+y_4, \cdots, y_{2^{n+1}-1} + y_{2^{n+1}} \right)
$$
is an $(n, \epsilon)$-partition of $x$. If $(x_1, x_2, \cdots, x_{2^n})$ is an $(n, \epsilon)$-partition of $x$, then for each $k\in\mathbb{N}$ with $k<n$, the $k$-\textbf{part of} $(x_1, \cdots, x_{2^n})$ is the following $2^k$-tuple:

$$
\left( \sum_{1\leq i \leq 2^{n-k}} x_i, \sum_{2^{n-k} < i \leq 2^{n-k+1}} x_i, \cdots, \sum_{2^n-2^{n-k} < i \leq 2^n}x_i \right)
$$
which is also a $(k, \epsilon)$-partition of $x$.

\end{defn}

\begin{defn}
A normed linear space $(X, \|\cdot\|_X)$ is \textbf{finitely representable} in another normed linear space $(Y, \|\cdot\|_Y)$ iff for each finite dimensional subspace $X_n$ of $X$ and for each $\lambda > 1$, there is an isomorphism $T_n: X_n \rightarrow Y$ such that for each $x \in X_n$, we have:
$$
\frac{1}{\lambda}\|x\|_X \leq \|T_n(x)\|_Y \leq \lambda\|x\|_X
$$
The statement is equivalent to the following: for each finite dimensional subspace $X_n$ of $X$ and each $\epsilon > 0$, there is an linear isomorphism $T_n: X_n \rightarrow Y$ such that for each $x \in X_n$, we have:
$$
(1-\epsilon)\|x\|_X \leq \|T_n x\|_Y \leq (1+\epsilon)\|x\|_X
$$
\end{defn}

\begin{defn}
A normed linear space $(X, \|\cdot\|_X)$ is \textbf{crudely finitely  representatble} in a normed linear space $(Y, \|\cdot\|_Y)$ iff there is $\lambda > 1$ such that for each finite dimensional subspace $X_n$ of $X$, there is an linear isomorphism $T_n: X_n \rightarrow Y$ such that:
$$
\frac{1}{\lambda}\|x\|_X \leq \|T_n x\|_Y \leq \lambda\|x\|_X
$$
\end{defn}

\begin{defn}
A Banach space $(X, \|\cdot\|_X)$ is \textbf{super-reflexive} if any Banach space $(Y, \|\cdot\|_Y)$ that is finitely representatble in $X$ is also reflexive.
\end{defn}

\begin{prop}[{\cite[Theorem 1]{21}}]\label{Proposition 7.12}
A Banach space $Y$ is super-reflexive iff no non-reflexive Banach space can be crudely finitely representable in $Y$
\end{prop}

\begin{proof}

Suppose $Y$ is not super-reflexive and hence there will be a non-reflexive Banach space $X$ that is finitely representable in $Y$. Obviously, if $X$ is finitely representable in $Y$, it will be crudely finitely representable in $Y$.\\

\noindent
Now suppose a non-reflexive Banach space $X$ is crudely representable in $Y$. By \textbf{30)} in \textbf{Theorem \ref{Theorem 6.8}}, we can find $\delta_0 > 0$ and a linearly independent sequence $\{x_n\}\subseteq X_{\leq 1}$ such that:

$$
\inf_{n\in\mathbb{N}} d \big[ \operatorname{conv}\{x_i\}_{i \leq n}, \operatorname{conv}\{x_i\}_{i > n} \big] > \delta_0
$$
Define $X_n = \operatorname{Span}\{x_i\}_{i \leq n}$ and then there will be an isomorphism $T_n: X_n \rightarrow Y$ such that for each $v \in X_n$:

\begin{equation}\label{e48}
\frac{1}{\lambda}\|v\|_X \leq \|T_n v\|_Y \leq \lambda\|v\|_X
\end{equation}
Define $y_k^n = \frac{1}{\lambda}T_n(x_k)$ for each $k \leq n$ and hence $\|y_k^n\| \leq 1$. For each $n, k\in\mathbb{N}$ with $k<n$, we have:
    
$$
\begin{aligned}
&\hspace{0.46cm} d \big[ \operatorname{conv}\{y_i^n\}_{i \leq k}, \operatorname{conv}\{y_i^n\}_{k < i \leq n} \big] \\
&\geq \frac{1}{\lambda^2}d \big[ \operatorname{conv}\{x_i\}_{i \leq k}, \operatorname{conv}\{x_i\}_{k < i \leq n} \big]\\
&\geq \frac{1}{\lambda^2} d \big[ \operatorname{conv}\{x_i\}_{i \leq k}, \operatorname{conv}\{x_i\}_{i>k} \big]\\
&\geq \frac{\delta_0}{\lambda^2}
\end{aligned}
$$
Since the inequality above holds for all $n>k$, we then can conclude:

\begin{equation}\label{e49}
d \big[ \operatorname{conv}\{y_i^n\}_{i \leq k}, \operatorname{conv}\{y_i^n\}_{i > k} \big] \geq \frac{\delta_0}{\lambda^2}
\end{equation}
Then by using the subsequence: 

$$
\{y_k^{2^n}\}_{2^{n-1} < k \leq 2^n, n\in\mathbb{N}} \subseteq \{y_k^n\}_{k \leq n, n\in\mathbb{N}}
$$
we are going to construct a Banach space that is finitely representable in $Y$ but non-reflexive. Let $\{\xi_i\}_{i \in\mathbb{N}}$ be a sequence of symbols and in $V = \operatorname{Span}\{\xi_i\}_{i\in\mathbb{N}}$, define:

$$
\left\| \sum_{k \leq n}a_i \xi_i \right\|_V = \limsup_M \left\| \sum_{i \leq n}a_i y_{i+2^{M-1}}^{2^M} \right\|_Y
$$
Clearly $(V, \|\cdot\|_V)$ is a normed linear space. Let $\overline{V}$ be the completion of $V$ with respect to the norm topology so that $(\overline{V}, \|\cdot\|_V)$ is a Banach space. Define $V_n = \operatorname{Span}\{\xi_j\}_{j \leq n}$. Fix $u = \sum_{k \leq n}a_i \xi_k\in V_n$ and suppose $\|u\|_V=1$. Fix an arbitrary $\epsilon\in(0, 1)$. We then can find $M(\epsilon, u) \in\mathbb{N}$ such that $M(\epsilon, u) > \log_2 n$ and for each $M\geq M(\epsilon, u)$:

\begin{equation}\label{e50}
1-\epsilon \leq \left\| \sum_{k \leq n}a_i y_{k+2^{M-1}}^{2^M} \right\|_Y \leq 1+\epsilon
\end{equation}
Fix an arbitrary $n\in\mathbb{N}$. Let $(V_n)_{\leq 1}$ be the closed unit ball of $V_n$ and suppose $\{u_j\}_{j \leq m}$ is a $\dfrac{\epsilon}{4}$-net of $(V_n)_{\leq 1}$. For each $j\leq m$, put $M_j = M\left( \dfrac{\epsilon}{4}, u_j \right)$. For each $N, k\in\mathbb{N}$, define:

$$
P_N(\xi_k)=
\begin{cases}
y_{k+2^{N-1}}^{2^N}, \hspace{1cm} k\leq 2^{N-1}\\
0, \hspace{2.04cm} k>2^{N-1}
\end{cases}
$$
and linearly extend $P_N$ to the entire $V$. Since the range of $P_N$ is finite-dimensional, $P_N$ is bounded and can be continuously extended to $\overline{V}$. For each $y\in V$, we have:

$$
\limsup_M \|P_M(y)\|_Y = \|y\|_V \hspace{0.5cm} \Longrightarrow \hspace{0.5cm} \sup_M \|P_M(y)\|<\infty
$$
Given $y\in\overline{V}$, suppose $\{y_n\}_{n\in\mathbb{N}}\subseteq V$ is a sequence that converges to $y$ in $\|\cdot\|_V$. Observe that:

$$
\limsup_n \|y_n\|_V = \liminf_n \|y_n\|_V = \lim_n \|y_n\|_V = \|y\|_V
$$
Fix $N(\epsilon)\in\mathbb{N}$ such that:

$$
\|y\|_V-\epsilon < \inf_{n\geq N(\epsilon)}\|y_n\|_V \leq \sup_{n\geq N(\epsilon)}\|y_n\|_V < \|y\|_V+\epsilon
$$
Fix an arbitrary $n\in\mathbb{N}$ with $n\geq N(\epsilon)$, and find $M(n)\in\mathbb{N}$ such that:

$$
\begin{aligned}
& \hspace{1cm} \|y_n\|_V-\epsilon < \sup_{M\geq M(n)} \|P_M(y_n)\|_Y < \|y_n\|_V+\epsilon \\
&\implies\, \|y\|_V-2\epsilon < \sup_{M\geq M(n)} \|P_M(y_n)\|_Y < \|y\|+2\epsilon \\
\end{aligned}
$$
which implies:

\begin{equation}\label{e51}
\begin{aligned}
& \hspace{1cm} \|y\|_V - 2\epsilon < \sup_{n\geq N(\epsilon)} \sup_{M\geq M\big( N(\epsilon) \big)} \|P_M(y_n)\|_Y < \|y\|_V + 2\epsilon \\
&\implies\, \|y\|_V -2\epsilon < \sup_{M\geq M\big( N(\epsilon) \big)} \sup_{n\geq N(\epsilon)} \|P_M(y_n)\|_Y < \|y\|_V + 2\epsilon 
\end{aligned}
\end{equation}
Since for each $M\in\mathbb{N}$, we have $\lim_n P_M(y_n)=P_M(y)$, from (\ref{e51}) we have:

$$
\begin{aligned}
&\hspace{1cm} \|y\|_V -2\epsilon < \sup_{M\geq M\big( N(\epsilon) \big)} \|P_M(y)\|_Y < \|y\|_V+2\epsilon\\ 
&\implies\, \sup_M \|P_M(y)\|_Y < \infty
\end{aligned}
$$
Since $y\in \overline{V}$ is arbitrarily picked, by \textbf{Uniform Boundedness Theorem}, we then can conclude $R = \sup_M \|P_M\| < \infty$. Now fix an arbitrary $y\in (V_n)_{\leq 1}$. Recall that $\{u_j\}_{j\leq m}$ is an $\dfrac{\epsilon}{4}$-net of $(V_n)_{\leq 1}$. Suppose $\|y-u_i\|_V < \dfrac{\epsilon}{4}$ for some $i\leq m$. Then we have:

\begin{equation}\label{e52}
\|P_{M_i}(y)\|_Y \leq \|P_{M_i}(y-u_i)\|_Y+\|P_{M_i}(u_i)\|_Y < \|y-u_i\|_V \|P_{M_i}\| + \left( 1+\frac{\epsilon}{4} \right)\|u_i\|_V < \frac{\epsilon}{4}R + \left(1+\frac{\epsilon}{4}\right)
\end{equation}
and by (\ref{e50}):

\begin{equation}\label{e53}
\begin{aligned}
\|P_{M_i}(y)\|_Y 
&\geq \|P_{M_i}(u_i)\|_Y - \frac{\epsilon}{4}\\
&\geq \left( 1-\frac{\epsilon}{4} \right)\|u_i\|_V-\|P_{M_i}(u_i-y)\|_Y\\ 
&\geq \left( 1-\frac{\epsilon}{4} \right)\|u_i\|_V-\frac{\epsilon}{4}R \\
&\geq \left( 1-\frac{\epsilon}{4} \right) \|y\|_V - \left( 1-\frac{\epsilon}{4} \right)\|u_i-y\|_V-\frac{\epsilon}{4}R \\
&> \left( 1-\frac{\epsilon}{4} \right) \|y\|_V - \left( 1-\frac{\epsilon}{4} \right) \frac{\epsilon}{4}R - \frac{\epsilon}{4}R \\
&\geq \left( 1-\frac{\epsilon}{4} \right) \|y\|_V - \frac{\epsilon}{2}R
\end{aligned}
\end{equation}
According to (\ref{e50}), both (\ref{e52}) and (\ref{e53}) also hold when $M_i$ is replaced by any $M\geq M_i$. Put $M_0 = \max_{i\leq m}M_i$. Since, for each $\delta\in(0, 1)$, there exists $\epsilon\in(0, \delta)$ such that $R\epsilon < \delta$, together with (\ref{e51}) and (\ref{e52}), for all $M\geq M_0$ and each $y\in V_n$ with $\|y\|=1$, we have:

$$
\left( 1 - \frac{\delta}{4} \right) - \frac{\delta}{2} < \|P_M(y)\|_Y < \frac{\delta}{4} + \left( 1+\frac{\delta}{4} \right)
$$
Therefore, for any $y \in V_n$ and any $M\geq M_0$:

$$
\begin{aligned}
&\hspace{1cm}1-\delta \leq \left\| P_M\left( \frac{y}{\|y\|_V} \right) \right\|_Y \leq 1+\delta\\
&\implies\,(1-\delta) \|y\|_V \leq \|P_M(y)\|_Y \leq (1+\delta)\|y\|_V
\end{aligned}
$$
Now we can conclude that, for each $n\in\mathbb{N}$, $\delta\in(0, 1)$, there exists $M(n)\in\mathbb{N}$ such that for any $M\in\mathbb{N}$ with $M\geq M(n)$ and any $y\in V_n$:

\begin{equation}\label{e54}
(1-\delta)\|y\|_V \leq \|P_M(y)\|_Y \leq (1+\delta)\|y\|_V
\end{equation}
Suppose $U_n\subseteq \overline{V}$ is an $n$-dimensional subspace $(n \geq 1)$ and has basis $\{u^i\}_{i\leq n}$. Fix $u=\sum_{i\leq n} \lambda_i u^i$ with $\|u\|=1$ and $\delta\in(0, 1)$. Find $\epsilon\in(0,\delta)$ such that $R\epsilon<\delta$. For each $i\leq n$, find $u_i\in V_{m_i}$ such that:

$$
\Big\| \sum_{i\leq n}\lambda_i \big( u^i - u_i\big) \Big\| < \epsilon
$$
Then for each $m_i$, find $M_i$ such that for each $M\geq M_i$, (\ref{e54}) holds for the given $\delta$ and each vector in $V_{m_i}$. Put $M' = \max_{i\leq n}M_i$. We then have $\sum_{i\leq n}\lambda_i u_i\in V_{M'}$ and hence for each $M\geq M'$ we have:

$$
\begin{aligned}
&\hspace{0.95cm} \left\| P_M\left( \sum_{i\leq n}\lambda_i u^i\right\| \right\|_Y \leq \left\| P_M\left( \sum_{i\leq n}\lambda_i u_i \right) \right\|_Y + \left\| P_M \left( \sum_{i\leq n}u^i - u_i \right) \right\|_Y \\
&\implies\, \left\| P_M\left( \sum_{i\leq n}\lambda_i u^i\right) \right\|_Y \leq (1+\delta) \left\|\sum_{i\leq n}\lambda_i u_i\right\|_V + R\left\| \sum_{i\leq n}\lambda_i (u^i - u_i) \right\|_Y \\
& \implies\, \left\| P_M\left( \sum_{i\leq n}\lambda_i u^i\right) \right\|_Y \leq (1+\delta) \left\|\sum_{i\leq n}\lambda_i u^i \right\|_V + (1+\delta) \left\| \sum_{i\leq n}\lambda_i (u^i - u_i) \right\|_Y + \delta\\
& \implies\, \left\| P_M\left( \sum_{i\leq n}\lambda_i u^i\right) \right\|_Y \leq (1+\delta) + (1+\delta)\epsilon + \delta \leq 1+4\delta
\end{aligned}
$$
Similarly, we can also show:

$$
\left\| P_M\left( \sum_{i\leq n}\lambda_i u^i\right) \right\|_Y \geq 1-4\delta
$$
Since $u= \sum_{i\leq n}\lambda_i u^i$ is arbitrarily picked, we can then conclude for each $y\in U_n$:
 
$$
\begin{aligned}
&\hspace{1cm} 1-4\delta \leq \left\| P_{M'}\left( \frac{y}{\|y\|_V} \right) \right\|_Y \leq 1+4\delta\\
&\implies\, (1-4\delta)\|y\|_V \leq \|P_{M'}(y)\|_Y \leq (1+4\delta)\|y\|_V
\end{aligned}
$$
As a result, $(\overline{V}, \|\cdot\|_V)$ is finitely representable in $Y$. Now fix arbitrary $k, n\in\mathbb{N}$ with $k<n$. Then there exists $M(k)\in\mathbb{N}$ such that (\ref{e53}) holds for all $y\in V_k$ whenever $M\geq M(k)$. Fix $N\geq M(k)$ and $N>\log_2 n$. Then together with (\ref{e49}) we have:

\begin{equation}\label{e55}
\begin{aligned}
& \hspace{0.46cm} d \big[ \operatorname{conv}\{\xi_i\}_{i \leq k}, \operatorname{conv}\{\xi_i\}_{k < i \leq n} \big]\\
& \geq \frac{1}{R} d \big[ P_N(\operatorname{conv}\{\xi_i\}_{i \leq k}), P_N(\operatorname{conv}\{\xi_i\}_{k < i \leq n}) \big]\\
& = \frac{1}{R} d\big[ \operatorname{conv}\{y_{i+2^{N-1}}^{2^N}\}_{i \leq k}, \operatorname{conv}\{y_{i+2^{N-1}}^{2^N}\}_{k < i \leq n} \big]\\
& \geq d \big[ \operatorname{conv}\{y_i^{2^N}\}_{i \leq k+2^{N-1}}, \operatorname{conv}\{y_i^{2^N}\}_{k+2^{N-1} < i \leq 2^N} \big]\\
& \geq \frac{1}{R} \frac{1}{\lambda^2}d \big[ \operatorname{conv}\{x_i\}_{i \leq k+2^{N-1}}, \operatorname{conv}\{x_i\}_{k+2^{N-1} < i \leq 2^N} \big]\\
& \geq \frac{\delta}{R\lambda^2}
\end{aligned}
\end{equation}
Since (\ref{e55}) holds for all $n>k$, we then can conclude:

\begin{equation}\label{e56}
d\big[ \operatorname{conv}\{\xi_i\}_{i\leq k}, \operatorname{conv}\{\xi_i\}_{i>k} \big] \geq \frac{\delta}{R\lambda^2}
\end{equation}
and since the $k$ in (\ref{e56}) is arbitrarily fixed, we then have:

$$
\inf_{k\in\mathbb{N}} d\big[ \operatorname{conv}\{\xi_i\}_{i\leq k}, \operatorname{conv}\{\xi_i\}_{i>k} \big] \geq \frac{\delta}{R\lambda^2}>0
$$
Since $\|\xi_i\|_V\leq 1$ for each $i\in\mathbb{N}$, by \textbf{Theorem \ref{Theorem 6.8}}, \textbf{(30)}, we can conclude $\overline{V}$ is not reflexive and hence $Y$ is not super-reflexive.

\end{proof}

\begin{prop}[Milman–Pettis theorem]\label{Proposition 7.13}

A uniformly convex (see \textbf{Definition \ref{Definition 1.1}}) Banach space $X$ is reflexive
    
\end{prop}

\begin{proof}

Pick $z\in X^{\ast\ast}$ with $\|z\|= 1$ and fix an arbitrary $\epsilon\in(0, 1)$. Let $\delta$ be given by uniform convexity. Pick $f\in X^*$ with $\|f\|\leq 1$ such that $\big\vert\, z(f) - \|z\|\, \big\vert < \dfrac{\delta}{2}$. Define:

$$
C = \left\{ x\in X_{\leq 1}: \vert\, f(x)-1\,\vert <\frac{\delta}{2}  \right\}
$$
Let $q:X\rightarrow X^{\ast\ast}$ be the canonical mapping and we claim that $z\in C_1$, the weak-$\ast$ closure of $q(C)$. Suppose $z\notin C_1$. Since $X^*$ is a total subspace of $X^{\ast\ast\ast}$, by \textbf{Theorem \ref{Theorem 1.25}}, the space of linear functionals that are defined on $X^{\ast\ast}$ and continuous with respect to weak-$\ast$ topology is $X^*$. Then by \textbf{Corollary \ref{Corollary 1.22}}, there exists $g\in X^*$ such that:

$$
z(g) < \inf\big\{ x(g): x\in C_1\big\}
$$
Meanwhile, by {\cite[Goldstine's Theorem]{22}}, since $q(X_{\leq1})$ is weak-$\ast$ dense in $X^{\ast\ast}_{\leq 1}$, there exists a net $\{x_{\lambda}\}\subseteq X_{\leq 1}$ such that $f(x_{\lambda}) \rightarrow z(f)$ for all $f\in X^*$. Since $\big\vert\, z(f)-1\, \big\vert<\dfrac{\delta}{2}$, there exists $\lambda_0$ such that $x_{\lambda}\in C$ for all $\lambda\geq \lambda_0$. Hence we have:

$$
z(g) < \inf_{\lambda\geq \lambda_0} g(x_{\lambda}) \leq \liminf_{\lambda} g(x_{\lambda}) = \lim_{\lambda} g(x_{\lambda}) = z(g)
$$
which is absurd. Therefore $z\in C_1$. Since $C$ is convex, for every $x_1, x_2\in C$ we have:

$$
\left\| \frac{x_1 + x_2}{2} \right\| \geq \left\vert\, f\left( \frac{x_1 + x_2}{2} \right) \,\right\vert > 1-\frac{\delta}{2} \hspace{0.3cm}\implies\hspace{0.3cm} \|x_1 + x_2\|>2-\delta
$$
By uniform convexity, we have $\|x_1 - x_2\|<\epsilon$. Since $x_1, x_2$ are arbitrarily picked, the diameter of $C_1$ is less than or equal to $\epsilon$. Hence there exists $x\in C$ such that $\|x-z\|<\epsilon$. Since $\epsilon$ is arbitrarily picked from $(0, 1)$, we can now conclude $z\in q(X_{\leq 1})$. Since $z$ is also arbitrarily picked, $X$ is reflexive.
    
\end{proof}

\begin{prop}\label{Proposition 7.14}
If a Banach space $(X, \|\cdot\|_X)$ is uniformly convex, then it would be super-reflexive.
\end{prop}

\begin{proof}

Assume $X$ is not super-reflexive and then by \textbf{Proposition \ref{Proposition 7.12}}, there exists a non-reflexive Banach space $(Y, \|\cdot\|_Y)$ that is finitely representable in $X$. Since $Y$ is not reflexive, by \textbf{Proposition \ref{Proposition 7.13}}, it would not be uniformly convex. Hence in $Y$, there exists $\epsilon > 0$ such that for each $\delta \in (0, 1)$, there exists $y_1, y_2 \in Y_{\leq 1}$ such that $\|y_1-y_2\|_Y > \epsilon$ and $\|y_1+y_2\| \geq 2-\delta$. Fix an arbitrary $\delta\in(0, \epsilon)$ such that $2\delta\in (0, 1)$. Find $y_1, y_2\in Y_{\leq 1}$ such that $\|y_1-y_2\|_Y>\epsilon$ and $\|y_1+y_2\|_Y \geq 2-\delta$. Define $V = \operatorname{Span}\{y_1, y_2\}$. Since $Y$ is finitely representable in $X$, there exists $T: V\rightarrow X$ such that for each $v\in V$:

$$
(1-\delta) \|v\|_Y \leq \|Tv\|_X \leq (1+\delta) \|v\|_Y
$$
Hence we have:

$$
\big\| T(y_1+y_2) \big\| \geq (1-\delta)(2-\delta), \hspace{2cm} \big\| T(y_1-y_2) \big\| >\epsilon(1-\delta)
$$
Without losing generality, assume $\delta$ is small enough such that:

$$
\frac{1-\delta}{1+\delta}(2-\delta) \geq 2-2\delta, \hspace{2cm} \frac{1-\delta}{1+\delta}\epsilon > \frac{\epsilon}{2}
$$
which implies that for any $\delta\in(0, 1)$:

$$
\left\| \frac{1}{1+\delta}Ty_1 + \frac{1}{1+\delta}Ty_2\right\| \geq 2-2\delta \hspace{2cm} \left\| \frac{1}{1+\delta}Ty_1 - \frac{1}{1+\delta}Ty_2 \right\| < \frac{\epsilon}{2}
$$
Since $\epsilon$ is fixed and $\delta$ is arbitrarily picked (and hence $2\delta$ is also arbitrarily picked from $(0, 1)$), we then can conclude $X$ is not uniformly convex. 

\end{proof}

\begin{lem}[\cite{23}]\label{Lemma 7.15}

In $X$ a LCTVS, given $K\subseteq X$ a separable weakly compact convex subset and $p$ a pseudo-norm given by the topology of $X$, for each $\epsilon\in (0, 1)$ there exists a closed convex subset $C\subseteq X$ such that $K\backslash C\neq \emptyset$ and $\big\vert\, p(x-y)\,\big\vert \leq \epsilon$ for all $x, y\in K\backslash C$.
    
\end{lem}

\begin{theorem}[{\cite[Theorem 2]{10}}]\label{Theorem 7.16}

A Banach space $B$ that has the finite-tree property is not reflexive.
    
\end{theorem}

\begin{proof}

Assume by contradiction that $B$ has the finite-tree property and is reflexive. Fix the $\epsilon>0$ given by the finite-tree property, an arbitrary $\delta\in (0, \epsilon)$, and $N\in\mathbb{N}$ with $N>1$. Let $F_N = \left( x^N_i: 1\leq i \leq 2^{N+1} \right)$ be the $(N, \epsilon)$-part tree given by the finite-tree property. Define:

$$
K_N = \operatorname{conv}(F_N), \hspace{1cm} F'_{N-1} = \left( \frac{x^N_{2i+1} + x^N_{2i+2}}{2}: 0\leq i \leq 2^N-1 \right), \hspace{1cm} K'_{N-1} = \operatorname{conv}\big( F'_{N-1} \big)
$$
and by definition, $F'_{N-1}$ is a $(N-1, \epsilon)$-part tree. Since $B$ is assumed to be reflexive, both $K_N$, $K'_{N-1}$ are weakly compact. According to \textbf{Lemma \ref{Lemma 7.15}}, there exists $C$ a closed convex set such that $K_N \backslash C\neq\emptyset$ and: 

$$
\begin{aligned}
& \hspace{0.9cm} \operatorname{diam}\big( K_N \backslash C \big) = \sup\big\{ \|x-y\|: x, y\in K'_{N_1}\backslash C \big\} < \dfrac{\epsilon}{2} \\
& \implies \operatorname{diam}\big( K'_{N-1} \backslash C \big) \leq \operatorname{diam}\big( K_N \backslash C \big) < \frac{\epsilon}{2}
\end{aligned}
$$
If $F'_{N-1}\subseteq C$, $K'_{N-1}\subseteq C$. In this case, for any $0\leq i \leq 2^N-1$:

$$
\frac{\epsilon}{2} > \operatorname{diam}\big( K_N \backslash C \big) \geq \left\| x^N_{2i+1} - \frac{x^N_{2i+1} + x^N_{2i+2}}{2} \right\| = \frac{1}{2}\big\| x^N_{2i+1} - x^N_{2i+2} \big\| \geq \frac{\epsilon}{2}
$$
Clearly we must have $F'_{N-1}\backslash C \neq\emptyset$. Suppose that for some $0\leq i \leq 2^N-1$:

$$
x'_i = \frac{x^N_{2i+1} + x^N_{2i+2}}{2} \notin C
$$
Then we either have $x^N_{2i+1}\notin C$ or $x^N_{2i+2}\notin C$. If both $x^N_{2i+1}$ and $x^N_{2i+2}\in K_N \backslash C$, we will then have:

$$
\frac{\epsilon}{2} > \operatorname{diam}\big( K_N\backslash C \big) \geq \big\| x^N_{2i+1} - x^N_{2i+2} \big\| \geq \epsilon
$$
which is absurd. Therefore, both $x^N_{2i+1}$ and $x^N_{2i+2}$ are in $C$, which implies that $x'_i\in C$ and contradicts that $x'_i\notin C$. Therefore $B$ cannot be reflexive.

\end{proof}

\begin{theorem}\label{Theorem 7.17}

If a Banach space $(B,\|\cdot\|)$ can be given a uniform convex norm that is equivalent to $\|\cdot\|$, then $B$ will not have the finite-tree property.
    
\end{theorem}

\begin{proof}

The conclusion follows immediately by \textbf{Proposition \ref{Proposition 7.13}} and \textbf{Theorem \ref{Theorem 7.16}}.
    
\end{proof}

\subsection{No finite tree property \texorpdfstring{$\,\Longrightarrow\,$}{implies} uniformly convexifiable}

After showing that a Banach space being super-reflexive is equivalent to being crudely representable, one can expect that, as in the proof of \textbf{Proposition \ref{Proposition 7.12}}, the proof of the direction in this section will also require defining an equivalent norm, and that is where we can see the \textbf{finite-tree property} again plays an important role. All results in this subsection are originally from \cite{5}.

\begin{lem}\label{Lemma 7.16}

If a Banach space $X$ does not have finite tree properties, then for every $\epsilon > 0$, we can find $\delta \in (0, \epsilon)$ so that for each $x\in X$ and each $(m,\epsilon)$-partition of $x$, $(x_1, x_2, \cdots, x_{2^m})$:

$$
\sum_{j \leq 2^m}\|x_j\|\geq(1+\delta)\|x\|
$$

\end{lem}

\begin{proof}
Fix $\|x\|=1$ and $(x_1, x_2, \cdots, x_{2^m})$ a $(m, \epsilon)$-partition of $x$. Consider the following $1$-part of the fixed $(m, \epsilon)$-partition of $x$:

$$
\big( x_1+x_2+\cdots+x_{2^{m-1}}, x_{2^{m-1}+1}+x_{2^{m-1}+2}+\cdots+x_{2^m} \big)
$$
By definition, we have $\|x_1+x_2+\cdots+x_{2^{m-1}}\|=\|x_{2^{m-1}+1}+x_{2^{m-1}+2}+\cdots+x_{2^m}\|$ and:

$$
\begin{aligned}
&\hspace{1.02cm} \big\| x_1+x_2+\cdots+x_{2^{m-1}} \big\| =  \big\| x_{2^{m-1}+1} + x_{2^{m-1}+2}+\cdots+x_{2^m} \big\| \geq \frac{1}{2}\\
&\implies\, \big\| (x_1+x_2+\cdots+x_{2^{m-1}})-(x_{2^{m-1}+1} + x_{2^{m-1}+2} + \cdots+x_{2^m}) \big\| \geq \frac{\epsilon}{2}\\
&\implies\, \big\| 2(x_1+x_2+\cdots+x_{2^{m-1}})-2(x_{2^{m-1}+1} + x_{2^{m-1}+2}+\cdots+x_{2^m}) \big\| \geq\epsilon
\end{aligned}
$$
Hence we obtain a $(1, \epsilon)$-part of a tree from the fixed $(m, \epsilon)$-partition of $x$. By induction, for all $k\leq m$, we can obtain a $(k, \epsilon)$-part of a tree by multiplying $2^k$ to each vector in the $k$-part of a given $(m, \epsilon)$-partition.\\

\noindent
Now suppose $X$ does not have the finite tree property. Fix $x\in X$ with $\|x\|=1$. Then for any $\epsilon\in(0, 1)$ there exists $N\in\mathbb{N}$ and $\delta\in(0, \epsilon)$ such that any $(N, \epsilon)$-part of a tree contains an element that has norm strictly greater than $1+\delta$. Next fix $\big( x_1, \cdots, x_{2^m} \big)$ an arbitrary $(m, \epsilon)$-partition of $x$. 

\begin{itemize}

    \item If $m=N$, then by the previous remark $\left( 2^Nx_1, 2^Nx_2, \cdots, 2^N x_{2^N} \right)$ is a $(N, \epsilon)$-part of a tree. By assumption, there exists $1\leq j \leq 2^N$ such that $\left\| 2^Nx_j \right\| \geq 1+\delta$. Since $\big( x_1, \cdots, x_{2^N} \big)$ is an $(N,\epsilon)$-partition of $x$, we have $\|x_i\|=\|x_j\|$ for all $1\leq i \leq 2^N$. Hence:

    $$
    \sum_{1\leq i \leq 2^N} \|x_i\| = 2^N\|x_j\| \geq 1+\delta = \big( 1+\delta \big) \|x\|
    $$

    \item If $m>N$, by definition the following is a $(N, \epsilon)$-partition of $x$:

    $$
    \left( \sum_{1\leq i \leq 2^{m-N}} x_i,\, \sum_{2^{m-N} < i \leq 2^{m-N+1}} x_i,\, \cdots,\, \sum_{2^m - 2^{m-N} < i \leq 2^m}x_i \right)
    $$

    By the previous remark, the following is a $(N, \epsilon)$-part of a tree:

    $$
    \left( 2^N\sum_{1\leq i \leq 2^{m-N}} x_i,\, 2^N\sum_{2^{m-N} < i \leq 2^{m-N+1}} x_i,\, \cdots,\, 2^N\sum_{2^m - 2^{m-N} < i \leq 2^m}x_i \right)
    $$

    Then by our assumption, there exists $1\leq j\leq 2^N$ such that:
    
    $$
    \left\| 2^N \sum_{(j-1)2^{m-N} < i \leq j2^{m-N}} x_i\right\| \geq 1+\delta
    $$

    and hence:

    $$
    \sum_{1\leq i \leq 2^m} \|x_i\| \geq \sum_{1\leq j \leq 2^N} \left\| \sum_{(j-1)2^{m-N} < i \leq j2^{m-N}} x_i\right\| = 2^N \left \| \sum_{(j-1)2^{m-N} < i \leq j2^{m-N}} x_i\right\| \geq 1+\delta
    $$

    \item If $m<N$, observe that for any $1\leq i < 2^m$ we have:

    $$
    \left\| \frac{x_i \,\slash\, 2}{\left\| x_i\,\slash\, 2\right\|} - \frac{x_{i+1} \,\slash\, 2}{\left\| x_{i+1} \,\slash\, 2\right\|} \right\| = \left\| \frac{x_i}{\|x_i\|} - \frac{x_{i+1}}{\|x_{i+1}\|} \right\| \geq \epsilon
    $$

    Hence the following $2^{m+1}$-tuple is a $(m+1, \epsilon)$-partition of $x$:

    \begin{equation}\label{e57}
    \left( \frac{x_1}{2},\, \frac{x_2}{2},\, \frac{x_1}{2},\, \frac{x_2}{2},\, \cdots,\, \frac{x_i}{2}, \frac{x_{i+1}}{2},\, \frac{x_i}{2},\, \frac{x_{i+1}}{2},\, \cdots,\, \frac{x_{2^m-1}}{2},\, \frac{x_{2^m}}{2},\, \frac{x_{2^m-1}}{2},\, \frac{x_{2^m}}{2} \right)
    \end{equation}

    Similarly, we can induce a $(m+2, \epsilon)$-partition of $x$ based on the $(m+1, \epsilon)$-partition in (\ref{e57}). Then by induction, suppose $(y_i)_{1\leq i \leq 2^N}$ is the $(N, \epsilon)$-partition of $x$ and is induced by the fixed $(m, \epsilon)$-partition of $x$ in the way of (\ref{e57}). Hence we have:
    
    $$
    \sum_{1\leq i \leq 2^m}\|x_i\| = \sum_{1\leq j\leq 2^N}\|y_j\|
    $$
    
    By assumption, there exists $1\leq j \leq 2^N$ such that $\|y_j\|\geq 1+\delta$. This immediately gives:

    $$
    \sum_{1\leq i \leq 2^m}\|x_i\| = \sum_{1\leq j \leq 2^N} \|y_j\| \geq 1+\delta = (1+\delta)\|x\|
    $$
    
\end{itemize}

\end{proof}

\begin{defn}

In a Banach space $X$, a real-valued function $f$ is said to be an \textit{ecart} if $f$ is non-negative, $f^{-1} \big(\{0\} \big) = \{0\}$, and for all $x\in X$ and $\alpha\in \mathbb{R}$, $f(\alpha x) = \vert\, \alpha\,\vert f(x)$.

\end{defn}

\begin{lem}\label{Lemma 7.18}

Suppose $X$ is a Banach space without finite tree property. Then for any $n\in\mathbb{N}$ and $\epsilon\in(0, 1)$, there exists an \textit{ecart} $f_n$ and $\delta_n > 0$ such that for all $x, y\in X$:

\begin{enumerate}[label = (\roman*)]

    \item $\left( 1 - \dfrac{\delta}{2^n} \right) \|x\| < f_n(x) < \left( 1 - \dfrac{1}{3} \dfrac{\delta}{2^n} \right)\|x\|$.

    \item $f_n (x+y) \leq f_n(x) + f_n(y)$
    
    \item if $\|x\|=\|y\|=1$ and $\|x-y\|\geq \dfrac{\epsilon}{2^n}$, then $f_n(x+y) < f_n(x) + f_n(y)-\delta_n$.
    
\end{enumerate}
where $\delta\in (0, \epsilon)$ is given by \textbf{Lemma \ref{Lemma 7.16}}.

\end{lem}

\begin{proof}

Fix an arbitrary $n\in\mathbb{N}$ and $\epsilon\in (0, 1)$. Let $\delta\in (0, \epsilon)$ be given by \textbf{Lemma \ref{Lemma 7.16}}. Since $\epsilon\in (0, 1)$ can be arbitrarily small and $\delta$ is always strictly smaller than $\epsilon$, it suffices to find an \textit{ecart} $f_n$ and $\delta_n>0$ such that $(1-\delta) \|x\| < f_n(x) < \left( 1-\dfrac{\delta}{3} \right)\|x\|$ for all $x\in X$ and, whenever $\|x\| = \|y\| = 1$ and $\|x-y\|\geq \epsilon$, $f_n(x+y) < f_n(x) + f_n(y) - \delta_n$. Then, by the same method, we can replace $\epsilon$ and $\delta$ by $\dfrac{\epsilon}{2^n}$ and $\dfrac{\delta}{2^n}$.

\begin{enumerate}[label = (\roman*)]

	\item For each $x\in X$ and $k\in\mathbb{N} \cup \{0\}$, let $\mathcal{P}(x, k, \epsilon)$ be the family of all $(k, \epsilon)$-partition of $x$ (i.e. $\mathcal{P}(x, k, \epsilon)$ is a family of $2^k$-tuple of points from $X$). Then define a function $f: X\rightarrow \mathbb{R}$ as the following: for each $x\in X$:

	$$
	f_n(x) = \inf\left\{\sum_{1\leq i \leq 2^m} \|x_i\| \left( 1+ \frac{\delta}{2} \sum_{0\leq j \leq m}\frac{1}{4^j} \right)^{-1}: \big(x_1, \cdots, x_{2^m} \big)\in \bigcup_{0\leq k \leq n} \mathcal{P}(x, k, \epsilon) \right\}
	$$
	Clearly $f$ is an \textit{ecart}. For each $x\in X$, the $1$-tuple $(x)$ is viewed as the $(0,\epsilon)$-partition of $x$. Since $\delta\in(0, 1)$, we have $\left(1+\dfrac{\delta}{2} \right)\, \left(1- \dfrac{\delta}{3} \right) = 1+\dfrac{1}{6}( \delta-\delta^2) > 1$. Hence:

	$$
	f(x)\leq \frac{\|x\|}{1+\dfrac{\delta}{2}} < \left( 1 - \dfrac{\delta}{3} \right)\|x\| 
	$$
	Meanwhile, since $\delta\in(0, 1)$, clearly we have $(1+\delta) > (1-\delta)\, \left(1+ \frac{2}{3}\delta\right)$. Then by \textbf{Lemma \ref{Lemma 7.16}} we have that for all $x\in X$:

	$$
	f(x) \geq \frac{(1+\delta)\|x\|}{1+ \dfrac{2\delta}{3}} > (1-\delta)\|x\|
	$$
	
	\item Fix $k, l\in\mathbb{N} \cup \{0\}$, and $\big(x_1, \cdots, x_{2^k} \big)\in \mathcal{P}(x, k, \epsilon)$, $\big( y_1, \cdots, y_{2^l} \big) \in \mathcal{P}(y, l, \epsilon)$. Without losing generality, assume $k\geq l$. Then there exists $\alpha\in (0, 1)$ such that:

    \begin{equation}\label{e58}
    \alpha \big( 2^k-2^l \big) \left( 1+\frac{\delta}{2} \sum_{0\leq j \leq k+1} \frac{1}{4^j} \right) < \frac{1}{4^{k+1}} \frac{\delta(1+\delta)}{2}
    \end{equation}
    Define the following $2^{k+1}$-tuple:

    $$
    \left( \underbracket{x_1,\, x_2,\, \cdots,\, x_{2^k}}_{1\, \sim\, 2^k} ,\, \underbracket{y_1,\, y_2,\, \cdots,\, y_{2^l}}_{2^k+1\, \sim\, 2^k+2^l} ,\, \underbracket{\alpha x,\, -\alpha x,\, \alpha x,\, -\alpha x,\, \cdots,\, \alpha x,\, -\alpha x}_{2^k+2^l+1\, \sim\, 2^{k+1}} \right)
    $$
    and clearly the $2^{k+1}$-tuple above is a $(k+1, \epsilon)$-partition of $x+y$. Therefore by (\ref{e58}) and \textbf{Lemma \ref{Lemma 7.16}}, we have:

    $$
    \begin{aligned}
    & \hspace{0.85cm} \left( \sum_{1\leq i \leq 2^k} \|x_i\| \right) \left( 1+\frac{\delta}{2} \sum_{0\leq j \leq k} \frac{1}{4^j} \right)^{-1} + \left( \sum_{1\leq i \leq 2^l} \|y_i\| \right) \left( 1+\frac{\delta}{2} \sum_{0 \leq j \leq l} \frac{1}{4^j} \right)^{-1}\\
    & \hspace{0.5cm} - \left( \sum_{1\leq i \leq 2^k}\|x_i\| + \sum_{1\leq i \leq 2^l} \|y_i\| + \alpha \big(2^k-2^l \big)\|x\| \right) \left(1+ \frac{\delta}{2} \sum_{0\leq j \leq k+1} \frac{1}{4^j} \right)^{-1}\\
    & > \left( \left( \frac{\delta}{2} \frac{1}{4^{k+1}} \sum_{1\leq i \leq 2^k}\|x_i\| \right) - \alpha \left( 1+\frac{\delta}{2} \sum_{0\leq j \leq k} \frac{1}{4^j} \right) (2^k-2^l) \|x\| \right)\left( 1+\frac{\delta}{2} \sum_{0\leq j \leq k+1} \frac{1}{4^j} \right)^{-2}\\
    & \geq \|x\| \left( \frac{1}{4^{k+1}} \frac{\delta(1+\delta)}{2} - \alpha \left( 1+\frac{\delta}{2} \sum_{0\leq j \leq k} \frac{1}{4^j} \right) (2^k-2^l) \right)\left( 1+\frac{\delta}{2} \sum_{0\leq j \leq k+1} \frac{1}{4^j} \right)^{-2} > 0
    \end{aligned}
    $$
    Since $\big(x_1, \cdots, x_{2^k} \big)\in \mathcal{P}(x, k, \epsilon)$ and $\big( y_1, \cdots, y_{2^l} \big)\in \mathcal{P}(y, l, \epsilon)$ are arbitrarily picked, we then can conclude that $f_n (x+y) \leq f_n(x) + f_n(y)$.
	
	\item Now fix $\|x\|=\|y\|=1$ and $\|x-y\|\geq\epsilon$. Then we can see $(x, y)$ is a $(1, \epsilon)$-partition of $x+y$. Fix an arbitrary $\gamma\in \left(0, \dfrac{1}{8^n} \right)$. Then by definition of $f$, there exists $k, l\in \{0, 1, \cdots, n\}$, and $\big( x_1, \cdots, x_{2^k} \big) \in \mathcal{P}(x, k, \epsilon)$, $\big(y_1, \cdots, y_{2^l} \big) \in \mathcal{P}(y, l, \epsilon)$ such that:

	\begin{equation}\label{e59}
	f(x) + \frac{\gamma}{2} > \sum_{1\leq i \leq 2^k} \|x_i\| \left( 1+ \frac{\delta}{2} \sum_{0\leq j \leq k}\frac{1}{4^j} \right)^{-1},
	\hspace{1cm}
	f(y) + \frac{\gamma}{2} > \sum_{1\leq i \leq 2^l} \|y_i\| \left( 1+ \frac{\delta}{2} \sum_{0 \leq j \leq l}\frac{1}{4^j} \right)^{-1}
	\end{equation}
	Without losing generality, suppose $l\geq k$. Then let $\big(w_1, \cdots, w_{2^k} \big)$ be the $k$-part of $(y_1, \cdots, y_{2^l} \big)$. Then by (\ref{e59}) we have:

	\begin{equation}\label{e60}
	f(y) + \frac{\gamma}{2} \geq \sum_{1\leq i \leq 2^k} \|w_i\| \left( 1+ \frac{\delta}{2} \sum_{0 \leq j \leq l}\frac{1}{4^j} \right)^{-1} \geq \sum_{1\leq i \leq 2^k} \|w_i\| \left( 1+ \frac{\delta}{2}\left( \frac{1}{3}\, \frac{1}{4^k} + \sum_{0 \leq j \leq k}\frac{1}{4^j} \right) \right)^{-1}
	\end{equation}
	Observe that $(x_1, x_2, \cdots, x_{2^k}, w_1, w_2, \cdots, w_{2^k})$ is a $(k+1, \epsilon)$-partition of $x+y$. Hence together with (\ref{e59}) and (\ref{e60}), we have:

	\begin{equation}\label{e61}
	\begin{aligned}
	f(x)+f(y)+\gamma - f(x+y)
	& \geq \left( \sum_{1\leq i \leq 2^k} \|x_i\| \right)\, \left( \left( 1+\frac{\delta}{2} \sum_{0 \leq j \leq k} \frac{1}{4^j} \right)^{-1} - \left( 1+\frac{\delta}{2} \sum_{0 \leq j \leq k+1} \frac{1}{4^j} \right)^{-1} \right)\\
	& + \left( \sum_{1\leq i \leq 2^k} \|w_i\| \right)\, \left( \left( 1+ \frac{\delta}{2}\left( \frac{1}{3} \frac{1}{4^k} + \sum_{0 \leq j \leq k} \frac{1}{4^j} \right) \right)^{-1} - \left( 1+ \frac{\delta}{2} \sum_{0 \leq j \leq k+1} \frac{1}{4^j} \right)^{-1} \right)\\
	& \geq \frac{\dfrac{\delta}{2} \dfrac{1}{4^{k+1}}}{\left( 1+\dfrac{\delta}{6} \right)^2} + \frac{\dfrac{\delta}{2} \dfrac{1}{4^{k+1}} - \dfrac{\delta}{6} \dfrac{1}{4^k}}{\left( 1+\dfrac{\delta}{6} \right)^2}\\
	& > \frac{\dfrac{\delta}{2} \dfrac{1}{4^{k+1}}}{\left( 1+\dfrac{\delta}{6} \right)^2} - \frac{\dfrac{\delta}{2} \dfrac{1}{4^{k+1}} - \dfrac{\delta}{6} \dfrac{1}{4^k}}{\left( 1+\dfrac{\delta}{6} \right)^2} = \frac{1}{4^k}\frac{\delta}{6} \left(1+\frac{\delta}{6} \right)^{-2}
	\end{aligned}
	\end{equation}
	Since $\gamma\in \left(0, \dfrac{1}{8^n} \right)$, we have:

	\begin{equation}\label{e62}
	\frac{1}{4^k}\frac{\delta}{6} \left( 1+\frac{\delta}{6} \right)^{-2} - \gamma > \underbracket{\frac{1}{4^n}\left( \frac{\delta}{6} \left(1+\frac{\delta}{6} \right)^{-2} - \frac{1}{2^n} \right)}_{\delta_n}
	\end{equation}
	and $\delta_n$ is defined only depending on the fixed integer $n$. Together with (\ref{e61}) and (\ref{e62}) we can now conclude $f(x+y) > f(x) + f(y) - \delta_n$ for all $x,y\in X$ with $\|x\|=\|y\|=1$ and $\|x-y\|\geq \epsilon$.

\end{enumerate}

\end{proof}

\begin{lem}\label{Lemma 7.19}

Let  $(X, \|\cdot\|)$ be a Banach space without finite tree property. Fix an arbitrary $\epsilon\in (0, 1)$ and $n\in\mathbb{N}$. Let $\delta\in (0, \epsilon)$ be given by \textbf{Lemma \ref{Lemma 7.16}}. Let $f_n$ be the \textit{ecart} that is given by \textbf{Lemma \ref{Lemma 7.18}} and corresponds to $\epsilon$. Let $\delta_n>0$ be given by \textbf{Lemma \ref{Lemma 7.18}} and correspond to $\epsilon$. Then there exists a norm $\|\, \cdot\,\|_{n, \epsilon}$ on $X$ such that for all $x, y\in X$:

\begin{enumerate}[label = (\roman*)]

    \item $\left( 1 - \dfrac{\delta}{2^n} \right) \|x\| < \|x\|_{n, \epsilon} < \left( 1-\dfrac{1}{3} \dfrac{\delta}{2^n} \right)\|x\|$
    
    \item if $\|x\|=\|y\|=1$ and $\|x-y\|\geq \dfrac{3}{2^n} \epsilon$, then $\|x+y\|_{n,\epsilon} \leq \|x\|_{n,\epsilon} + \|y\|_{n,\epsilon} - \delta_n$
    
\end{enumerate}

\end{lem}

\begin{proof}

Similar to the proof of \textbf{Lemma \ref{Lemma 7.18}}, since $\epsilon$ can be arbitrarily small and so is $\delta$, it suffices to find a new norm $\|\cdot\|_{n, \epsilon}$ such that $(1-\delta) \|x\| < \|x\|_{n, \epsilon} < \left( 1 - \dfrac{\delta}{3} \right) \|x\|$ and, whenever $\|x\|=\|y\|=1$ and $\|x-y\|\geq 3\epsilon$, $\|x+y\|_{n, \epsilon} \leq \|x\|_{n, \epsilon} + \|y\|_{n, \epsilon} - \delta_n$. Then by the same method, we can replace $\epsilon$ and $\delta$ by $\dfrac{\epsilon}{2^n}$ and $\dfrac{\delta}{2^n}$.\\

\noindent
For each $x\in X$ and for each $n\in\mathbb{N}$ with $n\geq 2$, with the fixed $\epsilon\in (0, 1)$, define:

$$
\mathcal{F}(n, x) = \left\{ \big( x_0, x_1, \cdots, x_n \big)\in X^{n+1}: x_0=0 ,\, x_n = x,\, \|x_{n-1}\| = \|x\|,\, \|x_{n-1} - x\|< \epsilon \right\}
$$
and:

$$
\|x\|_{n,\epsilon} = \inf_{\substack{n\in \mathbb{N} \\  n\geq 2}} \left\{ \sum_{0\leq i < n}f_n(x_{i+1} - x_i): \big(x_0, x_1, \cdots, x_n \big)\in \mathcal{F}(n, x) \right\}
$$
Fix an arbitrary $x\in X$. Since $(0, x, x)\in \mathcal{F}(2, x)$, $\|x\|_{\epsilon}\leq f_n(x)$. Meanwhile, for each $n\in\mathbb{N}$ and $\big(x_0,\, x_1,\, \cdots,\, x_n \big)\in \mathcal{F}(n, x)$, by \textbf{Lemma \ref{Lemma 7.18}}:

$$
f_n(x) = f_n \left( \sum_{0\leq i < n}(x_{i+1} - x_i) \right) \leq \sum_{0\leq i < n} f_n(x_{i+1} - x_i)
$$
we then have $f_n(x) = \|x\|_{n,\epsilon}$. By \textbf{Lemma \ref{Lemma 7.18}} and that $f_n$ is an \textit{ecart}, $\|\cdot\|_{n,\epsilon}$ defines a norm on $X$ and for each $x\in X$, $(1-\delta)\|x\| < \|x\|_{n,\epsilon} < \left(1-\dfrac{\delta}{3} \right)\|x\|$. Next fix $x, y\in X$ with $\|x\|= \|y\|=1$ and $\|x-y\|\geq 3\epsilon$. Fix an arbitrary $(x_0, x_1, x)\in \mathcal{F}(2, x)$ and, for an arbitrary $m\in\mathbb{N}$ with $m\geq 2$, fix $\big( y_0,\, y_1,\, \cdots,\, y_m \big)\in \mathcal{F}(m, y)$. Then we have:

$$
3\epsilon\leq \|x-y\| \leq \|x-x_1\| + \|x_1 - y_{m-1}\| + \|y-y_{m-1}\| < 2\epsilon + \|x_1 - y_{m-1}\|
$$
and hence $\|x_1-y_{m-1}\|\geq \epsilon$. By definition we have $\|x_1\| = \|x\| = \|y\| = \|y_{m-1}\|$, and hence by \textbf{Lemma \ref{Lemma 7.18}} we have:

\begin{equation}\label{e63}
\begin{aligned}
f(x+y_{m-1}) + f(y-y_{m-1})
& \leq f(x-x_1) + f(x_1+y_{m-1}) + f(y - y_{m-1})\\
& \leq f(x_1) + f(x-x_1) + f\left( \sum_{0\leq i < m-1} y_{i+1} - y_i \right) + f(y-y_{m-1}) -\delta_n\\
& \leq f(x_1) + f(x-x_1) + \sum_{0\leq i < m} f(y_{i+1} - y_i) - \delta_n
\end{aligned}
\end{equation}
Since $\big( 0, x+y_{m-1}, x+y \big)\in \mathcal{F}(2, x+y)$, and $m\in\mathbb{N}\, (m\geq 2)$, $\big( y_0,\, y_1,\, \cdots,\, y_m \big)\in \mathcal{F}(m, y)$ are arbitrarily picked, by (\ref{e63}) we can conclude:

\begin{equation}\label{e64}
\|x+y\|_{n,\epsilon} \leq f(x_1) + f(x-x_1) + \|y\|_{n,\epsilon} - \delta_n
\end{equation}
Therefore (\ref{e64}) proves the base case for $\mathcal{F}(k, x)$ when $k=2$. Assume the $k=l$ case for some $l \in \mathbb{N}$ and $l\geq 2$. Then for $l+1$, fix an arbitrary $\big( x_0,\, x_1,\, \cdots,\, x_{n+1} \big) \in \mathcal{F}(l+1, x)$. We then have $\big( x_0,\, x_2,\, \cdots,\, x_{l+1} \big) \in \mathcal{F}(l, x)$ and by our assumption and \textbf{Lemma \ref{Lemma 7.18}}:

$$
\begin{aligned}
\|x+y\|_{n,\epsilon}
& \leq f(x_2) + \sum_{2\leq i < l+1} f(x_{i+1} - x_i) + \|y\|_{n,\epsilon} - \delta_n\\
& \leq f(x_1) + f(x_2-x_1) + \sum_{2\leq i < l+1} f(x_{i+1} - x_i) + \|y\|_{n,\epsilon} - \delta_n\\
& \leq \sum_{0\leq i < l+1} f(x_{i+1} - x_i) + \|y\|_{n,\epsilon} - \delta_n
\end{aligned}
$$
We can now conclude that $\|x+y\|_{n,\epsilon} \leq \|x\|_{n,\epsilon} + \|y\|_{n,\epsilon} - \delta_n$.

\end{proof}

\begin{lem}\label{Lemma 7.20}

Let $X$ be a normed linear space endowed with two equivalent norms, $\|\cdot\|_1$ and $\|\cdot\|_2$, such that for each $x\in X$, $\|x\|_2 \leq \|x\|_1 \leq 2\|x\|_2$. Suppose for each $x, y\in X$ and $\epsilon\in (0, 1)$, whenever $\|x\|_1 = \|y\|_1 = 1$ and $\|x-y\|_1 \geq \epsilon$, there exists $\delta(\epsilon)\in (0, 1)$ such that $\|x+y\|_2 \leq \|x\|_2 + \|y\|_2 - \delta(\epsilon)$. Then $\|\cdot\|_2$ is uniformly convex.

\end{lem}

\begin{proof}

Fix an arbitrary $\epsilon\in (0, 1)$. Pick $x, y\in X$ with $\|x\|=\|y\|=1$ and $\|x-y\|_1 \geq \epsilon$. Then for any $\lambda \in (0, 1)$:

$$
\|\lambda x-y\|_1 \geq \big\vert\, \|\lambda x + (1-\lambda)x\|_1 - \|(1-\lambda)x - y\|_1 \,\big\vert \geq \big\vert\, \|x\|_1 - \|x-y\|_1 - \lambda\,\big\vert = \big\vert\, 1-\lambda - \|x-y\|_1 \,\big\vert
$$
Since both $x, y\in X$ are arbitrarily picked, we then can conclude:

\begin{equation}\label{e65}
\forall\,x, y\in X, \hspace{0.8cm} \big( \|x\|_1 = \|y\|_1 = 1,\, \|x-y\|_1\geq \epsilon \big)
\hspace{0.3cm} \Longrightarrow \hspace{0.3cm}
\inf_{\lambda \in (0, 1)} \|\lambda x-y\|_1 \geq \epsilon
\end{equation}
Next pick $x', y'\in X$  with $\|x'\|_2 = \|y'\|_2 = 1$ and $\|x' - y'\|_2 \geq \epsilon$. Then there exists $\alpha, \beta\in \left[ \dfrac{1}{2}, 1\right]$ such that $\alpha \|x'\|_1 = \beta \|y'\|_1 = 1$. Without losing generality, suppose $\beta \geq \alpha$. Since $\|x' - y'\|_1\geq \|x' - y'\|_2 \geq \epsilon$, by (\ref{e65}) we have:

$$
2\big\| \alpha x' - \beta y' \big\|_2 \geq \big\| \alpha x' - \beta y' \big\|_1 \geq \inf_{\lambda\in (0, 1)} \big\| \lambda\alpha x' - \beta y' \big\|_1 \geq \inf_{\lambda\in (0, 1)} \big\| \lambda x' - \beta y' \big\|_1\geq \epsilon
\hspace{0.4cm} \Longrightarrow \hspace{0.4cm}
\big\| \alpha x' - \beta y' \big\|_2 \geq \frac{\epsilon}{2}
$$
and hence $\big\| \alpha x' - \beta y' \big\|_1 \geq \dfrac{\epsilon}{2}$. By our assumption on $\delta$, we then have:

$$
\begin{aligned}
& \hspace{0.96cm} \big\| \alpha x' - \beta y' \big\|_2 \leq \alpha\|x'\|_2 + \beta\|y'\|_2 = \delta\left( \frac{\epsilon}{2} \right)\\
& \implies \|x' + y'\|_2 \leq \big\| \alpha x' + \beta y' \big\|_2 + \big\| (1-\alpha)x' + (1-\beta) y' \big\|_2 \leq 2 - \delta\left( \frac{\epsilon}{2}\right)
\end{aligned}
$$
Since $\epsilon$ is arbitrarily picked, we then can conclude $\|\cdot\|_2$ is uniformly convex.

\end{proof}

\begin{theorem}\label{Theorem 7.21}

A Banach space $(X, \|\cdot\|)$ without finite tree property can be given a uniformly convex norm that is equivalent to $\|\cdot\|$.
    
\end{theorem}

\begin{proof}

Fix an arbitrary $\epsilon\in \left(0, \dfrac{1}{2} \right)$. Let $\delta\in (0,\epsilon)$ be given by \textbf{Lemma \ref{Lemma 7.16}}. For each $n\in\mathbb{N}$, let $\|\cdot\|_{n, \epsilon}$ be the norm given by \textbf{Lemma \ref{Lemma 7.19}} and $\delta_n$ be given by \textbf{Lemma \ref{Lemma 7.19}}. Then for each $x\in X$, define:

$$
\|x\|_u = \sum_{n\in \mathbb{N}} \frac{1}{2^n} \|x\|_{n, \epsilon}
$$
By \textbf{Lemma \ref{Lemma 7.19}}, clearly $\|\cdot\|_u$ is a norm, and for each $x\in X$:

\begin{equation}\label{e66}
\begin{aligned}
& \hspace{1cm} (1-\epsilon)\|x\| = \|x\| \sum_{n\in \mathbb{N}} \frac{1}{2^n}(1-\epsilon) < \|x\|\sum_{n\in \mathbb{N}} \frac{1}{2^n} \left( 1 - \frac{\delta}{2^n} \right) < \|x\|_u < \|x\| \sum_{n\in \mathbb{N}} \frac{1}{2^n} \left( 1 - \frac{1}{3} \frac{\delta}{2^n} \right) < \|x\| \\
& \implies\, \|x\|_u < \|x\| < \frac{1}{1-\epsilon}\|x\|_u < 2\|x\|_u
\end{aligned}
\end{equation}
We will prove that $\|\cdot\|_u$ is uniformly convex using \textbf{Lemma \ref{Lemma 7.20}}. Fix an arbitrary $k\in \mathbb{N}$. Pick $x, y\in X$ with $\|x\| = \|y\| = 1$ and suppose $\|x-y\| \geq \dfrac{3\epsilon}{2^k}$. Then for all $n\in\mathbb{N}$ with $n\geq k$, $\|x-y\| \geq \dfrac{3\epsilon}{2^k}$ and hence by \textbf{Lemma \ref{Lemma 7.19}}:

\begin{equation}\label{e67}
\begin{aligned}
\|x+y\|_u
& \leq \sum_{1\leq i < k}\frac{1}{2^i}\big( \|x\|_{i, \epsilon} + \|y\|_{i, \epsilon} \big) + \sum_{i\geq k} \frac{1}{2^i} \|x+y\|_{i, \epsilon}\\
& \leq \sum_{1\leq i < k} \frac{1}{2^i} \big( \|x\|_{i, \epsilon} + \|y\|_{i, \epsilon} \big) + \sum_{i\geq k} \frac{1}{2^i} \big( \|x\|_{i, \epsilon} + \|y\|_{i, \epsilon} - \delta_i \big)\\
& = \|x\|_u + \|y\|_u - \sum_{i\geq k} \frac{1}{2^i} \delta_i
\end{aligned}
\end{equation}
Therefore, for any $x, y\in X$ with $\|x\|=\|y\|=1$ and $\|x-y\|\geq \lambda$ for some $\lambda\in (0, 1)$, there exists $N\in\mathbb{N}$ such that $\lambda > \dfrac{\epsilon}{2^N}$. By (\ref{e67}), we have $\|x+y\|_u \leq \|x\|_u + \|y\|_u - \sum_{i\geq N} \dfrac{1}{2^N}$, and together with (\ref{e66}) and \textbf{Lemma \ref{Lemma 7.20}}, $\|\cdot\|_u$ is uniformly convex.

\end{proof}

\begin{cor}\label{Corollary 7.24}

A Banach space $\big( X, \|\cdot\| \big)$ can be given a uniformly convex norm that is equivalent to $\|\cdot\|$ if and only if $X$ does not have the finite-tree property.
    
\end{cor}

\begin{proof}

The conclusion follows immediately by \textbf{Theorem \ref{Theorem 7.21}} and \textbf{Theorem \ref{Theorem 7.17}}.
    
\end{proof}

\begin{cor}\label{Corollary 7.25}
A Banach space $(X, \|\cdot\|)$ is super-reflexive iff $X$ does not have the finite-tree property.
\end{cor}

\begin{proof}

By \textbf{Theorem \ref{Theorem 7.21}} and \textbf{Proposition \ref{Proposition 7.14}}, when $X$ does not possess the finite-tree property, it is super-reflexive. When $X$ is super-reflexive, clearly $X$ can be crudely finitely representable in $X$ itself and hence $X$ is reflexive. Then the conclusion follows by \textbf{Theorem \ref{Theorem 7.17}}.

\end{proof}

\subsection{Characterizations by von Neumann constant}

\begin{defn}

Given a Banach space $(X, \|\cdot\|_X)$, define:

$$
\tilde{CJ}(X) = \inf\big\{ CJ(X, \|\cdot\|): \|\cdot\|\,\text{  is equivalent to  }\,\|\cdot\|_X \big\}
$$
    
\end{defn}

\begin{theorem}[{\cite[Theorem 2]{24}}]\label{Theorem 7.27}

If a Banach space $X$ is uniform convex, then $CJ(X)<2$.
    
\end{theorem}

\begin{proof}

Pick two different non-zero $x, y\in X$ such that $\|x\|^2 + \|y\|^2=1$. Fix an arbitrary $\epsilon \in (0, 1)$. If $\|x-y\|\geq \epsilon$, then according to \textbf{Proposition \ref{Proposition 1.3}}, there exists $\delta\in (0, 1)$ such that:

$$
\|x+y\| < 2(1-\delta)\max\big( \|x\|, \|y\| \big)
$$
which implies:

\begin{equation}\label{e68}
\begin{aligned}
& \hspace{0.95cm} \left\| \frac{x+y}{2} \right\|^2 + \left\| \frac{x-y}{2} \right\|^2 \leq \frac{(1-\delta)\max\big( \|x\|, \|y\| \big)}{2} + \frac{\|x\|^2 + \|y\|^2}{2} \leq 1-\frac{\delta}{2}\\
& \implies\, \|x+y\|^2 + \|x-y\|^2 \leq 2(2-\delta) 
\end{aligned}
\end{equation}
If $\|x-y\| < \epsilon$, we will have:

\begin{equation}\label{e69}
\|x+y\|^2 + \|x-y\|^2 \leq 2\big( \|x\|^2 + \|y\|^2 \big) + \epsilon^2 \leq 2\left( 1+\frac{\epsilon^2}{2} \right)
\end{equation}
Combining (\ref{e68}) and (\ref{e69}) gives us:

$$
\frac{\|x+y\|^2 + \|x-y\|^2}{2\big( \|x\|^2 + \|y\|^2 \big)} \leq 1+\max\left( 1-\delta, \frac{\epsilon^2}{2} \right) < 2
$$
which implies that $CJ(X, \|\cdot\|)<2$
    
\end{proof}

\begin{defn}

In a Banach space $X$, given $(g_j)_{j\in \mathbb{N}} \subseteq X^*_{\leq 1}$, for each $n\in \mathbb{N}$, define:

$$
S\big( \{p_i\}_{1\leq i \leq 2n}; (g_j)_{j\in \mathbb{N}} \big) = \left\{x\in X:\forall\,1\leq i \leq 2n,\, \forall\,k\in \big\{p_{2i-1}+1, \cdots, p_{2i} \big\}, \hspace{0.2cm} \frac{3}{4} \leq (-1)^{i-1}g_k(x) \leq 1 \right\}
$$
where $(p_i)_{1\leq i \leq 2n}$ is a strictly increasing sequence of integers. For each $n\in \mathbb{N}$, define:

$$
K\big( n, (g_j)_{j\in \mathbb{N}} \big) = \liminf_{p_1\rightarrow \infty}\,\cdots\, \liminf_{p_{2n} \rightarrow \infty}\inf\left\{ \|z\|: z\in S\big( \{p_i\}_{1\leq i \leq 2n}; (g_j)_{j\in \mathbb{N}} \big) \right\}
$$
and:

$$
K_n(X) = \inf\left\{ K\big( n, (g_j)_{j\in \mathbb{N}} \big): (g_j)_{j\in \mathbb{N}} \subseteq X^*_{\leq 1} \right\}
$$

\end{defn}

\begin{lem}[{\cite[Lemma III.$\S$I.2]{25}}]\label{Lemma 7.30}

Given a Banach space $X$, we have $K_n(X)\leq 2n$ for all $n\in\mathbb{N}$ if $X$ is not reflexive.
    
\end{lem}

\begin{proof}

When $X$ is not reflexive, fix $\theta\in \left( \dfrac{3}{4}, 1\right)$ and let $(g_n)_{n\in\mathbb{N}} \subseteq X^*_{\leq 1}$, $(x_i)_{i\in \mathbb{N}}\subseteq X_{\leq 1}$ be given by \textbf{Theorem \ref{Theorem 3.2}}. Fix $n\in\mathbb{N}$ and $(p_i)_{i\leq 2n}$ a strictly increasing sequence of integers. Define:

$$
w_n = \sum_{j\leq n} (-1)^{j-1} \big( x_{p_{2j-1}} + x_{p_{2j}} \big)
$$
Therefore for all $1\leq i\leq n$ and $k\in \big\{ p_{2i-1}+1, p_{2i-1}+2, \cdots, p_{2i} \big\}$. According to \textbf{Theorem \ref{Theorem 3.2}}, for all $1\leq j \leq 2n$ we have:

$$
f_k\big( x_{p_{2j-1}} + x_{p_{2j}} \big) = 
\begin{cases}
0, \hspace{1cm}j\neq i\\
\theta, \hspace{1cm}j=i
\end{cases}
\hspace{0.5cm} \Longrightarrow \hspace{0.5cm}
f_k(w_n) = (-1)^{i-1}\theta
$$
which implies $w_n\in S\big( (p_i)_{1\leq i \leq 2n}; (f_j)_{j\in \mathbb{N}} \big)$. Clearly $\|w_n\|\leq 2n$ and hence $K_n(X)\leq 2n$.
    
\end{proof}

\begin{prop}\cite{25}\label{Proposition 7.31}
If a Banach space $X$ is uniformly non-square, then it is reflexive.
\end{prop}

\begin{proof}

Assume that $X$ is not reflexive. Then according to \textbf{Lemma \ref{Lemma 7.30}}, we have $K_n(X)\leq 2n$ for all $n\in \mathbb{N}$. Clearly $K_n(X)$ increases as $n$ increases. For any $n\in \mathbb{N}$ and $(g_j)_{j\in \mathbb{N}}\subseteq X^*_{\leq 1}$, $\{p_i\}_{i\leq 2n}\subseteq \mathbb{N}$ that is strictly increasing, if $x\in S\big( \{p_i\}_{i\leq 2n}, (g_j)_{j\in \mathbb{N}} \big)$, then clearly $\|x\|\geq \dfrac{3}{4}$ and hence $K_n(X) \geq \dfrac{3}{4}$ for all $n\in \mathbb{N}$.\\

\noindent
Fix $\delta\in (0, 1)$ and $r\in (1-\delta, 1)$. Therefore, for each $n\in \mathbb{N}$:

\begin{equation}\label{e70}
\begin{aligned}
& \hspace{1cm} \forall\, \epsilon\in \left(0, \frac{3}{4}\frac{1-r}{2r+1} \right), \hspace{0.3cm} K_n(X) \geq \frac{3}{4} > \frac{(2r+1)\epsilon}{1-r} \\
& \implies\, \forall\, \epsilon\in \left(0, \frac{3}{4}\frac{1-r}{2r+1} \right), \hspace{0.3cm} \frac{K_n(X)-\epsilon}{K_n(X) + 2\epsilon} > r
\end{aligned}
\end{equation}
From $K_n(X)\leq 2n$ for all $n\in \mathbb{N}$, we can see the increasing speed of $K_n(X)$ is at most linear. Also, from (\ref{e70}), we have:

$$
\forall\,n\in \mathbb{N},\,\forall\, \epsilon\in \left(0, \frac{3}{4}\frac{1-r}{2r+1} \right), \hspace{0.3cm} \frac{K_n(X)-\epsilon}{K_{n-1}(X) - \epsilon} = \frac{K_n(X)-\epsilon}{K_n(X) + 2\epsilon} \frac{K_n(X)+2\epsilon}{K_{n-1}(X)-\epsilon} > r
$$
which implies that $\liminf_n \dfrac{K_n(X)}{K_{n-1}(X)} = 1$. Therefore, for any $M\in\mathbb{N}$, there exists $m\geq M$ such that $\dfrac{K_{m-1}(X)-\epsilon}{K_m(X) + 2\epsilon} > 1-\delta$. Fix such an $m\in\mathbb{N}$ and let $(g_j)_{j\in \mathbb{N}} \subseteq X^*_{\leq 1}$ be such that:

$$
K\big( m, (g_j)_{j\in \mathbb{N}} \big) < K_m(X) + \epsilon
$$
Fix two strictly increasing sequence of integers $(p_i)_{i\leq 2m}$, $(q_j)_{j\leq 2m}$ such that for any $1\leq i < m$, we have $p_{2i} < p_{2i+1} < q_{2i} < q_{2i+1} < p_{2i+2} < p_{2i+3}$ and $p_1<q_1<p_2$, $q_{2m-1} < p_{2m} < q_{2m}$. Define:

$$
\begin{aligned}
& \forall\,1\leq i \leq 2m, \hspace{0.5cm} \tilde{p}_i = 
\begin{cases}
q_i, \hspace{1.02cm} 2\nmid i\\
p_i, \hspace{1cm} 2\mid i
\end{cases} \\
& \forall\,1\leq i \leq 2m-2, \hspace{0.5cm} \tilde{q}_i = 
\begin{cases}
p_{i+2}, \hspace{0.63cm} 2\nmid i\\
q_i, \hspace{1cm} 2\mid i
\end{cases}
\end{aligned}
$$
Furthermore, by definition of $K_n\big( n, (g_j)_{j\in \mathbb{N}} \big)$, we can assume $(p_i)_{i\leq 2m}$ and $(q_j)_{j\leq 2m}$ satisfies:

\begin{equation}\label{e71}
\begin{aligned}
&\inf\left\{ \|z\|: z\in S\big( (p_i)_{i\leq 2m}; (g_j)_{j\in \mathbb{N}} \big) \right\} < K\big(m, (g_j)_{j\in \mathbb{N}} + \epsilon\\
&\inf\left\{ \|z\|: z\in S\big( (q_j)_{j\leq 2m}; (g_j)_{j\in \mathbb{N}} \big) \right\} < K\big(m, (g_j)_{j\in \mathbb{N}} + \epsilon\\
&\inf\left\{ \|z\|: z\in S\big( (\tilde{p}_i)_{i\leq 2m}; (g_j)_{j\in \mathbb{N}} \big) \right\} > K\big(m, (g_j)_{j\in \mathbb{N}} - \epsilon\\
&\inf\left\{ \|z\|: z\in S\big( (\tilde{q}_i)_{i\leq 2m}; (g_j)_{j\in \mathbb{N}} \big) \right\} > K\big(m-1, (g_j)_{j\in \mathbb{N}} - \epsilon\\
\end{aligned}
\end{equation}
We can then find $u\in S\big( (p_i)_{i\leq 2m}; (g_j)_{j\in \mathbb{N}} \big)$ and $v\in S\big( (q_i)_{i\leq 2m}; (g_j)_{j\in \mathbb{N}} \big)$ such that:

\begin{equation}\label{e72}
\max\big( \|u\|, \|v\| \big) < K\big(m, (g_j)_{j\in \mathbb{N}} \big) + \epsilon
\end{equation}
and for all $1\leq i \leq m$:

$$
\begin{cases}
\dfrac{3}{4} \leq (-1)^{i-1} g_k(u) \leq 1, \hspace{1cm}p_{2i-1} < k \leq p_{2i}\\\\
\dfrac{3}{4} \leq (-1)^{i-1} g_k(v) \leq 1, \hspace{1cm}q_{2i-1} < k \leq q_{2i}
\end{cases}
$$
Since for each $1\leq i \leq m$ we have $q_{2i-1} < p_{2i} < q_{2i}$, therefore for all $1\leq i \leq m$:

$$
\forall\,q_{2i-1} < k \leq p_{2i}, \hspace{0.3cm} \dfrac{3}{4} \leq (-1)^{i-1} g_k\left( \frac{u+v}{2} \right) \leq 1
\hspace{0.5cm} \Longrightarrow \hspace{0.5cm}
\frac{v+u}{2} \in S\big( (\tilde{p}_i)_{i\leq 2m}; (g_j)_{j\in \mathbb{N}} \big)
$$
which, according to (\ref{e71}), implies:

$$
\left\| \frac{v+u}{2} \right\| \geq K\big( m, (g_j)_{j\in \mathbb{N}} \big) - \epsilon
$$
Similarly, we have:

$$
\frac{v-u}{2}\in S\big( (\tilde{q}_i)_{i\leq 2m}; (g_j)_{j\in \mathbb{N}} \big)
\hspace{0.5cm} \Longrightarrow \hspace{0.5cm}
\left\| \frac{v-u}{2} \right\|\geq K\big(m, (g_j)_{j\in \mathbb{N}} \big) - \epsilon
$$
Next put:

$$
x = \frac{u}{K_m(X)+2\epsilon}, \hspace{1cm} y = \frac{v}{K_m(X)+2\epsilon}
$$
By (\ref{e72}) we have both $\|x\|<1$ and $\|y\|<1$. Meanwhile:

$$
\begin{aligned}
& \left\| \frac{x+y}{2} \right\| = \frac{1}{K_m(X)+2\epsilon} \left\| \frac{u+v}{2} \right\| \geq \frac{K_m(X)-\epsilon}{K_m(X) + 2\epsilon} > 1-\delta\\
& \left\| \frac{x-y}{2} \right\| = \frac{1}{K_m(X)+2\epsilon} \left\| \frac{u-v}{2} \right\| \geq \frac{K_{m-1}(X)-\epsilon}{K_m(X) + 2\epsilon} > 1-\delta\
\end{aligned}
$$
Therefore, $X$ is not uniformly non-square.
    
\end{proof}

\begin{theorem}

A Banach space $(X,\|\cdot\|_X)$ is super-reflexive if and only if $\tilde{CJ}(X)<2$.
    
\end{theorem}

\begin{proof}

When $X$ is super-reflexive, according to \textbf{Corollary \ref{Corollary 7.24}}, \textbf{Corollary \ref{Corollary 7.25}} and \textbf{Theorem \ref{Theorem 7.27}} we immediately have $\tilde{CJ}(X)<2$. Conversely, given $\tilde{CJ}(X)<2$, there exists a norm $\|\cdot\|$ that is equivalent to $\|\cdot\|_X$ such that $CJ(X, \|\cdot\|)<2$, which implies that $(X, \|\cdot\|)$ is uniformly non-square according to \textbf{Theorem \ref{Theorem 7.4}}. Next suppose $(Y, \|\cdot\|_Y)$ is a Banach space that can be crudely representable in $(X, \|\cdot\|_X)$. Then clearly $(Y, \|\cdot\|_Y)$ can be crudely representable in $(X, \|\cdot\|)$. Then there exists $\lambda\in (1, \infty)$ such that for any two $x, y \in Y_{\leq 1}$, there exists a linear mapping $T: \operatorname{Span}\big( \{x, y\} \big) \rightarrow X$ such that:

\begin{equation}\label{e73}
\begin{aligned}
& \frac{1}{\lambda} \|x\pm y\|_Y \leq \|T(x\pm y)\| \leq \lambda\|x\pm y\|_Y\\
& \frac{1}{\lambda} \|x\|_Y \leq \|Tx\| \leq \lambda\|x\|_Y\\
& \frac{1}{\lambda} \|x\|_Y \leq \|Ty\| \leq \lambda\|y\|_Y
\end{aligned}
\end{equation}
Let $\delta\in (0, 1)$ be given by the uniform non-squareness of $(X, \|\cdot\|)$ and together with (\ref{e73}) we have:

$$
\begin{aligned}
\left\| T\left( \frac{x}{\lambda} \right) \pm T\left( \frac{y}{\lambda} \right) \right\| \leq 2(1-\delta)
& \,\implies\, \left\| T\left( \frac{x}{\lambda} \right) \mp T\left( \frac{y}{\lambda} \right) \right\| > 2(1-\delta)\\
& \,\implies\, \|x\mp y\|_Y \geq \left\| T\left( \frac{x}{\lambda} \right) \mp T\left( \frac{y}{\lambda} \right) \right\| > 2(1-\delta)
\end{aligned}
$$
which implies that $(Y, \|\cdot\|_Y)$ is also uniformly non-square. Therefore, together with \textbf{Proposition \ref{Proposition 7.31}} and \textbf{Proposition \ref{Proposition 7.12}}, we can then conclude $X$ is super-reflexive.
    
\end{proof}

\newpage

\bibliographystyle{amsalpha}
\bibliography{Reference_List_2}

\end{document}